\documentclass[11pt,reqno]{amsart}

\usepackage[utf8]{inputenc}
\usepackage[margin=1.25in]{geometry}
\usepackage{hyperref}
\usepackage{appendix}
\usepackage{amsfonts}
\usepackage{amsthm}
\usepackage{amssymb}
\usepackage{stmaryrd} 
\usepackage{amsmath}
\usepackage{amsthm}
\usepackage[dvipsnames]{xcolor}
\usepackage{mathrsfs}
\usepackage{lipsum} % for filler text

\theoremstyle{plain}
\newtheorem{theorem}{Theorem}[section]
\newtheorem{corollary}[theorem]{Corollary}
\newtheorem{lemma}[theorem]{Lemma}
\newtheorem{Proposition}[theorem]{Proposition}

\newtheorem{Definition}[theorem]{Definition}

\newtheorem{fact}[theorem]{Fact}
\newtheorem{Reduction}[theorem]{Reduction}
\newtheorem{notation}[theorem]{Notation}

\theoremstyle{remark}

\newtheorem{remark}[theorem]{Remark}

\usepackage{biblatex} %Imports biblatex package
\numberwithin{equation}{section}
\title[Simplicity of singular value spectrum and two-point invertibility]{Simplicity of singular value spectrum of random matrices and two-point quantitative invertibility}

\author{Yi HAN}
\address{Department of Mathematics, Massachusetts Institute of Technology, Cambridge, MA
}
\email{hanyi16@mit.edu}

\begin{document}

\begin{abstract}
Let $A$ be an $n\times n$ random matrix with independent, identically distributed mean 0, variance 1 subgaussian entries. We prove that 
$$
\mathbb{P}(A\text{ has distinct singular values})\geq 1-e^{-cn}
$$ for some $c>0$, confirming a conjecture of Vu. This result is then generalized to singular values of rectangular random matrices with i.i.d. entries.

We also prove that for two fixed real numbers $\lambda_1,\lambda_2$ with a sufficient lower bound on $|\lambda_1-\lambda_2|$, we have a joint singular value small ball estimate for any $\epsilon>0$
$$
\mathbb{P}(\sigma_{min}(A-\lambda_1I_n)\leq\epsilon n^{-1/2},\sigma_{min}(A-\lambda_2I_n)\leq\epsilon n^{-1/2})\leq C\epsilon^2+e^{-cn}, 
$$ where $\sigma_{min}(A)$ is the minimal singular value of a square matrix $A$ and  $I_n$ is the identity matrix. For much smaller $|\lambda_1-\lambda_2|$ we derive a similar estimate with $C$ replaced by $C\sqrt{n}/|\lambda_1-\lambda_2|$. This generalizes the one-point estimate of Rudelson and Vershynin, which proves $\mathbb{P}(\sigma_{min}(A)\leq \epsilon n^{-1/2})\leq C\epsilon+e^{-cn}$. Analogous two-point bounds are proven when $A$ has i.i.d. real and complex parts, with $\epsilon^4$ in place of $\epsilon^2$ on the right hand side of the estimate and for any complex numbers $\lambda_1,\lambda_2$.
These two point estimates can be used to derive strong anticoncentration bounds for an arbitrary linear combination of two eigenvalues of $A$.

\end{abstract}

\maketitle

\section{Introduction}

Let $A$ be an $n\times n$ random matrix with i.i.d., mean 0, variance 1 subgaussian entries. The singularity probability of $A$, in the Bernoulli case where entries of $A$ are uniformly distributed over $\{-1,1\}$, has a long history dating back to Komlós \cite{komlos1967determinant}, Kahn, Komlós and Szemerédi \cite{kahn1995probability}, Tao and Vu \cite{tao2007singularity}, Bourgain, Vu, and Wood \cite{bourgain2010singularity}. Recent major breakthroughs have been made by Tikhomirov \cite{tikhomirov2020singularity} and Jain, Sah and Sawhney \cite{jain2021singularity}, settling longstanding conjectures.
In this paper, we consider the general case of a subgaussian entry distribution, for which Rudelson and Vershynin \cite{rudelson2008littlewood} proved in 2007 that the singularity probability of $A$ is exponentially small. They also obtained a quantitative version of the least singular value estimate. Let $\sigma_{min}(A)$ denote the least singular value of $A$, then by \cite{rudelson2008littlewood}, we can find $c,C>0$ depending only on subgaussian constant of entries so that for any $\epsilon>0$,
\begin{equation}\label{seintang}
    \mathbb{P}(\sigma_{min}(A)\leq\epsilon n^{-1/2})\leq C\epsilon+e^{-cn}.
\end{equation} Recently, Sah, Sahasrabudhe and Sawhney \cite{sah2025spielman} showed that the constant $C$ in \eqref{seintang} can be taken to be $C=1+o(1)$, yielding a singular value estimate that matches the Gaussian case up to a constant $1+o(1)$ and all the way down to $\epsilon\sim\exp(-\Omega(n))$ (see \cite{von1963design} for the first study in the Gaussian case, and the implication of this estimate on stability of algorithms). This proves a conjecture of Spielman and Teng \cite{spielmanteng}.
In a different direction, \eqref{seintang} can be generalized to an estimate for $\sigma_{min}(A-z I_n)$ for any $z\in\mathbb{C}$ where $I_n$ is the $n$-dimensional identity matrix. When $z$ satisfies $|\Im z|\geq\delta\sqrt{n}>0$, the estimate in \eqref{seintang} can be improved to $C_\delta\epsilon^2+e^{-cn}$ for a constant $C_\delta$ depending on $\delta$. This idea was first proposed by Rudelson and Vershynin \cite{rudelson2016no} to prove a novel eigenvector no-gaps delocalization property. Then Ge \cite{ge2017eigenvalue} adopted this idea to prove that with probability $1-o(1)$, all the eigenvalues of $A$ are distinct. See also Luh and O’Rourke \cite{luh2021eigenvectors} for a subsequent exponential bound on this probability.

Now that we have a comprehensive understanding of the singularity of $A$ and the estimates for $\sigma_{min}(A)$, there are still some open problems concerning the invertibility of $A$ that are not implied by all these results. In this paper we study two such problems, namely
\begin{enumerate}
    \item Is it true that with probability $1-o(1)$, all singular values of $A$ are distinct?
    \item Can we prove a joint least singular value estimate for $\sigma_{min}(A-\lambda_1 I_n)$ and $\sigma_{min}(A-\lambda_2 I_n)$ for relatively distant $\lambda_1,\lambda_2\in\mathbb{R}$ (i.e., with a lower bound on $|\lambda_1-\lambda_2|$ ) so that $$\mathbb{P}(\sigma_{min}(A-\lambda_1 I_n)\leq\epsilon,\sigma_{min}(A-\lambda_2 I_n)\leq\epsilon)\leq C\epsilon^2+e^{-cn}?$$ We can use this to show with high probability there are no two real eigenvalues $\lambda_1,\lambda_2$ of $A$ satisfying any fixed linear relation, say, $a\lambda_1+b\lambda_2=c$ for fixed $a,b,c$.
\end{enumerate}

The problem (1) was stated as an open problem in Vu \cite{vu2021recent}, Conjecture 8.5. In fact, he stated the problem for both a real symmetric random matrix and an i.i.d. random matrix, and claimed problem (1) is open for both models. We justify this Conjecture 8.5 for an i.i.d. matrix $A$ with an exponential bound on the probability of simple singular value spectrum.

The problem (2) is related to an intrinsic difficulty in least singular value estimates, and to our best knowledge such a joint estimate has never been proven before for any random matrix with discrete entry distribution. We will prove a version of the estimate in (2) that works for any $\lambda_1\neq\lambda_2$, with a factor before $\epsilon^2$ depending on the distance $|\lambda_1-\lambda_2|$quantitatively. Then we further generalize the estimate to complex-valued random matrices with genuinely complex shift $z_1,z_2$, and obtain a bound of order $\epsilon^4$.

\subsection{Simplicity of singular value spectrum}
Now we rigorously state the main results of this paper. The first result concerns the singular value spectrum.

\begin{theorem}\label{theorem1}
    Let $\xi$ be a mean 0, variance 1 subgaussian real random variable and let $A$ be an $n\times n$ random matrix with independent identically distributed entries of distribution $\xi$. Denote by $\sigma_1(A)\geq\cdots\geq\sigma_n(A)$ the singular values of $A$ arranged in decreasing order. Then for any $\epsilon>0$ we have
    $$\max_{k=1,\cdots,n-1}
\mathbb{P}(|\sigma_{k+1}(A)-\sigma_k(A)|\leq\epsilon n^{-1/2})\leq C\epsilon+e^{-cn},
    $$ for some constants $C,c>0$ depending only on $\xi$.

    In particular, we have $$\mathbb{P} (\text{All the singular values of $A$ are distinct})\geq 1-e^{-cn}.$$
\end{theorem}
This positively proves a version of \cite{vu2021recent}, Conjecture 8.5 for i.i.d. matrices, and provides a stronger statement that the probability of two equal singular values is exponentially small. This bound is best possible up to the constant $c>0$: when the distribution $\xi$ is discrete valued, then with an exponentially small probability $A$ may have two zero singular values.
\begin{remark}
In \cite{vu2021recent}, Conjecture 8.5 another conjecture was made concerning whether the singular values of a random \textit{symmetric} matrix are distinct with high probability, which is equivalent to whether two eigenvalues of the matrix sum up to 0. Our method, after significant generalization, can possibly be used to show that the linear span of two eigenvectors of the matrix has no rigid arithmetic structure (i.e., the essential LCD \eqref{essentiallcd}, \eqref{whatissmalllcd?} of the vector pair is large). While this result is nontrivial and of independent interest, this does not completely solve the original problem. In \cite{han2024small} the author attempted to solve this conjecture but the result only works when the entry distribution has a bounded density.
\end{remark}

After drafting the manuscript, we realized that the proof of Theorem \ref{theorem1} can be generalized to rectangular random matrices with i.i.d. entries. We have the following theorem:

\begin{theorem}\label{theoremrectan}
    Let $\xi$ be a random variable as in Theorem \ref{theorem1}. Fix $\mathbf{a}>1$ and let $N,n$ be positive integers satisfying $\mathbf{a}n\geq N\geq n$. Let $Z$ be an $N\times n$ rectangular random matrix with i.i.d. coordinates of distribution $\xi$. Denote by $\sigma_1(Z)\geq\cdots\geq \sigma_n(Z)$ the singular values of $Z$. Then all the claims in Theorem \ref{theorem1} remain true with $\sigma_k(Z)$ in place of $\sigma_k(A)$, and with constants $C,c>0$ depending only on $\xi$ and $\mathbf{a}$. In particular, $Z$ has distinct singular values with probability at least $1-\exp(-\Omega(n))$.
\end{theorem}

The assumption $N\geq n$ is clearly necessary, as otherwise $Z$ has $n-N$ zero singular values. Theorem \ref{theoremrectan} shows that this is essentially the only barrier to a simple singular value spectrum, since all non-zero singular values should be separated.
As we focus on square matrices in this paper, we will not alter the proof of Theorem \ref{theorem1} but present in Section \ref{rectangeneral} an outline of how the proof of Theorem \ref{theorem1} can be modified to yield Theorem \ref{theoremrectan}.

After the paper was submitted to the arXiv, the author learned about a concurrent work \cite{christoffersen2025gaps} which also considered simplicity of singular values. They considered slightly more general models and relaxed moment assumptions, but did not achieve an exponential bound in probability and the singular gap estimate is slightly weaker.

\subsection{Two-point invertibility: the real case}

The next result proves the joint two-point estimate for singular values. We first state a less technical version, requiring a separation of $|\lambda_1-\lambda_2|$:

\begin{theorem}\label{theorem2} Let $\xi$ be a mean 0, variance 1 subgaussian real random variable and let $A$ be an $n\times n$ random matrix with independent identically distributed entries of distribution $\xi$.
    Fix some $D>0$ and let $\lambda_1,\lambda_2\in[-4\sqrt{n},4\sqrt{n}]$ be two ($n$-dependent) real numbers satisfying $|\lambda_1-\lambda_2|\geq D\sqrt{n}$. Then we can find some $C_D>0$ depending on $D$ and the random variable $\xi$, and some $c>0$ depending only on $\xi$ such that, for any $\epsilon>0$,
    $$\mathbb{P}(\sigma_{min}(A-\lambda_1 I_n)\leq\epsilon,\sigma_{min}(A-\lambda_2 I_n)\leq\epsilon)\leq C_D\epsilon^2+e^{-cn}.$$
\end{theorem}

In a non-rigorous sense, Theorem \ref{theorem2} proves that the local spectrum of $A$ at two macroscopically separated locations $\lambda_1,\lambda_2$ are independent in a certain sense, and are independent all the way down to the minimal scale $\epsilon\sim \exp(-\Omega(n))$.
The assumption that $\lambda_1$ and $\lambda_2$ are globally separated, and the constant dependence on $D>0$ is restrictive in many applications. We are able to remove this restriction in the following theorem:

\begin{theorem}\label{theorem1.3}
    In the setting of Theorem \ref{theorem2}, for any $\lambda_1,\lambda_2\in[-4\sqrt{n},4\sqrt{n}]$ we can prove that, for any $\epsilon>0$,
       \begin{equation}\label{124wecanprove}\mathbb{P}(\sigma_{min}(A-\lambda_1 I_n)\leq\epsilon,\sigma_{min}(A-\lambda_2 I_n)\leq\epsilon)\leq \frac{C\sqrt{n}}{|\lambda_1-\lambda_2|}\epsilon^2+e^{-cn}\end{equation}
       where $C>0,c>0$ are two constants depending only on $\xi$.
\end{theorem}

The remarkable feature of Theorem \ref{theorem1.3} is that we can take $|\lambda_1-\lambda_2|$ arbitrarily small, all the way down in the microscopic range $(\exp(-\Omega(n)),n^{-1/2})$ and still maintain an effective estimate as long as we take $\epsilon>0$ to be simultaneously small relative to $|\lambda_1-\lambda_2|$. 

As an immediate corollary of Theorem \ref{theorem1.3}, we have:

\begin{corollary}\label{strongrepulsion}
     In the setting of Theorem \ref{theorem1.3}, let $a_n,b_n$ be two sequences of real numbers with $a_n=\exp(o(n))$. Then 
     \begin{equation}
         \mathbb{P}(\text{There exists two real eigenvalues $\lambda_1,\lambda_2$ of $A$ such that } \lambda_1+a_n\lambda_2=b_n)\leq e^{-cn}
     \end{equation} for some $c>0$ depending only on $\xi$.
\end{corollary}
For example, taking $a_n=-1$ and $b_n=D\sqrt{n}$, we get the probability that $A$ has two real eigenvalues differing by $D\sqrt{n}$. Taking $a_n=-3$ and $b_n=0$ we get the probability that $A$ has two real eigenvalues, one equals three times the other one, is exponentially small. Corollary \ref{strongrepulsion} rules out the possibility of any linear (and possibility many nonlinear) relations between the real eigenvalues of $A$.

In Theorem \ref{theorem1.3} we only considered real $\lambda_1,\lambda_2$, as the resulting estimate for complex $\lambda_1,\lambda_2$ seems very difficult to formulate. Ideally, the singular value bound would have the order $\epsilon^4$ rather than $\epsilon^2$, with a leading factor depending on $\Im \lambda_1,\Im\lambda_2$ and $\Im(\lambda_1-\lambda_2),\Re(\lambda_1-\lambda_2)$. We feel this estimate is too complicated to do here. But later in Section \ref{section1.4} we will formulate a clean estimate for complex i.i.d. matrices with factor $\epsilon^4$.

There is certainly one remaining question: does $A$ have many real eigenvalues? We prove that an i.i.d. random matrix with real entries asymptotically have many real eigenvalues, without assuming the matrix has matching fourth moments with real Gaussian matrix as in Tao and Vu \cite{tao2015random}. This is essentially a corollary of a recent universality result of Mohammed Osman \cite{osman2025bulk} and the computations in the Gaussian case in \cite{edelman1994many}, \cite{forrester2007eigenvalue}.

\begin{theorem}
\label{universalityrealroots} Let $A_n=(a_{ij})$ be an $n\times n$ random matrix with i.i.d. entries having distribution $\xi$, where $\xi$ is a random variable with mean 0, variance 1 and having all moments bounded. Let $N_\mathbb{R}(A_n)$ denote the number of real eigenvalues of $A_n$. Then there exists some $c'>0$ such that, for any $\epsilon>0$, with probability $1-O(n^{-c'})$, 
$$
N_\mathbb{R}(A_n)\geq \sqrt{\frac{2n}{\pi}}(1-\epsilon)+O(n^{1/2-c'}).
  $$
\end{theorem}

In \cite{vu2021recent}, Conjecture 11.3 the author asked if a square matrix with independent $\pm 1$ entries has $\Theta(\sqrt{n})$  real eigenvalues with very high probability, and we prove in Theorem \ref{universalityrealroots} that we have at least $C\sqrt{n}$ real eigenvalues in the asymptotic sense. Note that the lower bound of real eigenvalues is the harder part of the problem. The message we emphasize here is that there are with high probability many real eigenvalues, so the problem considered in Corollary \ref{strongrepulsion} is not vacuum for any such random variable $\xi$, and we believe the problem settled in Corollary \ref{strongrepulsion} is not accessible from previous techniques.

In \cite{eberhard2022characteristic}, Theorem 1.5 a related problem was studied on decoupling the correlation of eigenvector events for a random integral matrix over a finite field, which can be thought of as the finite field analogue of the topic in this work. However, the method of \cite{eberhard2022characteristic} does not seem to enable us to pass the estimate from finite field to $\mathbb{R}$.

\begin{remark}
In the mesoscopic range $n^{-1/2}\ll |\lambda_1-\lambda_2|\ll n^{1/2}$ one may expect that the independence persists in the sense that we can replace $C\sqrt{n}/|\lambda_1-\lambda_2|$ on the right hand side of \eqref{124wecanprove} by a fixed constant, at least for $\epsilon\geq n^{-\omega(1)}$. Our method may not be able to capture such mesoscopic independence. However, for much smaller $|\lambda_1-\lambda_2|$, no independence is expected and our method is the first to address two-location joint estimates in this range.
\end{remark}

\subsection{Two-point invertibility: the complex case}\label{section1.4}
When we consider a complex i.i.d. random matrix, similar but much more elegant results can be proven that generalize Theorem \ref{theorem1.3} where we can now shift the matrix in any complex direction $z_1,z_2\in\mathbb{C}$.

\begin{theorem}\label{complextwoballbounds}
    Let $G$ be an $n\times n$ random matrix with i.i.d. entries, each entry has the distribution $\xi+i\xi'$ where $\xi$ is a real-valued mean 0, variance 1 subgaussian random variable and $\xi'$ is an independent copy of $\xi$. Then for any $z_1,z_2\in\mathbb{C}$ with $|z_1|,|z_2|\leq8\sqrt{n}$ we have that for any $\epsilon>0$,
    \begin{equation}
        \mathbb{P}(\sigma_{min}(G-z_1I_n)\leq\epsilon n^{-1/2},\sigma_{min}(G-z_2I_n)\leq\epsilon n^{-1/2})\leq \frac{Cn}{|z_1-z_2|^2}\epsilon^4+e^{-cn}
    \end{equation} for two constants $C,c>0$ depending only on $\xi$.
\end{theorem}
The one location estimate, that is $\mathbb{P}(\sigma_{min}(G)\leq\epsilon n^{-1/2})\leq C\epsilon^2+e^{-cn}$, was obtained in Luh \cite{luh2018complex} who used it to prove that $\mathbb{P}(G\text{ has a real eigenvalue})\leq e^{-cn}$ for some $c\in(0,1)$. In this direction, our two-point estimate yields the following strong anti-concentration result:

\begin{corollary}\label{corollary2comp}
    Let $G$ be as in Theorem \ref{complextwoballbounds}. Then 
    $$
\mathbb{P}(\text{There exists two eigenvalues of $G$ whose sum is a real number})\leq e^{-cn} $$ for some constant $c\in(0,1)$ depending only on $\xi$.
\end{corollary}
Of course, Theorem \ref{complextwoballbounds} allows us to obtain many other anti-concentration phenomena that a linear or non-linear relation between two eigenvalues does not hold with high probability. Taking a union bound we can then show that any one of a (sub-exponentially large) family of such events should not happen with high possibility.
\begin{remark}
    We note that for complex i.i.d. matrices recently \cite{cipolloni2023universality}, Proposition 3.8 obtained  a similar version of approximate independence of small singular value events at different locations. More precisely, \cite{cipolloni2023universality} considered locations $z_i$ at the edge and where the separation of $z_i$ and $z_j$ can be only mesoscopic, while still retaining quantitative independence without our $n/|z_1-z_2|^2$ diverging term. The error in probability in \cite{cipolloni2023universality} is polynomial whereas we have an exponential bound. For bulk eigenvalues of complex i.i.d. matrix, similar estimates are recently derived in \cite{cipolloni2024maximum}, Lemma 10.4. The error rates in these works do not seem to be strong enough for getting a union bound over all possible eigenvalues of $G$, as for example \cite{cipolloni2024maximum} requires $\epsilon\geq(\log n)^{-C}$.
\end{remark}

\subsection{Main difficulties, essential steps and related works}
As one may expect, the main technical difficulty in proving Theorem \ref{theorem1} and \ref{theorem2} lies in the discrete entry distribution. Indeed, if the distribution of entries of $A$ has a bounded density, then we can prove both theorems in a considerably simplified way. For discrete entry distribution, one usually needs to eliminate arithmetic obstructions via applying inverse Littlewood-Offord type theorems \cite{rudelson2008littlewood}. This helps us to prove that the normal vector to a certain subspace has no simple arithmetic structure with very high probability. Implementing this step in our model is however very difficult for both Theorem \ref{theorem1} and \ref{theorem2}. For example, in Theorem \ref{theorem1} let $v$ be a singular vector of $A$, then we have $A^*Av=\lambda v$, but the entries of $A^*A$ are all correlated so that no Littlewood-Offord type theorem can be used. In Theorem \ref{theorem2} the problem is as follows. For any square matrix $A$ let $A[1]$ denote the matrix $A$ with its first row set to zero. Then to apply the framework of Rudelson and Vershynin \cite{rudelson2008littlewood} we need to prove that the linear span of the orthogonal complement of both $(A-\lambda_1 I_n)[1]$ and  $(A-\lambda_2 I_n)[1]$ has no vectors with simple arithmetic structure. Unfortunately this space is not annihilated by any linear equation of $A$ but only by a quadratic equation of $A$, and again the correlation of entries of $A^2$ prevents the use of classical Littlewood-Offord theorems.

To settle the problem raised in the last paragraph, we first take a linearization approach and transfer the problem to considering the following two linearized models \begin{equation}\label{ourblockmatrix}\begin{bmatrix}
    0&A\\A^T&0
\end{bmatrix}\quad \text{ or }\begin{bmatrix}
    0&A\\A&0
\end{bmatrix}\end{equation} (we will in fact consider the model when some rows and columns are removed). This linearization approach is frequently used in analyzing polynomials of random matrices, see \cite{haagerup2005new} and \cite{cook2022spectrum}. Then we have structured (block) random matrices where entries in each block are independent, but the two blocks are identical. This reminds us of a similar problem: invertibility of symmetric random matrices. We are in a similar situation but our matrix has fewer randomness than a symmetric matrix: the two principal diagonal blocks are zero.

We now briefly review the history of invertibility of a symmetric random matrix, that is, the probability that a random $\pm 1$ matrix is singular. The problem was first considered by Costello, Tao and Vu \cite{costello2006random} who proved almost sure invertibility. Then the invertibility probability was strengthened by  Nguyen \cite{nguyen2012inverse} to $n^{-\omega(1)},$ Vershynin \cite{vershynin2014invertibility} to $\exp(-n^c)$, Ferber and Jain \cite{ferber2019singularity} to $c=1/4$, and  Campos, Mattos, Morris, and Morrison \cite{campos2021singularity} to $c=1/2$. Finally, the exponential upper bound $c=1$ was proven by Campos, Jensen, Michelen and Sahasrabudhe
\cite{campos2025singularity}. An optimal singular value estimate was then derived in \cite{campos2024least}. Our proof of Theorem \ref{theorem1} takes many fruitful ideas from these works. As our block matrix (the left one in \eqref{ourblockmatrix}) has less randomness than the symmetric matrix in these cited works, our estimates are more sensitive to the loss of randomness if we cannot fully decouple dependence of two identical blocks. 

For Theorem \ref{theorem1}, we will use the conditioned inverse Littlewood-Offord theorem and a double counting technique which is similar to the one in \cite{campos2025singularity}. This idea was first introduced in the breakthrough work of \cite{tikhomirov2020singularity} and was coined “inversion of randomness”, where we use randomness both in the matrix and from the random choice of a vector from the net. The idea was then made into a different usage in \cite{campos2025singularity} to prove that the singularity probability of a random symmetric matrix is $\exp(-\Omega(n))$. In this paper we will take the latter viewpoint, but use the random net construction in a different way for the first matrix in \eqref{ourblockmatrix}.

In proving Theorem \ref{theorem2} and \ref{theorem1.3}, we face difficulties of another kind when dealing with the second matrix in \eqref{ourblockmatrix}. Here we no longer need the conditioned version of Littlewood-Offord theorems from \cite{campos2025singularity}, but for a vector $(v,w)\in\mathbb{R}^{n+n}$ in the kernel of the matrix, we need to stratify all possibilities of overlap between $v$ and $w$. Construction of nets depending on this overlap is also highly involved for lack of rotational symmetry. We will use a two-stage process to define the nets, the first stage uses a randomized procedure and the second stage constructs a net for vectors where either $v$ or $w$ has sub-exponentially large LCD. To set up the two-point estimate of least singular values, we will use the no-gaps delocalization property from \cite{rudelson2016no} for an initial reduction, and we feel this is a novel step not appearing in any previous works. A more careful description of the proof for both Theorem \ref{theorem1} and \ref{theorem1.3} are given in later chapters.

\subsection{Organization of the paper}
The proof of Theorem \ref{theorem1} is given in Section \ref{section222}.  The proof of  Theorem \ref{theoremrectan} is outlined in Section \ref{rectangeneral}. The proof of Theorem \ref{theorem2} and \ref{theorem1.3} are given in Section \ref{secgwe333}. The proof of Theorem \ref{complextwoballbounds} is given in Section \ref{secgwe444}.
Finally, Section \ref{prooflittlewoodofford} contains the proof of Corollary \ref{strongrepulsion}, \ref{corollary2comp}, Theorem \ref{universalityrealroots} and other auxiliary results.

\section{Distinct singular values}\label{section222}

Throughout the paper let $\xi$ be a mean $0$, variance 1 random variable. We say that $\xi$ is subgaussian if the following defined subgaussian moment
$$
\|\xi\|_{\psi_2}:=\sup_{p\geq 1}p^{-1/2}(\mathbb{E}|\xi|^p)^{1/p}
$$
is finite. For any $B>0$ denote by $\Gamma_B$ the collection of mean 0 variance 1 random variables with subgaussian moment bounded by $B$, and denote $\Gamma=\cup_{B\geq 0}\Gamma_B$.

Throughout the paper we also use $[n]$ to denote the subset $\{1,2,\cdots,n\}$ for each $n\in\mathbb{N}_+$, and let $[n_1,n_2]$ denote the subset $\{n_1,n_1+1,\cdots,n_2\}$ for each $n_1,n_2\in\mathbb{N}_+:n_1\leq n_2$.

For a fixed random variable $\xi$ we use the notation $\operatorname{Col}_n(\xi)$ to denote an $n$-dimensional random vector with i.i.d. coordinates having the distribution $\xi$.

\subsection{First reduction and proof outline}

We begin with a simple but essential observation.

\begin{fact}\label{fact2.1}The singular values of $A$ coincide with the non-negative eigenvalues of \begin{equation}\label{doublesides}
L_A:=\begin{pmatrix}
    0&A\\A^*&0
\end{pmatrix}.\end{equation}
     Moreover, if $v=(v^1,v^2)$ is an eigenvector of $L_A$ with nonzero eigenvalue, where we use $v^1$ (resp.$v^2$) to denote the first $n$ (resp. the last $n$ )coordinates of $v$, then $v^2$ is a singular vector of $A$.
\end{fact}

\begin{proof}
    Let $v=(v^1,v^2)$ be an eigenvector of $L_A$ associated to an eigenvalue $\lambda>0$, then $Av^2=\lambda v^1,A^*v^1=\lambda v^2$, so that $A^*Av^2=\lambda^2v^2$, and $v^2$ is a singular vector of $A$ with singular value $\lambda$. On the other hand, if $v^2\in\mathbb{R}^n$ satisfies $A^*Av^2=\lambda^2v^2$, then we set $v^1=Av^2/\lambda$ and get that $(v^1,v^2)$ is an eigenvector of $L_A$ with eigenvalue $|\lambda|$.
\end{proof}

Therefore, the proof of Theorem \ref{theorem1} is reduced to showing that the eigenvalues of $L_A$ are distinct with high probability. A crucial reduction from linear algebra, tailored for this purpose, was introduced in the work of Nguyen, Tao and Vu (\cite{nguyen2017random},Section 4, which is also used in \cite{campos2024least}, Fact 6.7) which states as follows:

\begin{fact}\label{fact2.2}
Consider an $n\times n$ real symmetric matrix $M$ with an eigenvalue $\lambda$ and associated eigenvector $u$. Fix $j\in[n]$ and let $M^{[j]}$ be the matrix $M$ with the $j$-th row and column removed. Let $\lambda'$ be an eigenvalue of $M^{[j]}$ associated to an 
eigenvector $v$. Then
$$
|\langle v,X^{(j)}\rangle|\leq|\lambda-\lambda'|/|u_j|
$$ where we take $X^{(j)}$ to be the $j$-th column of $M$ with its $j$-th entry removed.
\end{fact}

Therefore we only need to consider the matrix $L_A$ \eqref{doublesides} with a row and column removed. Without loss of generality we consider the case where its $j$-th row and column are removed for some $j\in[n]$, as the removal of the $j+n$-th row and column is completely analogous. As entries of $A$ are i.i.d., the removed rows and columns are independent from the rest of the matrix ${L}_A^{[j]}$, and ${L}_A^{[j]}$ has the same distribution as the following defined matrix $\mathcal{L}_A$:
\begin{notation} We define an $(2n-1)\times (2n-1)$ random matrix $\mathcal{L}_A$ via
    $$\mathcal{L}_A:=\begin{pmatrix}
        0&A_{/[1]}\\A_{/[1]}^T&0
    \end{pmatrix} $$ where $A_{/[1]}$ is an $(n-1)\times n$ matrix with i.i.d. entries distributed as $\xi$ and $A_{/[1]}^T$ is the transpose of $A[1]$.
\end{notation}

The eigenvectors of $\mathcal{L}_A$ satisfy a very important property: its two components associated to a nonzero eigenvalue have equal $\ell^2$ norm.

\begin{fact}\label{fact2.43}
Let $\lambda\neq 0$ be an eigenvalue of $\mathcal{L}_A$ with eigenvector $v=(v^1,v^2)$, where $v^1\in\mathbb{R}^{n-1}$ and $v^2\in\mathbb{R}^n$. Then $\|v^1\|_2=\|v^2\|_2$.
\end{fact}

\begin{proof}
    We have $A_{/[1]}v^2=\lambda v^1,A_{/[1]}^Tv^1=\lambda v_2$, so that 
    $ A_{/[1]}^TA_{/[1]}v^2=\lambda^2v_2$, and thus we have $\|A_{/[1]}v^2\|_2=\lambda \|v^2\|_2$. This finally implies $\|v^1\|_2=\|v^2\|_2$ whenever $\lambda\neq 0$.
\end{proof}
The condition $\lambda\neq 0$ is important. Indeed, $\mathcal{L}_A$ always has 0 as its eigenvalue with eigenvector $(0,v_0)$ and $ A_{/[1]}v_0=0$, because $A_{/[1]}$ has rank at most $n-1$.

\begin{notation}
    In the following we shall take the normalization that the eigenvector $v=(v^1,v^2)$ of $\mathcal{L}_A$ associated to a nonzero eigenvalue satisfies $\|v^1\|_2=\|v^2\|_2$=1.
\end{notation}

To state the main results on arithmetic structures of the eigenvector of $\mathcal{L}_A$, we follow the idea (from \cite{rudelson2008littlewood}) of partitioning the unit sphere into various subsets. We introduce the notion of compressible and incompressible vectors and the essential LCD as follows:

\begin{Definition}(Compressible) Fix two real numbers $\delta,\rho\in(0,1)$.
We recall the notion of $(\delta,\rho)$ compressibility of unit vectors from \cite{rudelson2008littlewood}. We will be working with unit vectors in both $\mathbb{S}^{n-1}$ and $\mathbb{S}^{n-2}$ using the same (in)compressibility notation. Specifically, a vector in $\mathbb{S}^{n-1}$ (or $\mathbb{S}^{n-2}$) is called $(\delta,\rho)$-compressible if it has distance no more than $\rho$ from a vector with support having size at most $\delta n$. We take $\delta<1$ so that when $n$ large we have $\delta n\leq n-1$.

For fixed $\delta,\rho\in(0,1)$, we use the notation $\operatorname{Comp}_n(\delta,\rho)$ to denote the collection of all compressible vectors in $\mathbb{S}^{n-1}$ and the notation $\operatorname{Comp}_{n-1}(\delta,\rho)$ to denote the collection of all compressible vectors in $\mathbb{S}^{n-2}$. The corresponding set of $(\delta,\rho)$- incompressible vectors is defined as
$$
\operatorname{Incomp}_n(\delta,\rho):=\mathbb{S}^{n-1}\setminus\operatorname{Comp}_n(\delta,\rho),\quad \operatorname{Incomp}_{n-1}(\delta,\rho):=\mathbb{S}^{n-2}\setminus\operatorname{Comp}_{n-1}(\delta,\rho).
$$
\end{Definition}

\begin{Definition}(Essential LCD)
   Fix two constants $\alpha>0,\gamma\in(0,1)$. For a vector $v\in\mathbb{R}^n$ we define its essential least common denominator as the following quantity $D_{\alpha,\gamma}(v)$
\begin{equation}\label{essentiallcd}
    D_{\alpha,\gamma}(v):=\inf\{\theta>0:\|\theta\cdot v\|_\mathbb{Z}\leq \min(\sqrt{\alpha n},\gamma\|\theta\cdot v\|_2)\},\end{equation}
    where we define, for any $w\in\mathbb{R}^n$,  $\|w\|_\mathbb{Z}:=\inf_{p\in\mathbb{Z}^n}\|w-p\|_2.$
\end{Definition}

We also need that incompressible vectors are spread (\cite{rudelson2008littlewood}, Lemma 3.4):
\begin{lemma}\label{incpmpspreads}
    Consider a vector $v\in\operatorname{Incomp}_n(\delta,\rho)$ (resp. $v\in\operatorname{Incomp}_{n-1}(\delta,\rho)$). Then 
    $$
(\rho/2)n^{-1/2}\leq |v_i|\leq \delta^{-1/2}n^{-1/2}
    $$
    for at least $\rho^2\delta n/2$ values of $i\in [n]$ (resp. $i\in[n-1]$).
\end{lemma}

The main quasi-random condition we will prove on the eigenvector of $\mathcal{L}_A$ is the following:

\begin{theorem}\label{theorem2.5s}
    Let $A$ be as in Theorem \ref{theorem1}. Then we can find constants $\alpha,\gamma,\delta,\rho\in(0,1)$ and some $c>0$ depending only on $\xi$ such that, with probability $1-e^{-cn}$, the following holds: 
    
    All the eigenvectors $(v^1,v^2)$ associated to a nonzero eigenvalue $\lambda\neq 0$ of $\mathcal{L}_A$ satisfy that both $v^1$ and $v^2$ are $(\delta,\rho)$-incompressible and \begin{equation}\label{twoconditionson}D_{\alpha,\gamma}(v^1)\geq e^{cn},\quad D_{\alpha,\gamma}(v^2)\geq e^{cn}.\end{equation} Moreover, all eigenvectors of $\mathcal{L}_A$ with zero eigenvalue has the form $(0,v^0)$ satisfying that $v^0$ is $(\delta,\rho)$-incompressible and $D_{\alpha,\gamma}(v^0)\geq e^{cn}$.
\end{theorem}

We also need a similar but weaker statement for the $2n\times 2n$ matrix $L_A$.

\begin{lemma}\label{originals}
    Let $A$ be as in Theorem \ref{theorem1}. Then we can find constants $\delta,\rho\in(0,1)$ and some $c>0$ depending on $\xi$ such that, with probability $1-e^{-cn}$, the following holds: 
    
    All the eigenvectors $(v^1,v^2)$ associated to a nonzero eigenvalue $\lambda\neq 0$ of $L_A$ satisfy that both $v^1$ and $v^2$ are $(\delta,\rho)$-incompressible. Moreover, $L_A$ has no zero eigenvalue.
\end{lemma}

Assuming the result in Theorem \ref{theorem2.5s}, we can complete the proof of Theorem \ref{theorem1}.

\begin{proof}[\proofname\ of Theorem \ref{theorem1}] Let $\sigma_1(A)\geq\sigma_2(A)\geq\cdots\geq\sigma_n(A)$ denote the singular values of $A$ in decreasing order. By Fact \ref{fact2.1} we have that, if $|\sigma_k(A)-\sigma_{k+1}(A)|\leq\epsilon$ for some $k\in[n-1]$ and $\epsilon>0$, then the $k$-th and the $k+1$-th largest eigenvalues of $\begin{pmatrix}
    0&A\\A^*&0
\end{pmatrix}$
have difference at most $\epsilon$. Then by Cauchy interlacing, we have $|\lambda_k(L_A)-\lambda_k(L_A^{[i]})|<\epsilon$ for each $i\in[2n]$ with $\lambda_k(\cdot)$ being the $k$-th largest eigenvalue of a symmetric matrix.

Let $\mathcal{A}$ denote the event on which (i) all principal $(2n-1)$-minors of $L_A$ satisfy the event stated in Theorem \ref{theorem2.5s} and (ii) $L_A$ satisfies the event stated in Lemma \ref{originals}, that is, all eigenvectors  $(v^1,v^2)$ of any $(2n-1)$-principal minor of $L_A$, with nonzero eigenvalues have both its components $v^1,v^2$ be $(\delta,\rho)$-incompressible and satisfy \eqref{twoconditionson}, and the same for the nonzero component of the eigenvector with eigenvalue zero. Also, all eigenvalues of $L_A$ have its two components being $(\delta,\rho)$-incompressible.  Then $\mathbb{P}(\mathcal{A})\geq 1-e^{-c'n}$ for some $c'>0$ by Theorem \ref{theorem2.5s} and Lemma \ref{originals}.

For each $i\in[2n]$ we let $\mathcal{E}_i$ denote the event on which
$$
|\langle w^{(i)},X_i\rangle|\leq\epsilon/c_{\rho,\delta},
  $$  where we take $X_i$ to be the $i$-th column of $L_A$ with coordinate $i$ removed, and $w^{(i)}$ is the unit eigenvector of $L_A^{[i]}$ (the principal minor of $L_A$ removing $i$-th row and column) with eigenvalue $\lambda_k(L_A^{[i]})$. The constant $c_{\rho,\delta}$ is defined in the next paragraph.

  First, we assume $\lambda_k(L_A^{[i]})\neq 0$. Under the event $\mathcal{A}$, both components of each $w^{(i)}$ are $(\delta,\rho)$-incompressible and thus has, by Lemma \ref{incpmpspreads}, at least $c_{\rho,\delta}n$ coordinates with absolute value at least $c_{\rho,\delta}n^{-1/2}$ where $c_{\rho,\delta}>0$ is a constant depending only on $\rho$ and $\delta$. Combining this fact with Fact \ref{fact2.2} and using that $|\lambda_k(L_A)-\lambda_k(L_A^{[i]})|<\epsilon$ , we can conclude that on the event $\mathcal{A}$, the event $\mathcal{E}_i$ should occur for at least $c_{\rho,\delta}n$ tuples of $i\in[2n]$.

Now we let $N$ be the number of indices $i\in[2n]$ such that $\mathcal{E}_i$ holds. Then
\begin{equation}\begin{aligned}
    \mathbb{P}(|\sigma_{k+1}(A)-\sigma_k(A)|\leq\epsilon n^{-1/2})&\leq \mathbb{P}(N\geq c_{\rho,\delta} n\text{ and }\mathcal{A})
    +O(e^{-c'n})\\&\leq \frac{2}{c_{\rho,\delta}}
    \sum_{i=1}^{2n}\frac{\mathbb{P}(\mathcal{E}_i\cap\mathcal{A})}{2n}+O(e^{-c'n}), 
\end{aligned}\end{equation} where we apply Markov's inequality in the last step.

Next we prove for some $C>0$, $\mathbb{P}(\mathcal{E}_i\cap\mathcal{A})\leq  C\epsilon$ for each $i\in[2n]$ and for all $\epsilon\gg e^{-cn}$. For this, note that $X_i$ is a $2n$-dimensional vector which is zero in its first (or last) $n$ dimensions and has i.i.d. entries $\xi$ in its last (or first) $n$ dimensions, and we have assumed that both components of $w^{(i)}$ have a large GCD as in \eqref{twoconditionson}. Then using independence of $X$ and $w^{(i)}$, and applying the Littlewood-Offord theorem stated in Theorem \ref{theorem6.8}, we can get

\begin{equation}\label{327fu327}
    \mathbb{P}(\mathcal{E}_i\cap \mathcal{A})\leq \max_{w\in\mathbb{R}^n:D_{\alpha,\gamma}(w)\geq e^{cn}}\mathbb{P}_X(|\langle w,Y\rangle|\leq\epsilon/c_{\rho,\delta})\leq C\epsilon,
\end{equation} where $Y\sim\operatorname{Col}_n(\xi)$ is the non-zero part of $X$ and $C>0$ depends only on $\xi,\rho,\delta$.

Finally there is a remaining case where $\lambda_k(L_A^{[i]})=0$. In this case, by construction of $\mathcal{A}$ the nonzero component of $w^{(i)}$ and of $X_i$ are in the same half of the coordinates (both in the first $n$ or the last $n$) so \eqref{327fu327} still applies.
This completes the proof of Theorem \ref{theorem1}.
\end{proof}

\subsection{A roadmap to the proof, and some preparations}\label{section2.2roadmap} The rest of this section is devoted to the proof of Theorem \ref{theorem2.5s}. On a high level, the proof takes the main idea of \cite{campos2025singularity} on invertibility of symmetric random matrices. As the proof takes nearly 15 pages and is extremely difficult, we outline its main ideas here.

The first step is to show that both $v^1,v^2$ are incompressible with high probability in Proposition \ref{proposition2.15}. Then difficulty starts to arise for incompressible vectors: the entries in $\mathcal{L}_A$ are dependent, and because of the transpose, the dependence is even more complicated. To simplify the picture, we remove some entries away from $\mathcal{L}_A$ and consider instead the following matrix with independent entries modulo symmetry $$    M=\begin{bmatrix}
0&\begin{bmatrix}0&H_2^T\\H_1&0\end{bmatrix}\\\begin{bmatrix}0&H_1^T\\H_2&0\end{bmatrix}&0
    \end{bmatrix}$$
where we assume that the entries of $v^1$ (resp. $v_2$) in the columns of $H_2$ (resp. $H_1$) have absolute value within a certain interval $[\kappa_0n^{-1/2},\kappa_1n^{-1/2}]$ thanks to incompressibility.

We will take the randomness from both the matrix $M$ and the randomly selected coordinate vectors, and randomly generate $v^1,v^2$ from certain integer lattices (boxes in Definition \ref{definitionbox}). The randomly generated vectors have large LCD with high possibility (Lemma \ref{lemma2.7uniformdistribution}), so that the invertibility of $H_2v_1'$ (resp. $H_1v_2'$, where we denote by $v_1',v_2'$ the restriction of $v_1,v_2$ to indices labeled by columns of $H_2,H_1$) could be proven via standard Littlewood-Offord inequalities. However, this leads to the loss of half the information as the randomness in $H_1$,$H_2$ have been used, so we instead use the conditioned Littlewood-Offord theorem (Theorem \ref{mainlittlewoodofford})which was invented in \cite{campos2024least}. The main idea is we can maintain an inverse Littlewood-Offord inequality while conditioning on the small singular values of $H_1$ and $H_2$. 

With a good lower bound for $\sigma_{min}(H_1),\sigma_{min}(H_2)$, we can use the randomness from the randomly generated vectors $v^1,v^2$ to restore anti-concentration in the rows corresponding to $H_1^T,H_2^T$ via Lemma \ref{smallballsing}. We will measure the arithmetic structure of the vectors with respect to an auxiliary threshold function $\tau_L$ \eqref{tauells}, and we define a $\epsilon$- net $\mathcal{N}_\epsilon$ with respect to $\tau_L$ in \eqref{netforn}. We will apply a double counting procedure to estimate the size of $\mathcal{N}_\epsilon$ in Lemma \ref{smalllemmas} and Theorem \ref{notthesecond}.

Finally, in Proposition \ref{prop2.40final} we will reduce the auxiliary net $\mathcal{N}_\epsilon$ back to a net $\Sigma_\epsilon'$ constructed from the LCD of $v^1$ and $v^2$. The proof is complete.

The main difference of our argument compared to \cite{campos2025singularity} is that our matrix $\mathcal{L}_A$ has a block structure and less randomness, but the essential ideas can still be used. While in \cite{campos2025singularity} there is only one block $H$ and $H^T$ to apply inverse Littlewood-Offord, here we have two blocks $H_1,H_2$ and their transpose, and we apply inverse Littlewood-Offord to each of them separately thanks to the fact that $H_1$ and $H_2$ are independent. In other words, our proof can be thought of a two-dimensional version of \cite{campos2025singularity}.

We will need a series of results that compare our random matrix $A$ to a lazier version. First we consider a symmetrized version of the random variable $\xi$ restricted to a bounded interval. For this let $\xi'$ be an independent copy of $\xi$ and we write
$$
\bar{\xi}=\xi-\xi'.
$$
Fix $B>0$, $\xi\in \Gamma_B$ and let $I_B:=(1,16B^2)$.  Denote by $$p:=\mathbb{P}(|\tilde{\xi}|\in I_B).$$ By \cite{campos2024least}, Lemma II.1, we have
\begin{lemma}\label{lemma2.11}
      $p\geq \frac{1}{2^7B^4}$.
\end{lemma}
Now we fix a parameter $\nu\in(0,1)$ and we define $\xi_\nu$ via
$$
\xi_{\nu}:=1\{|\tilde{\xi}|\in I_B\}\tilde{\xi}Z_\nu
$$ for $Z_\nu$ an independent Bernoulli variable with expectation $\nu$. We compute the characteristic function of $\xi_\nu$ via
$$
\phi_{\xi_\nu}(t):=\mathbb{E}e^{i2\pi t\xi_\nu}=1-\nu p+\nu p\mathbb{E}_{\bar{\xi}}\cos(2\pi t \bar{\xi}).
$$ For any $x\in\mathbb{R}$ we denote by $\|x\|_\mathbb{T}:=\operatorname{dist}(x,\mathbb{Z}).$ Then we have
\begin{equation}
\exp(-32\nu p\cdot \mathbb{E}_{\bar{\xi}}\|t\bar{\xi}\|_\mathbb{T}^2)\leq \phi_{\xi_\nu}(t)\leq \exp(-\nu p\cdot\mathbb{E}_{\bar{\xi}}\|t\bar{\xi}\|_\mathbb{T}^2
).   
\end{equation}

We also have 
$$
\phi_{\tilde{\xi}Z_\nu}(t)\leq \phi_{\xi_\nu}(t)\quad\forall t\in\mathbb{R}.
$$

We quote the following result on subgaussian concentration of operator norms, which can be derived from \cite{feldheim2010universality}.
We use $\|A\|_{op}$ for the operator norm of a matrix $A$.

\begin{fact}\label{operatornormfact}
    Let $\xi\in \Gamma$, and consider the random matrix $A$ as in Theorem \ref{theorem1}. Then
    $$
\mathbb{P}(\|A\|_{op}\geq 4\sqrt{n})\leq 2e^{-\Omega(n)}.
    $$
\end{fact}
Let $\mathcal{K}$ be the event $\mathcal{K}:=\{\|A\|_{op}\leq 4\sqrt{n},\quad \|\mathcal{L}_A\|_{op}\leq 4\sqrt{n}\}$. Then we define a measure $\mathbb{P}^\mathcal{K}$ via
\begin{equation}
    \mathbb{P}^\mathcal{K}(\mathcal{E})=\mathbb{P}(\mathcal{K}\cap\mathcal{E})
\end{equation}
for any measurable event $\mathcal{E}$. We include  $\mathcal{L}_A$ here in the event $\mathcal{K}$ because we will simultaneously consider $A$ and $\mathcal{L}_A$.

We recall the Littlewood-Offord theorem of Rudelson and Vershynin \cite{rudelson2008littlewood} as follows:
\begin{theorem}\label{theorem6.8}
Fix $B>0$, $\gamma,\alpha\in(0,1)$ and $n\in\mathbb{N}$, $\epsilon>0$. Consider some $v\in\mathbb{S}^{n-1}$ satisfying $D_{\alpha,\gamma}(v)>C\epsilon^{-1}$ and a vector $X\sim \operatorname{Col}_n(\xi)$, with $\xi\in\Gamma_B$. Then we have
$$
\mathbb{P}(|\langle X,v\rangle|\leq\epsilon)\lesssim \epsilon+e^{-c\alpha n}.
$$
    The constant $C>0$ only depends on $B$ and $\gamma$.
\end{theorem}

We also use frequently the tensorization lemma from \cite{rudelson2008littlewood}, Lemma 2.2.

\begin{lemma}\label{Tensorization}(Tensorization)
    Consider a family of independent non-negative random variables $\xi_1,\cdots,\xi_n$ and some $K,\epsilon_0\geq 0$. Assume that $\mathbb{P}(\xi_k\leq\epsilon)\leq K\epsilon$ for all $\epsilon\geq\epsilon_0$. Then $$\mathbb{P}(\sum_{k=1}^n\xi_k^2\leq \epsilon^2n)\leq (CK\epsilon)^n$$ for all $\epsilon>\epsilon_0$, where $C>0$ is a universal constant.
\end{lemma}

\subsection{Ruling out compressible vectors} We first show that with high probability $\mathcal{L}_A$ has no (non-zero) eigenvectors which are compressible in either its first or second part. 

To state our main result, we introduce a subset of balanced compressible vectors:
$$\begin{aligned}
\operatorname{B-comp}(\delta,\rho):=\{&v=(v^1,v^2)\in\mathbb{S}^{n-2}\times\mathbb{S}^{n-1}:v^1\in\operatorname{Comp}_{n-1}(\delta,\rho)\text{ or }v^2\in\operatorname{Comp}_{n}(\delta,\rho)\}.
\end{aligned}$$

We first show that these balanced compressible vectors do not lie in the kernel of $\mathcal{L}_A-\lambda I_{2n-1}$ with high probability, for any $\lambda\in[-4\sqrt{n},4\sqrt{n}]$. While it is now a standard approach to rule out compressible vectors for the eigenvector of a random matrix (see \cite{rudelson2008littlewood}), we aim for a much stronger statement: we need to show both $v^1$ and $v^2$ are incompressible.

\begin{Proposition}\label{proposition2.15}
    Fix $B>0$. Consider $\xi\in\Gamma_B$. Then we can find $\rho,\delta,c>0$ depending only on $B$ so that 
    $$\begin{aligned}
\mathbb{P}^\mathcal{K}&(\text{There exists } \lambda\in[-4\sqrt{n},4\sqrt{n}]\\&\text{ and }x\in\operatorname{B-Comp}(\delta,\rho)\text{ such that }(\mathcal{L}_A-\lambda I_{2n-1})x=0)\leq 2e^{-cn}.
    \end{aligned}$$
\end{Proposition}

\begin{proof} We first fix a value $\lambda\in[-4\sqrt{n},4\sqrt{n}]$.
    Suppose that there exists $v=(v^1,v^2)\in B-\operatorname{comp}(\delta,\rho)$ with $\|(\mathcal{L}_A-\lambda I_{2n-1})v\|_2\leq cn^{-4}$. We first assume that $v^2$ is $(\delta,\rho)$-compressible for some $\delta,\rho\in(0,1)$. From the equation that $v$ solves, we have
    $$
\|A_{/[1]}v^2-\lambda v^1\|\leq cn^{-4}, \quad \|A_{/[1]}^Tv^1-\lambda v^2\|\leq cn^{-4},
    $$ so that by Triangle inequality and the upper bound on $\|A\|_{op}$, we have
    $$\|A_{/[1]}^TA_{/[1]}v^2-\lambda^2 v^2\|\leq 8cn^{-3}.$$ As $v^2$ is compressible by assumption, $v^2$ is within Euclidean distance $\rho$ to a vector $w^2$ supported on less than $\delta n$ vertices, so that on $\mathcal{K}$, by triangle inequality
    $$
\|A_{/[1]}^TA_{/[1]}w^2-\lambda^2 w^2\|\leq 32\rho n. 
    $$
    Without loss of generality we may assume that $w^2$ is supported on its last $\delta n$ coordinates (this can be achieved via multiplying by permutation  matrices to relabel rows and column indices). Decomposing $A_{/[1]}$ into a block form we can write
    $$
A_{/[1]}=\begin{bmatrix}
    E&F\\G&H
\end{bmatrix}
    $$ where $E$ has size $(n-1-\delta n)\times(n-\delta n)$. Then we have
    $$
\|(E^TF+G^TH)w^2\|\leq 32\rho n.
    $$
    As $E,F,G,H$ are mutually independent, we condition on a realization of $G$ and $H$.

    Now $F$ is a rectangular matrix of size $(n-1-\delta n)\times \delta n$. By \cite{rudelson2008littlewood}, Proposition 2.5, choosing $\delta>0$ sufficiently small we can ensure that for some fixed constants $c_2>0,c_3>0$ that do not depend on $\delta$, we have
    $$
\mathbb{P}(\sigma_{min}(F)\leq c_2\sqrt{n})\leq\exp(-c_3n),
    $$ so that (we take $\rho<\frac{1}{2}$ so $\|w^2\|\geq\frac{1}{2}$) we have $\|Fw^2\|\geq 
    c_2\sqrt{n}/2$ with probability at least $1-\exp(-c_3n)$. 
    
    Next, since $E$ has independent coordinates, we can deduce a small ball probability bound for $E^T(Fw^2)$. For any $\delta<1/2$ we can apply \cite{rudelson2008littlewood}, Corollary 2.7 to deduce that there are constants $c_4,c_5>0$ depending on $\xi\in\Gamma_B$ such that for any unit vector $v\in\mathbb{S}^{n-\delta n-2}$, 
    $$\sup_{w\in\mathbb{R}^{n-\delta n}}\mathbb{P}(\|E^Tv-w\|\leq c_4\sqrt{n})\leq \exp(-c_5n).$$ Combining the last two facts, we see that    $$
\mathbb{P}(\|(E^TF+G^TH)w^2\|\leq c_2c_4 n/2)\leq \exp(-c_3n)+\exp(-c_5n).
    $$
Then we take a $\rho/4$-net $\mathcal{N}_\rho$ of $\mathbb{R}^{\delta n}\cap B_{\delta n}(0,2)$ and deduce that
  $$
\mathbb{P}^\mathcal{K}(\inf_{w^2\in\mathbb{S}^{\delta n}}\|(E^TF+G^TH)w^2\|\leq (c_2c_4+4\rho) n/2)\leq |\mathcal{N}_\rho| (\exp(-c_3n)+\exp(-c_5n)).
    $$
Next we choose $\rho>0$ sufficiently small, in addition to the above constraints, such that $4\rho\leq c_2c_4$. By standard volumetric computation we have $|\mathcal{N}_\rho|\leq (C_6(\rho))^{\delta n}$ where we may take the constant $C_6(\rho)=3\pi\frac{4}{\rho}+1$.

Then we can derive that, for some constant $c_\rho>0$ depending only on $\rho$ and $\xi$,
\begin{equation}\label{onelocationbounds}\begin{aligned}
\mathbb{P}^\mathcal{K}&(\text{There exists } x\in\operatorname{B-Comp}(\delta,\rho)\text{ such that }\|(\mathcal{L}_A-\lambda I_{2n-1})x\|\leq c_\rho n)\leq 2e^{-cn},
    \end{aligned}\end{equation}
to check this, note that the the left hand side of \eqref{onelocationbounds} is bounded by
\begin{equation}\label{concludeproof}
\binom{n}{\lceil \delta n\rceil} (C_6(\rho))^{\delta n}e^{-\min(c_3,c_5)n}\leq\exp(4e\delta\log(1/\delta)n+\log C_6(\rho)\delta n-\min(c_3,c_5)n), 
\end{equation} where the first combinatorial factor comes from choosing $\delta n$ out of $n$, and the second factor is the cardinality of a $\frac{1}{4}\rho$ net in $\mathbb{S}^{\delta n}$ coordinates. Finally we take $\delta>0$ sufficiently small so that the right hand side of \eqref{concludeproof} is $\exp(-\Omega(n))$. 

Having checked \eqref{onelocationbounds} for each $\lambda\in[-4\sqrt{n},4\sqrt{n}]$, we can take a $n^{-4}$-net for all $\lambda\in[-4\sqrt{n},4\sqrt{n}]$ and use the operator norm bound $\|A\|_{op}\leq 4\sqrt{n}$ to extend the probability uniformly over all $\lambda\in[-4\sqrt{n},4\sqrt{n}]$. This completes the proof that $v^2$ is incompressible. A similar argument shows that $v^1$ is incompressible as well.

\end{proof}

Having checked that any nonzero eigenvector $v=(v^1,v^2)$ of $\mathcal{L}_A$ must be incompressible in both parts $v^1,v^2$, we now introduce a convenient notion that encodes which coordinates of $v$ lie in a fixed interval of $[\kappa_0n^{-1/2},\kappa_1n^{-1/2}]$.

\begin{notation}
    
Fix two parameters $\kappa_0=\rho/3$ and $\kappa_1=\rho/6+\delta^{-1/2}$ determined in Proposition \ref{proposition2.15}, and two subsets $D_1\subset [n-1]$ and $D_2\subset [n,2n-1]$. We consider the set of unit vectors in $\mathbb{S}^{n-2}\times\mathbb{S}^{n-1}$ that are flat on both $D_1$ and $D_2$:
$$\begin{aligned}
\mathcal{I}(D_1,D_2)&=\{v=(v^1,v^2)\in\mathbb{S}^{n-2}\times\mathbb{S}^{n-1}:(\kappa_0+\kappa_0/2)n^{-1/2}\leq |(v^1)_i|\leq (\kappa_1-\kappa_0/2)n^{-1/2}\\& \text{ for all }i\in D_1,(\kappa_0+\kappa_0/2)n^{-1/2}\leq |(v^2)_i|\leq (\kappa_1-\kappa_0/2)n^{-1/2} \text{ for all }i\in D_2\},\end{aligned}
$$ where $(v^1)_i,i\in[n-1]$ denotes the $i$-th coordinate of $v^1$ and  $(v^2)_i,i\in[n,2n-1]$ denotes the $i$-th coordinate of $v^2$.   
And for a given $d\in\mathbb{N}_+$ we define 
\begin{equation}\label{definitionofid}
\mathcal{I}=\mathcal{I}_d:\cup_{D_1\subset[n-1],D_2\subset[n,2n-1]:|D_1|=|D_2|=d}\mathcal{I}(D_1,D_2),\end{equation}
 that is we consider all subsets $D_1,D_2$ both of which has cardinality $d$.
\end{notation}

We will be fixing the value of $d$ at a later stage. We always choose $d$ to be a multiple of $n$, but we will end up be choosing $d/n$ very small so as to apply Theorem \ref{mainlittlewoodofford}. The next result, which is a reformulation of Proposition \ref{proposition2.15}, shows that we can take any $d\leq\rho^2\delta n/2$:
\begin{lemma}\label{incompressibilityinvert}
    Let $\rho,\delta$ be fixed as in Proposition \ref{proposition2.15}. Then whenever $d\leq\rho^2\delta n/2$,
        $$\begin{aligned}
\mathbb{P}^\mathcal{K}&(\text{There exists } \lambda\in[-4\sqrt{n},4\sqrt{n}]\text{ and }x\in\mathcal{I} \text{ such that }(\mathcal{L}_A-\lambda I_{2n-1})x=0)\leq 2e^{-cn}.
    \end{aligned}$$
\end{lemma}

\begin{proof}
    This follows from Proposition \ref{proposition2.15} combined with Lemma \ref{incpmpspreads}, so that any B-incompressible vector should belong to $\mathcal{I}_d$.
\end{proof}

We can also outline here the proof of Lemma \ref{originals}.
\begin{proof}[\proofname\ of Lemma \ref{originals}] To check both $v^1$ and $v^2$ are incompressible with high probability, we simply run the same proof of Proposition \ref{proposition2.15} to the matrix $L_A$. The fact that $L_A$ has no zero eigenvalue with high probability follows from the fact that $\mathbb{P}(\det A=0)=e^{-\Omega(n)}$, which is the main result of \cite{rudelson2008littlewood}.
    
\end{proof}

\subsection{A Littlewood-Offord theorem and Fourier replacement}

To eliminate incompressible vectors with rigid arithmetic structure, we will use an upgraded version of inverse Littlewood-Offord theorem introduced in \cite{campos2025singularity}, Theorem 6.1 in the Bernoulli case and then in \cite{campos2024least}, Theorem VII.1 for the general subgaussian case. We use the notation $X\sim\Phi_\nu(d;\xi)$ for a $d$-dimensional vector with i.i.d. coordinates distributed as $\xi_\nu$. In the following the constant $B>0$ is chosen such that $\xi\in\Gamma_B$. Here for a $d$-dimensional vector $v$ we define its LCD via (we replace $\sqrt{\alpha n}$ in \eqref{essentiallcd} by $\sqrt{\alpha d}$ in the following)
\begin{equation}\label{essentiallcd2d}
    D_{\alpha,\gamma}(v):=\inf\{\theta>0:\|\theta\cdot v\|_\mathbb{Z}\leq \min(\sqrt{\alpha d},\gamma\|\theta\cdot v\|_2)\}.\end{equation}
\begin{theorem}[\cite{campos2024least}]\label{mainlittlewoodofford}
    Fix $n\in\mathbb{N}$, $0<c_0\leq 2^{-50}B^{-4}$, $d\leq c_0^2n$ and fix $\alpha,\gamma\in(0,1)$. Consider any $0\leq k\leq 2^{-32}B^{-4}\alpha d$ and $N\leq\exp(2^{-32}B^{-4}\alpha d)$. Assume $X\in \mathbb{R}^d$ satisfies $\|X\|_2\geq c_0 2^{-10}\gamma^{-1}n^{1/2}N$, and consider $H$ a $(n-d)\times 2d$ random matrix with i.i.d. rows distributed as $\Phi_\nu(2d;\xi)$ and $\nu=2^{-15}$. Then if $D_{\alpha,\gamma}(r_n\cdot X)\geq 2^{10}B^2$ then
    \begin{equation}
\mathbb{P}_H(\sigma_{2d-k+1}(H)\leq c_02^{-4}\sqrt{n}, \|H_1X\|_2\leq n,\|H_2X\|_2\leq n)\leq e^{-c_0nk/3}(\frac{R}{N})^{2n-2d},        
    \end{equation}
    where we take $H_1:=H_{[n-d]\times[d]}$, $H_2:=H_{[n-d]\times[d+1,2d]}$ and $r_n:=\frac{c_0}{32\sqrt{n}}$, $R:=2^{43}B^2c_0^{-3}$.
$\sigma_1(H)\geq\sigma_2(H)\geq\cdots\geq\sigma_{2d}(H)$ are singular values of $H$ arranged in decreasing order.
\end{theorem}

As we consider a set of vectors in $\mathcal{I}(D_1,D_2)$, we shall make the following notations
\begin{notation} We shall always remove the element $2n-1$ from $D_2$ whenever $2n-1\in D_2$.

Then we denote by $D_2-n+1:=\{x-n+1:x\in D_2\}$, and we set
\begin{equation}
D=D(D_1,D_2):=D_1\cup (D_2-n+1)\subseteq[n-1].
\end{equation} 
Then we have $d\leq |D|\leq 2d$ as we assume that $|D_1|=|D_2|=d$.
\end{notation}

We will treat differently the coordinates in $D$ and $[n]\setminus D$. To exploit independence, we take the idea of zero out matrices in \cite{campos2025singularity} by replacing $A_{/[1]}$ by maintaining its elements $a_{ij}$ with $i\in D,j\notin D$ and $i\notin D, j\in D$ but setting all other $a_{ij}$ to be identically zero. 

As we shall finally take the union bound over $D_1$ and $D_2$, and hence the union bound over $D$, we now assume without loss of generality that $D=\{1,2,\cdots,|D|\}$ where $|D|$ is the cardinality of $D$. The zeroed out matrix corresponding to $\mathcal{L}_A$ is then defined as follows: 
\begin{notation}\label{notation538} For a given constant $|D|\leq n-1$ we define the zeroed-out version of $\mathcal{L}_A$ as the following random matrix

$$    M=\begin{bmatrix}
0&\begin{bmatrix}0&H_2^T\\H_1&0\end{bmatrix}\\\begin{bmatrix}0&H_1^T\\H_2&0\end{bmatrix}&0
    \end{bmatrix}$$
    where $H_1$ is an $(n-1-|D|)\times |D|$ random matrix with i.i.d. entries of distribution $\tilde{\xi}Z_\nu$, and $H_2$ is an $(n-|D|)\times |D|$ random matrix with i.i.d. entries having distribution $\tilde{\xi}Z_\nu$. $H_1$ and $H_2$ are independent.
\end{notation}

Recall that for a random variable $Y\in\mathbb{R}^d$, we define its anti-concentration function via 
\begin{equation}\label{anticoncentrationlevy}
    \mathcal{L}(Y,\epsilon):=\sup_{w\in\mathbb{R}^d} \mathbb{P}(\|Y-w\|_2\leq\epsilon),
\end{equation}
and we introduce a notation relative to the matrix $\mathcal{L}_A$: for any $\epsilon>0$ and $v\in\mathbb{R}^{2n-1}$,

\begin{equation}
    \mathcal{L}_{\mathcal{L}_A,op}(v,\epsilon\sqrt{n}):=\sup_{w\in\mathbb{R}^{2n-1}}\mathbb{P}^\mathcal{K}(\|\mathcal{L}_Av-w\|_2\leq \epsilon\sqrt{n}).
\end{equation}

We will not be working with the matrix $\mathcal{L}_A$ for the moment, but rather work with $M$ first and construct a net from small ball probability of $M$. 

For a sufficiently large $L>0$ we define the threshold value relative to $M$ as follows:

\begin{Definition}
For a fixed constant $L>2$ and a vector $v=(v^1,v^2)\in\mathbb{S}^{n-2}\times\mathbb{S}^{n-1}$, the threshold $\tau_L(v)$ of $v$ relative to $M$ is defined via
\begin{equation}\label{tauells}
\tau_L(v):=\sup\{t\in[0,1]:\mathbb{P}(\|Mv\|_2\leq t\sqrt{n})\geq (4Lt)^{2n-1}\}.
\end{equation}
\end{Definition}

Next we show that the anticoncentration property proven for $M$ can be transferred to anticoncentration of  $\mathcal{L}_A$. The proof relies on Fourier replacement, see \cite{campos2025singularity}, Lemma 8.1.

\begin{lemma}\label{lemmareplacement}
    Assume that $|D|\leq 2^{-8}n$. Then for any $v\in\mathbb{S}^{n-2}\times\mathbb{S}^{n-1}$ and $t\geq\tau_L(v)$, 
    $$
\mathcal{L}(\mathcal{L}_Av,t\sqrt{n})\leq (50Lt)^{2n-1}.
    $$
\end{lemma}
The proof of this lemma is deferred to Section \ref{prooflittlewoodofford}.

\subsection{Random generation of a net}

Next we will use a randomized algorithm to generate nets for the matrix $M$. The method is inspired from \cite{campos2025singularity}. The idea of inversion of randomness used here, and in particular the notion of $(N,\kappa,d)$ boxes, originated from \cite{tikhomirov2020singularity}.

As in our case the vector $v=(v^1,v^2)$ satisfy that both components are incompressible, we will more precisely need a $(N,\kappa,D_1,D_2)$ box depending on the two subsets $D_1,D_2$.

\begin{Definition}\label{definitionbox} Fix a constant $\kappa\geq 2$ and two subsets $D_1\subset[n-1]$ and $D_2\subset[n,2n-1]$.
    An $(N,\kappa,D_1,D_2)$ box is defined as a product set of the form $$\mathcal{B}=B_1\times B_2\times\cdots\times B_{2n-1}\subset\mathbb{Z}^{2n-1},$$ where we have
    \begin{enumerate}
    \item $|B_i|\geq N$ for each $i$, \item $B_i=[-\kappa N,-N]\cup[N,\kappa N]$ for $i\in D_1\cup D_2$, and \item $|\mathcal{B}|\leq(\kappa N)^{2n-1}$.\end{enumerate} 
\end{Definition}
 Here we make no assumption as to whether $D_1$ and $D_2-(n-1)$ coincide or even overlap.

Then we show that a random vector sampled from $\mathcal{B}$ has a large essential LCD, with super-exponential probability. We will use the following lemma from \cite{campos2025singularity}, Lemma 7.4.
\begin{lemma}\label{lemma2.7uniformdistribution}
    Fix $\alpha\in(0,1),K\geq 1,\kappa\geq 2$. Consider $n\geq d\geq K^2\alpha$, $N\geq 2$ with $KN<2^d$. Then consider $\mathcal{B}=([-\kappa N,-N]\cup[N,\kappa N])^d$ and consider $X$ chosen uniformly over $\mathcal{B}$. Then
 \begin{equation}
     \mathbb{P}_X(D_\alpha(r_n\cdot X)\leq K)\leq (2^{20}\alpha)^{d/4},
 \end{equation} with the choice $r_n:=c_02^{-4}n^{-1/2}$. Here  we write $D_\alpha(X):=D_{\alpha,\frac{1}{2}}(X).$
\end{lemma}

To estimate randomness from the uniformly chosen base, we need the following lemma from \cite{campos2025singularity}, Lemma 7.5. 

\begin{lemma}[\cite{campos2025singularity}, Lemma 7.5]\label{smallballsing}
    Fix $N,n,d,k\in\mathbb{N}$ satisfying $n-d\geq 2d>2k$. Let $H$ be a $2d\times (n-d)$ matrix with $\sigma_{2d-k}(H)\geq c_0\sqrt{n}/16$. Let $B_1,\cdots,B_{n-d}\subset\mathbb{Z}$ with $|B_i|\geq N$ each. With $X$ uniformly chosen from $\mathcal{B}:=B_1\times\cdots B_{n-d}$ we have
    $$
\mathbb{P}_X(\|HX\|_2\leq n)\leq (\frac{Cn}{dc_0N})^{2d-k}
    $$ for some universal constant $C>0$. 
\end{lemma}

We now state an estimate for small ball probability bounds using randomness from both $X$ and $M$. This result will be useful in computing cardinality of the net.

\begin{lemma}\label{smalllemmas}
    For any $L\geq 2$ and $0<c_0<2^{-50}B^{-4}$, consider $n>L^{64/c_0^2}$ with $\frac{1}{4}c_0^2n\leq d\leq c_0^2n$. Given $N\geq 2$ that satisfies $\log N\leq c_0L^{-8n/d}d$ and any $\kappa\geq 2$. Fix two subsets $D_1\in[n-1],D_2\in[n,2n-1]$ of cardinality $d$ each such that $D=D(D_1,D_2)=[1,|D|]$. Consider $\mathcal{B}$ a $(N,\kappa,D_1,D_2)$-box and let $X$ be uniformly generated from $\mathcal{B}$. Then 
    $$
\mathbb{P}_X\left(\mathbb{P}_M(\|MX\|_2\leq n)\geq \left(\frac{L}{N}\right)^{2n-1}
\right)\leq \left(\frac{R}{L}\right)^{4n-2},
    $$
    with the constant $R:=Cc_0^{-3}$ and where $C>0$ is a universal constant.
\end{lemma}

Now we prepare for the proof of Lemma \ref{smalllemmas}. We let $H_3$ be an $(n-1-|D|)\times |D|$ random matrix which is an i.i.d. copy of $H_1$, and let $H_4$ be an $(n-|D|)\times |D|$ random matrix which is an i.i.d. copy of $H_2$. Then we define an $(n-1-|D|)\times 2d$ matrix $H[1]$ and an $(n-|D|)\times 2d$ matrix $H[2]$ via
$$
H[1]:=\begin{bmatrix}
    H_1&H_3
\end{bmatrix},\quad H[2]:=\begin{bmatrix}
    H_2&H_4
\end{bmatrix}.
$$

For any vector $X\in\mathbb{R}^{2n-1}$ we consider four events $\mathcal{A}_1=\mathcal{A}_1(X)$, $\mathcal{A}_2=\mathcal{A}_2(X)$:

$$
\mathcal{A}_1:=\{H[1]:\|H_1X_{[n,2n-2-|D|]}\|_2\leq n,\quad \|H_3X_{[n,2n-2-|D|]}\|_2\leq n \},
$$
$$
\mathcal{A}_2:=\{H[2]:\|H_2X_{[|D|]}\|_2\leq n, \quad \|H_4X_{[|D|]}\|_2\leq n,  \}
$$
and we let $\mathcal{A}_3=\mathcal{A}_3(X),\mathcal{A}_4=\mathcal{A}_4(X)$ be the events
$$\mathcal{A}_3:=\{H[1]:\|(H[1])^TX_{[|D|+1,n-1]}\|_2\leq 2n\},
$$
$$\mathcal{A}_4:=\{H[2]:\|(H[2])^TX_{[2n-1-|D|,2n-1]}\|_2\leq 2n\}.
$$
Then our second moment computation can be reformulated as
\begin{fact}\label{variantof694}
    Fix $X\in\mathbb{R}^{2n-1}$, and the events $\mathcal{A}_i=\mathcal{A}_i(X),i=1,2,3,4$ as above. Then 
    $$
(\mathbb{P}_M(\|MX\|_2\leq n))^2\leq\mathbb{P}_{H[1],H[2]}(\mathcal{A}_1\cap \mathcal{A}_2\cap\mathcal{A}_3\cap\mathcal{A}_4).
    $$
\end{fact}

\begin{proof}
    Consider $M'$ an independent copy of $M$. Then 
    $$
(\mathbb{P}_M(\|MX\|_2\leq n))^2=\sum_{M,M'}\mathbb{P}(M')\mathbb{P}(M)\mathbf{1}(\|MX\|_2,|M'X\|_2\leq n),
    $$
    which is upper bounded by
    $$\begin{aligned}
&\sum_{H[1],H[2]}\mathbb{P}(H[1])\mathbb{P}(H[2])\\&\mathbf{1}(\|H_2X_{[|D|]}\|_2\leq n,\|H_4X_{[|D|]}\|_2\leq n,\|H_1X_{[n,2n-2-|D|]}\|_2\leq n,\|H_3X_{[n,2n-2-|D|]}\|_2\leq n,\\&\|(H[1])^TX_{[|D|+1,n-1]}\|_2\leq 2n,\|(H[2])^TX_{[2n-1-|D|,2n-1]}\|_2\leq 2n),
    \end{aligned}$$
    and the last expression equals $\mathbb{P}_{H[1],H[2]}(\mathcal{A}_1\cap\mathcal{A}_2\cap\mathcal{A}_3\cap\mathcal{A}_4).$
\end{proof}

Now we are ready to prove Lemma \ref{smalllemmas}.

\begin{proof}[\proofname\ of Lemma \ref{smalllemmas}]
    Denote by 
    $$
\mathcal{E}:=\{X\in\mathcal{B}:\mathbb{P}_M(\|MX\|_2\leq n)\geq (L/N)^{2n-1}\}
    $$
    and set $\alpha=2^{13}L^{-8n/d}$. We define
    \begin{equation}
        T=T(\mathcal{B}):=\{X\in\mathcal{B}:D_{\alpha/2}(c_0X_{[|D|]}/(32\sqrt{n})),  D_{\alpha/2}(c_0X_{[n,2n-2-|D|]}/(32\sqrt{n}))\geq 2^{10}B^2  \}.
    \end{equation} 
    Here we note a technical point resulting from dimensional dependence in \eqref{essentiallcd2d}: using $|D|\leq 2|D_1|$, for any $X$ we have the inequality $D_{\alpha}(X_{D_1})\leq D_{\alpha/2}(X_{[|D|]})$.
    Then for $X$ chosen over $\mathcal{B}$ we have, by Lemma \ref{lemma2.7uniformdistribution} applied to the coordinates in $D_1$ and $D_2$ separately and by independence of coordinates of $X$, that
    \begin{equation}\label{probnotint}
        \mathbb{P}_X(X\notin T)\leq(2^{20}\alpha)^{d/2}\leq  \left(\frac{2}{L}\right)^{4n-2} 
    \end{equation} by our choice of $L,\alpha$ and $d$.

Defining 
$$f(X):=\mathbb{P}_M(\|MX\|_2\leq n)\mathbf{1}(X\in T),$$ then we obtain, through applying Markov's inequality,
\begin{equation}
    \mathbb{P}_X(\mathcal{E})\leq\mathbb{P}_X(f(X)\geq (L/N)^{2n-1})+(2/L)^{4n-2}\leq (N/L)^{2n-1}\mathbb{E}_X f(X)^2+(2/L)^{4n-2}. 
\end{equation}

We need a robust notation of rank for our random matrix, which we define here as
$$
\mathcal{E}_k[1]:=\{H[1]:\sigma_{2|D|-k}(H[1])\geq c_0\sqrt{n}/16,\text{ and }\sigma_{2|D|-k+1}(H[1])<c_0\sqrt{n}/16\},
$$$$
\mathcal{E}_k[2]:=\{H[2]:\sigma_{2|D|-k}(H[2])\geq c_0\sqrt{n}/16,\text{ and }\sigma_{2|D|-k+1}(H[2])<c_0\sqrt{n}/16\},
$$ so that for each $i=1,2$ precisely one of the events $\mathcal{E}_0[i],\cdots,\mathcal{E}_{2|D|}[i]$ holds. We set $\alpha'=2^{-32}B^{-4}\alpha$. Then applying Theorem \ref{mainlittlewoodofford}, we have that for any $X\in\mathcal{B}$,
\begin{equation}\label{expectationofxs}
\begin{aligned}
    f(X)^2&\leq \sum_{k_1,k_2=0}^{\alpha'|D| }\mathbb{P}_{H[1]}(\mathcal{A}_3\mid\mathcal{A}_1\cap\mathcal{E}_{k_1})\mathbb{P}_{H[2]}(\mathcal{A}_4\mid\mathcal{A}_2\cap\mathcal{E}_{k_2})\\& \exp(-c_0n(k_1+k_2)/4)(\frac{R}{N})^{4n-2-4|D|}+\mathbb{P}(\cup_{k\geq\alpha'|D|}\mathcal{E}_k[1])+\mathbb{P}(\cup_{k\geq\alpha'|D|}\mathcal{E}_k[2]),
\end{aligned}\end{equation}
where we use the fact that $\mathcal{A}_1,\mathcal{A}_3$ and $\mathcal{A}_2,\mathcal{A}_4$ are mutually independent and we apply Theorem \ref{mainlittlewoodofford} twice to each probability $\mathbb{P}(\mathcal{A}_1\cap\mathcal{E}_{k_1})$ and $\mathbb{P}(\mathcal{A}_2\cap\mathcal{E}_{k_2})$.

For $k\geq\alpha'd$, the possibility that $\sigma_{2|D|-k}(H[1])$ is small is at most $\exp(-\omega(kn))$, which can be verified via Hanson-Wright inequality or via using \cite{campos2025singularity}, Lemma VII.2,3 Fact VII.4,5. More precisely we can deduce that $\mathbb{P}(\cup_{k\geq\alpha'|D|}\mathcal{E}_k[1])\leq \exp(-c_0\alpha'dn/4)$ from these bounds.

Next we estimate the probability of $\mathbb{P}_{H[1]}(\mathcal{A}_3\mid\mathcal{A}_1\cap\mathcal{E}_{k_1})$ for each $1\leq k\leq\alpha'd$.
Consider a quantity $g_{k}(X)[1]:=\mathbb{P}(\mathcal{A}_3\mid\mathcal{A}_1\cap\mathcal{E}_k)$ (and also $g_{k}(X)[2]:=\mathbb{P}(\mathcal{A}_4\mid\mathcal{A}_2\cap\mathcal{E}_k)$). Then we have that
$$
\mathbb{E}_X[g_k(X)[1]]=\mathbb{E}_{X_{[2n-1]\setminus[|D|+1,n-1]}}\mathbb{E}_{H[1]}[\mathbb{E}_{X_{[|D|+1,n-1]}}\mathbf{1}[\mathcal{A}_2]|\mid\mathcal{A}_1\cap\mathcal{E}_k]. 
$$For this $H\in\mathcal{A}_1\cap\mathcal{E}_k$ we have $\sigma_{2d-k}(H)\geq c_0\sqrt{n}/16$, so that by Lemma \ref{smallballsing},
$$
\mathbb{E}_{X_{|D|+1,n-1}}\mathbf{1}[\mathcal{A}_3]=\mathbb{P}_{X_{[|D|+1,n-1]}}(\|(H[1])^TX_{[|D|+1,n-1]}\|_2\leq n)\leq (\frac{C'n}{c_0dn})^{2|D|-k}\leq (\frac{4C'}{c_0^3N})^{2|D|-k}
$$ where $C'>0$ is an absolute constant.

Therefore given any $0\leq k\leq\alpha'|D|$ with $R:=\max\{8C'c_0^{-3},2R_0\}$ there holds
\begin{equation}
    \mathbb{E}_X[g_k(X)[1]]\leq (\frac{R}{2N})^{2|D|-k}.
\end{equation}
Taking expectation with respect to $X$ on both sides of \eqref{expectationofxs}, we have
$$
\mathbb{E}_Xf(X)^2\leq \left(\frac{R}{2N}\right)^{4n-2}\sum_{k_1,k_2=0}^{\alpha'|D|}\left(\frac{2N}{R}\right)^{k_1+k_2}\exp(-c_0n(k_1+k_2)/4)+\exp(-c_0\alpha'|D|n/4).
$$ Under the condition $N\leq \exp(c_0n/4),N\leq \exp(c_0L^{-8n/d}d)$ we conclude
\begin{equation}
    \mathbb{E}_X f(X)^2\leq 2\left(\frac{R}{2N}\right)^{4n-2}
\end{equation} for a universal constant $R>0$,
which concludes the proof combined with estimate \eqref{probnotint}.
\end{proof}

Now we can state and prove the main theorem of this chapter on the cardinality of nets. Before that, we need to estimate the cardinality of boxes $\mathcal{B}$ needed to form a covering. 

Here we define the following subset which is slight modification of $\mathcal{I}(D_1,D_2):$
$$
\mathcal{I}'(D_1,D_2):=\{v\in\mathbb{R}^{2n-1}: \kappa_0 n^{-1/2}\leq|v_i|\leq \kappa_1 n^{-1/2} \text{ for all }i\in D_1\cup D_2
\},
$$
where $D_1\subset[n-1]$ and $D_2\subset[n,2n-1]$ with $|D_1|=|D_2|=d$, and that $0<\kappa_0<\sqrt{1/2}<\kappa_1$ are absolute constants.

When we have made the normalization $D:=D(D_1,D_2)=[|D|]$, we shall define 
$$
\Lambda_\epsilon:=(B_{n-2}(0,2)\times B_{n-1}(0,2))\cap (4\epsilon n^{-1/2}\cdot\mathbb{Z}^{2n-1})\cap \mathcal{I}'(D_1,D_2),
$$ where we keep implicit the dependence on $D_1,D_2$ in the notation $\Lambda_\epsilon$.

We show that $\Lambda_\epsilon$ can be covered by exponentially many boxes $\mathcal{B}$:

\begin{lemma}\label{specified850} For any fixed $D_1,D_2$ as above with $|D_1|=|D_2|=d$, any $\epsilon\in[0,1]$, and for any $\kappa\geq\max\{\kappa_1/\kappa_0,2^8\kappa_0^{-4}\}$, we can find a family $\mathcal{F}$ of $(N,\kappa,D_1,D_2)$ boxes with cardinality $|\mathcal{F}|\leq\kappa^{2n}$ satisfying
\begin{equation}
    \Lambda_\epsilon\subseteq \cup_{\mathcal{B}\in\mathcal{F}}(4\epsilon n^{-1/2})\cdot\mathcal{B}
\end{equation}
    and where we set $N=\kappa_0/(4\epsilon).$
\end{lemma}

The proof of this lemma is mostly standard and deferred to Section \ref{appendix4a}.

\subsection{A net defined by small ball probability and its cardinality}

We define a new version of net $\mathcal{N}_\epsilon$ depending on small ball probabilities measured by $M$.

For fixed subsets $D_1,D_2$ as above and $\epsilon>0$, we define a net 
\begin{equation}\label{netforn}
\mathcal{N}_\epsilon:=\{v\in\Lambda_\epsilon:(L\epsilon)^{2n-1}\leq\mathbb{P}(\|Mv\|_2\leq4\epsilon\sqrt{n}),\text{and }\mathcal{L}_{\mathcal{L}_A,op}(v,\epsilon\sqrt{n})\leq (2^8L\epsilon)^{2n-1}\}.
\end{equation}

The computations in previous section essentially imply an upper bound for $|\mathcal{N}_\epsilon|$. In the proof we only use the first condition in the definition of $\mathcal{N}_\epsilon$, but not the second condition.
\begin{theorem}\label{notthesecond}
    For any $L\geq 2$, $0<c_0\leq 2^{-50}B^{-4}$, $n\geq L^{64/c_0^2}$ and take $d\in[c_0^2n/4,c_0^2n]$. Take $\epsilon>0$ with $\log\epsilon^{-1}\leq \frac{1}{4}c_0^3nL^{-32/c_0^2}$. Then we have, for a universal constant $C>0$, the upper bound
$$
|\mathcal{N}_\epsilon|\leq \left(\frac{C}{c_0^6 L^2\epsilon}\right)^{2n-1},
$$

\end{theorem}

\begin{proof}
Since $\Lambda_\epsilon$ admits a covering by $(N,\kappa,D_1,D_2)$ boxes $\mathcal{B}$, we have
$$
|\mathcal{N}_\epsilon|\leq |\mathcal{F}|\cdot\max_{\mathcal{B}\in\mathcal{F}}|(4\epsilon n^{-1/2}\cdot \mathcal{B})\cap\mathcal{N}_\epsilon|,
$$ where $\mathcal{F}$ is specified in Lemma \ref{specified850}. Then we have by definition and by Lemma \ref{smalllemmas},
$$
|(4\epsilon n^{-1/2}\cdot\mathcal{B})\cap\mathcal{N}_\epsilon|\leq |\{X\in\mathcal{B}:\mathbb{P}_M(\|MX\|_2\leq n)\geq (\frac{L}{N})^{2n-1}
\}|\leq \left(\frac{R}{L}\right)^{4n-2}|\mathcal{B}|,
$$ where the application of Lemma \ref{smalllemmas} is fulfilled as we have $\log 1/\epsilon\leq\frac{1}{4}c_0^3 nL^{-32/c_0^2}$ so that 
$
\log N=\log\kappa_0/(4\epsilon)\leq \frac{1}{4}c_0^3nL^{-32/c_0^2}\leq c_0 L^{-8n/d}d.
$  Therefore we conclude 
$$
|\mathcal{N}_\epsilon|\leq\kappa^{2n}\left(\frac{R}{L}\right)^{4n-2}|\mathcal{B}|\leq \left(\frac{C}{c_0^6L^2\epsilon}\right)^{2n-1},
$$ for some universal constant $C>0$.
\end{proof}

Next we show that $\mathcal{N}_\epsilon$ is bona fide a net for a subset of vectors $\Sigma_\epsilon$ determined via its threshold function $\tau_L(v)$.
\begin{Definition}
We define, for any $\epsilon>0$, the subset \begin{equation}
    \Sigma_\epsilon:=\{v\in\mathcal{I}(D_1,D_2):\mathcal{T}_L(v)\in[\epsilon,2\epsilon]\}\subset\mathbb{S}^{n-2}\times\mathbb{S}^{n-1}. 
\end{equation}
\end{Definition}
Recall that the threshold function $\tau_L$ was defined in \eqref{tauells}.

Then we prove that $\mathcal{N}_\epsilon$ is a net of $\Sigma_\epsilon$ in the following sense:
\begin{lemma}\label{approximation}
Fix $\epsilon\in(0,\kappa_0/8)$ and $d\leq 2^{-11}B^{-4}n$. For any $v\in\Sigma_\epsilon$ we can find $u\in\mathcal{N}_\epsilon$ satisfying $\|u-v\|_\infty\leq 4\epsilon n^{-1/2}$.
\end{lemma}

\begin{proof}
    For this $v\in\Sigma_\epsilon$ define a random vector $r=(r_1,\cdots,r_{2n-1})$ with independent coordinates, $\mathbb{E}r_i=0$, $|r_i|\leq 4\epsilon n^{-1/2}$ and satisfy $v-r\in 4\epsilon n^{-1/2}\mathbb{Z}^{2n-1}$ for any $r$. The aim is to prove with positive probability we can find some $u:=v-r\in\mathcal{N}_\epsilon$.

    By definition we have $\|r\|_\infty\leq 4\epsilon n^{-1/2}$ and $u\in\mathcal{I}'(D_1,D_2)$ since $v\in\mathcal{I}(D_1,D_2)$ and $\|u-v\|_\infty\leq4\epsilon/\sqrt{n}$. Next we prove the existence of a $u:=v-r$ such that 
    \begin{equation}\label{firstineq}
        \mathbb{P}(\|Mu\|_2\leq4\epsilon\sqrt{n})\geq (L\epsilon)^{2n-1}\quad\text{and }\mathcal{L}_{\mathcal{L}_A,op}(u,\epsilon\sqrt{n})\leq (2^8L\epsilon)^{2n-1}.
    \end{equation} The second estimate on $\mathcal{L}_{\mathcal{L}_A,op}$ holds for any $u$ satisfying the first estimate. Specifically, for this $u=v-r\in\mathbb{R}^{2n-1}$ let $w(u)\in\mathbb{R}^{2n-1}$ satisfy
    $$\begin{aligned}
\mathcal{L}_{\mathcal{L}_A,op}(u,\epsilon\sqrt{n})&=\mathbb{P}^\mathcal{K}(\|\mathcal{L}_Av-\mathcal{L}_Ar-w(u)\|\leq\epsilon\sqrt{n})\\&\leq \mathbb{P}^\mathcal{K}(\|\mathcal{L}_Av-w(v)\|\leq5\epsilon\sqrt{n})\\&\leq \mathcal{L}_{\mathcal{L}_A,op}(v,5\epsilon\sqrt{n})\leq(2^8L\epsilon)^{2n-1},\end{aligned}$$ 
where the last inequality follows from applying Lemma \ref{lemmareplacement} to this $v\in\Sigma_\epsilon.$

For the first inequality in \eqref{firstineq}, we claim that the following inequality holds true:
\begin{equation}\label{eq854}
    \mathbb{E}_u\mathbb{P}_M(\|Mu\|_2\leq4\epsilon\sqrt{n})\geq\frac{1}{2}\mathbb{P}_M(\|Mv\|_2\leq 2\epsilon\sqrt{n})\geq\frac{1}{2}(2\epsilon L)^{2n-1} 
\end{equation} where the second inequality holds by the definition of $v\in\Sigma_\epsilon$. The validity of equation \eqref{eq854} then implies that we can find $u\in\Sigma_\epsilon$ satisfying the requirement. 

Then we only need to prove the first inequality in \eqref{eq854}. Consider the event $\mathcal{E}:=\{M:\|Mv\|_2\leq2\epsilon\sqrt{n}\}$, then
$$
\mathbb{P}_M(\|Mu\|_2\leq 4\epsilon\sqrt{n})\geq \mathbb{P}_M(\|Mr\|_2\leq2\epsilon\sqrt{n}\text{ and }\mathcal{E}),
$$ so that 
$$
\mathbb{P}_M(\|Mu\|_2\leq4\epsilon\sqrt{2n})\geq(1-\mathbb{P}_M(\|Mr\|_2>2\epsilon\sqrt{n}\mid\mathcal{E}))\mathbb{P}_M(\|Mv\|_2\leq2\epsilon\sqrt{n}).
$$ Taking expectation with respect to $u$ on both sides and exchanging expectation, we see that it suffices to prove
$$
\mathbb{E}_M[\mathbb{P}_r(\|Mr\|_2>2\epsilon\sqrt{n})\mid\mathcal{E}]\leq\frac{1}{2}. 
$$ For this, we check that for any $M\in\mathcal{E}$ by Markov's inequality there holds $\mathbb{P}_r(\|Mr\|_2\geq2\epsilon\sqrt{n})\leq\frac{1}{4}$. Indeed, for a fixed $(2n-1)\times(2n-1)$ matrix $M$ with $|M_{ij}|\leq 16B^2$ (see the definition of $\xi_\nu$) and at most $4dn$ nonzero entries, we have
$$
\mathbb{E}_r\|Mr\|_2^2=\sum_{i}\mathbb{E}
(\sum_j M_{i,j}r_j)^2=\sum_i\mathbb{E}r_i^2\sum_j M_{ij}^2\leq 2^{11}B^4\epsilon^2d\leq\epsilon^2n,$$ using independence and mean zero property of $r_i$ and $\|r\|_\infty\leq4\epsilon/\sqrt{n}$.
\end{proof}

\subsection{A different net via LCD}
We need to transfer the net constructed from small ball probability estimates of $M$ to a net constructed from the LCD.

Let $c_\Sigma>0$ be a small constant to be chosen later. Consider the set of structured directions on the product sphere
\begin{equation}
\Sigma=\Sigma_{\alpha,\gamma}:=\{v\in\mathbb{S}^{n-2}\times\mathbb{S}^{n-1}:D_{\alpha,\gamma}(v^1)\leq e^{C_{\Sigma} n} \text{ or } D_{\alpha,\gamma}(v^2)\leq e^{C_{\Sigma} n}\}.
\end{equation}
That is, we consider those $v$ where either $v^1$ or $v^2$ has a small LCD.

The region $\Sigma\subset\mathbb{S}^{n-2}\times\mathbb{S}^{n-1}$ will be covered by two subsets $S,S'$ under different criterions. Define
$$
S:=\{v\in\mathbb{S}^{n-2}\times\mathbb{S}^{n-1}:\tau_L(v)\geq \exp(-2c_\Sigma n)\},
$$ and define
$$
S':=\{v\in\mathbb{S}^{n-2}\times\mathbb{S}^{n-1}:\tau_L(v)\leq\exp(-2c_\Sigma n),\text{but }\min(D_{\alpha,\gamma}(v^1),D_{\alpha,\gamma}(v^2))\leq \exp(c_\Sigma n)\}.
$$

We also recall the fact that incompressible vectors have a lower bound on its LCD, once we have chosen $\gamma$ sufficiently small. See \cite{campos2024least}, Fact III.6.

\begin{fact}\label{fact2.35}
    Let $v\in\mathbb{S}^{n-1}$ be such that $\kappa_0n^{-1/2}\leq |v_i|\leq \kappa_1n^{-1/2}$ for $i\in D\subset[n]$. Set $\gamma<\kappa_0\sqrt{|D|/2n}$. Then $D_{\alpha,\gamma}(v)\geq (2\kappa_1)^{-1}\sqrt{n}$.
\end{fact}

We have already constructed a net $\Sigma_\epsilon$ for $S$. Our next goal is to construct a net for $S'$, for which we will use a standard net construction based on LCD. For fixed subsets $D_1\subset[n-1],D_2\subset[n,2n-1]$, denote by 
\begin{equation}\label{sigmaepsilon'}
\Sigma_\epsilon':=\{v\in\mathcal{I}(D_1,D_2)\cap S':\min(D_{\alpha,\gamma}(v^1),D_{\alpha,\gamma}(v^2))\in [(4\epsilon)^{-1},(2\epsilon)^{-1}]
\}.
\end{equation} The net we shall use is
\begin{equation}\begin{aligned}
    G_\epsilon:=&\left\{(\frac{p_1}{\|p_1\|_2},\frac{p_2}{\|p_2\|_2}):p\in (\mathbb{Z}^{[n-1]}\oplus \sqrt{\alpha}\mathbb{Z}^{[n,2n-1]})\cap (B_{n-2,n-1}(0,\epsilon^{-1})\setminus\{0\})
    \right\}
 \\&   \cup \left\{(\frac{p_1}{\|p_1\|_2},\frac{p_2}{\|p_2\|_2}):p\in (\sqrt{\alpha}\mathbb{Z}^{[n-1]}\oplus \mathbb{Z}^{[n,2n-1]})\cap (B_{n-2,n-1}(0,\epsilon^{-1})\setminus\{0\})
    \right\}\end{aligned},
\end{equation} where we denote by $p=(p_1,p_2)\in\mathbb{R}^{2n-1}$ and we have used the following notation $$B_{n-2,n-1}(0,\epsilon^{-1})\setminus\{0\}:=(B_{n-2}(0,\epsilon^{-1})\setminus\{0\})\times  (B_{n-1}(0,\epsilon^{-1})\setminus\{0\}).$$

We first check that $G_\epsilon$ is a net for $\Sigma_\epsilon'$.
\begin{lemma}
    Let $\epsilon\leq\gamma(\alpha n)^{-1/2}/8$. For any $v\in\Sigma_\epsilon'$ we can find some $u\in G_\epsilon$ with $\|u-v\|_2\leq32\epsilon\sqrt{\alpha n}$.
\end{lemma}

\begin{proof}
    Assume without loss of generality that $D=D_{\alpha,\gamma}(v^1)\leq D_{\alpha,\gamma}(v^2)$. Then we can find $p^1\in\mathbb{Z}^{[n-1]}\cap B_{n-2}(0,\epsilon^{-1}),p^1\neq 0$ satisfying 
    $$
\|Dv^1-p^1\|_2\leq\min\{\gamma D,\sqrt{\alpha n}\}\leq \sqrt{\alpha n}.
    $$ Then we choose $P^2\in\sqrt{\alpha}\mathbb{Z}^{[n,2n-1]}\cap B_n(0,\epsilon)^{-1}$ greedily so it satisfies
    $$
\|Dv^2-p^2\|_2\leq\sqrt{\alpha n}.
    $$ Clearly we have $p^2\neq 0$ since $(4\epsilon)^{-1}>\sqrt{\alpha n}$.
    Denoting by $p=\frac{p^1}{\|p^1\|}\oplus \frac{p^2}{\|p^2\|}$, then we have
$$\begin{aligned}
\|(v^1,v^2)-p\|_2&\leq\frac{1}{D}(\|Dv^1-p^1\|_2+\|Dv^2-p^2\|_2+|D-\|p_1\|_2|+|D-\|p_2\|_2|)\\&\leq 4D^{-1}\sqrt{\alpha n}\leq 32\epsilon\sqrt{\alpha n},
\end{aligned}$$ as we claimed.
\end{proof}

The cardinality of the net $G_\epsilon$ can be bounded as follows. 
\begin{fact}\label{fact2.37}
    Given $\alpha\in(0,1)$ and some $K>1$, then for any $\epsilon\leq Kn^{-1/2}$ we have
    $$|G_\epsilon|\leq 
\left(\frac{32K}{\epsilon\sqrt{n}}\right)^{2n-1}\alpha^{-n/2}.
    $$
\end{fact}

The proof of this fact follows standard arguments (see \cite{rudelson2008littlewood}). The exact exponent of $\alpha$ is very important and specific to our setting: this is the gain we get from considering arithmetic structures of LCD. 

Then  we can modify the net $G_\epsilon$ to be a net $G_\epsilon'$ so the net is now  a subset of $\Sigma_\epsilon'$.
\begin{corollary}\label{netcorollarys}
    For given $\alpha\in(0,1)$, $K\geq 1$ and some $\epsilon\leq Kn^{-1/2}$ we can find some $64\epsilon\sqrt{\alpha n}$ net $G_\epsilon'$ for $\Sigma_\epsilon'$ which also satisfies $G_\epsilon'\subseteq\Sigma_\epsilon'$ and that 
    $$|G_\epsilon'|\leq\left(\frac{32K}{\epsilon\sqrt{n}}\right)^{2n-1}\alpha^{-n/2}.
    $$
\end{corollary}

Before we are ready to prove the main theorem, we recall the following fact that an incompressible factor has LCD at least $\Omega(\sqrt{n})$.

\begin{fact}\label{factsmallmass}
    For any $v=(v^1,v^2)\in\mathcal{I}$ and $\gamma<\kappa_0\sqrt{d/2n}$ then we have 
    $$
D_{\alpha,\gamma}(v^1)\geq (2\kappa_1)^{-1}\sqrt{n},\quad D_{\alpha,\gamma}(v^2)\geq (2\kappa_1)^{-1}\sqrt{n}.
    $$
\end{fact}

\begin{proof}
    It suffices to note that for any $t\leq (2\kappa_1)^{-1}\sqrt{n}$,
    $$
d(tv^1,\mathbb{Z}^{[n-1]})\geq d(tv^1_{D_1},\mathbb{Z}^{[d]})=t\|v_D\|_2\geq t\kappa_0\sqrt{d/2n}>\gamma t.
    $$
\end{proof}

Next we take a union bound via stratifying the whole subset $S$ and finally prove Theorem \ref{theorem2.5s}. The following proposition contains the main technical step towards Theorem \ref{theorem2.5s}:

\begin{Proposition}\label{prop2.40final}
    Fix $B>0$. There exist constants $\alpha,\gamma,c_\Sigma,c\in(0,1)$ depending on $B$ such that 
    $$
\mathbb{P}_{\mathcal{L}_A}(\exists v\in\Sigma,\lambda\in[-4\sqrt{n},\sqrt{n}]:\mathcal{L}_Av=\lambda v)\leq 2e^{-cn}.
    $$
\end{Proposition}

This Proposition involves various constants, and we outline how the constants are chosen before giving the proof to make the structure of the proof more transparent. First, at the end of Step 2 we choose $L$ sufficiently large relative to the absolute constant. Second, we choose $c_\Sigma>0$ small enough relative to $L$ to apply Theorem \ref{notthesecond}. Third, we choose $\alpha>0$ small relative to $L\kappa_1$ at the end of Step 3. Fourth, we choose $\gamma>0$ small to apply Fact \ref{factsmallmass}. Fifth, the constant $c_0$ was chosen in Step 1 to cover compressible vectors.

\begin{proof} We will prove the result for different subsets of $\Sigma$  in three steps.

\textbf{Step 1. First reduction}.

Since the event $\mathcal{K}:=\{\|A\|_{op}\leq 4\sqrt{n}\}$ holds with probability $1-2\exp(-\Omega(n))$ by Fact \ref{operatornormfact}, and that invertibility on compressible vectors has been verified in Lemma \ref{incompressibilityinvert},
we only need to prove that
\begin{equation}
    q_{n,S}:=\sup_{w\in\mathbb{R}^{2n-1}}\mathbb{P}_{\mathcal{L}_A}^\mathcal{K}(\exists v\in\mathcal{I}\cap S,s\in[-4\sqrt{n},4\sqrt{n}]:\mathcal{L}_Av=sv+w)\leq e^{-\Omega(n)},
\end{equation}
and a similar inequality for  $q_{n,S'}$, where we define $q_{n,S'}$ by using $S'$ in place of $S$ in the definition of $q_{n,S}$. Recall that the set $\mathcal{I}=\mathcal{I}_d$ was defined in \eqref{definitionofid}. In the following we shall take $c_0^2/4\leq \rho^2\delta/2$ and also $c_0^2\leq 2^{-11}B^{-4}$
for $c_0$ the constant in Theorem \ref{notthesecond} and $\rho,\delta$ the constants in Proposition \ref{incompressibilityinvert}, so that when we choose $d\in[c_0^2n/4,c_0^2n]$, then these two theorems and Lemma \ref{approximation} can be used simultaneously on all incompressible vectors.

We first discretize the range of $s$. Take $\mathcal{W}=(2^{-n}\mathbb{Z})\cap [-4\sqrt{n},4\sqrt{n}]$. Then for any $s\in[-4\sqrt{n},4\sqrt{n}]$ we can find $s'\in\mathcal{W}$ with $|s-s'|\leq 2^{-n}$.
Taking a union bound over all possible $D_1,D_2$ and all $s\in\mathcal{W}$, we get 
$$
q_{n,S}\leq 10^n\sup_{w,|s|\leq 4\sqrt{n},D_1,D_2}\mathbb{P}_{\mathcal{L}_A}^\mathcal{K}(\exists v\in \mathcal{I}(D_1,D_2)\cap S,\|\mathcal{L}_Av-sv-w\|_2\leq 2^{-n+1}).
$$ We shall check the supremum of the probability of the right hand side is at most $2^{-4n}$, and check the same upper bound when $S$ is changed to $S'$.

For the set $S$, we take $\eta:=\exp(-2c_\Sigma n)$, so that for any $v\in S$, we have
\begin{equation}
    \eta\leq\mathcal{T}_L(v)\leq 1/L\leq\kappa_0/8,
\end{equation} where we will later take $L$ large enough to ensure that the second inequality is satisfied.
Then we may write
$$
\mathcal{I}(D_1,D_2)\cap S\subseteq \cup_{j=0}^{j_0} \{v\in\mathcal{I}:\mathcal{T}_L(v)\in [2^j\eta,2^{j+1}\eta]\}=\cup_{j=0}^{j_0}\Sigma_{2^j\eta},
$$ taking $j_0$ as the largest integer with $2^{j_0}\eta\leq\kappa_0/2$.

Therefore we only need to prove, for each $\epsilon\in[\eta,\kappa_0/4]$, 
\begin{equation}
\mathcal{Q}_\epsilon:=\sup_{w\in\mathbb{R}^{2n-1},|s|\leq4\sqrt{n}}\mathbb{P}_{\mathcal{L}_A}^\mathcal{K}(\exists v\in\Sigma_\epsilon:\|\mathcal{L}_Av-sv-w\|_2\leq 2^{-n+1})\leq 2^{-4n}.
\end{equation}

Next we introduce a similar partition for the set $S'$. Since for any $v\in \mathcal{I}(D_1,D_2)\cap S'$, 
$$
(2\kappa_1)^{-1}\sqrt{n}\leq \min(D_{\alpha,\gamma}(v^1),D_{\alpha,\gamma}(v^2))\leq \exp(c_{\Sigma } n)=\eta^{-1/2}, 
$$ where the lower bound is from Fact \ref{factsmallmass}. We have the inclusion
$$
\mathcal{I}(D_1,D_2)\cap S'\subseteq \cup_{j=-1}^{j_1}\Sigma_{2^j\sqrt{\eta}}'
$$ where $j_1$ is the smallest integer such that $2^{j_1}\sqrt{\eta}\geq \kappa_1/2\sqrt{n}$ and where we define
$$
\Sigma_\epsilon':=\{v\in\mathcal{I}(D_1,D_2)\cap S':\min(D_{\alpha,\gamma}(v^1),D_{\alpha,\gamma}(v^2))\in [(4\epsilon)^{-1},(2\epsilon)^{-1}].
$$Therefore, for any $\epsilon\in[\sqrt{\eta},\kappa_1\sqrt{n}]$ it suffices to prove that we have
\begin{equation}\mathcal{Q}_\epsilon':=\max_{w\in\mathbb{R}^{2n-1},|s|\leq 4\sqrt{n}}\mathbb{P}^\mathcal{K}_{\mathcal{L}_A}(\exists v\in \Sigma_\epsilon':\|\mathcal{L}_Av-sv-w\|_2\leq 2^{-n+1})\leq 2^{-6n}.
\end{equation}

\textbf{Step 2. Upper bound of $\mathcal{Q}_\epsilon$}. 

We use Lemma \ref{approximation}: for any $v\in\Sigma_\epsilon$ we take a $u\in\mathcal{N}_\epsilon$ such that $\|v-u\|_\infty\leq4\epsilon n^{-1/2}$. Then using $\|A\|_{op}\leq 4\sqrt{n}$, we have
$$\begin{aligned}
\|\mathcal{L}_Au-su-w\|_2&\leq \|\mathcal{L}_Av-sv-w\|_2+|s|\|v-u\|_2+\|\mathcal{L}_A(v-u)\|_2\\&\leq \|\mathcal{L}_Av-sv-w\|_2+8\sqrt{n}\|v-u\|_2\leq 65\epsilon\sqrt{n}.
\end{aligned}
$$
Therefore we have the inclusion
$$
\{\exists v\in\Sigma_\epsilon:\|\mathcal{L}_Av-sv-w\|_2\leq 2^{-n+1}\}\cap\mathcal{K}\subseteq\{\exists u\in\mathcal{N}_\epsilon:\|\mathcal{L}_Au-su-w\|_2\leq 65\epsilon\sqrt{n}
\}.$$
After taking a union bound over $\mathcal{N}_\epsilon$, we deduce that
$$\begin{aligned}\mathcal{Q}_\epsilon&\leq\mathbb{P}_{\mathcal{L}_A}^\mathcal{K}(\exists v\in\mathcal{N}_{\epsilon}:\|\mathcal{L}_Av-sv-w\|\leq 65\epsilon\sqrt{n})\\&\leq\sum_{u\in\mathcal{N}_\epsilon}\mathbb{P}_{\mathcal{L}_A}^\mathcal{K}(\|\mathcal{L}_Au-su-w\|_2\leq 65\epsilon\sqrt{n})
\leq\sum_{u\in\mathcal{N}_{\epsilon}}\mathcal{L}_{\mathcal{L}_A,op}(u,65\epsilon\sqrt{n}).\end{aligned}
$$

Then we use the fact that $\mathcal{L}_{\mathcal{L}_A,op}(u,65\epsilon\sqrt{n})\leq (131)^{2n}\mathcal{L}_{\mathcal{L}_A,op}(u,\epsilon\sqrt{n})$ (which follows from a covering argument, see \cite{campos2025singularity}, Fact 6.2), so that, for any fixed $u\in\mathcal{N}_\epsilon$ there holds $\mathcal{L}_{\mathcal{L}_A,op}(u,65\epsilon\sqrt{n})\leq (2^{19}L\epsilon)^{2n-1}$.  Thus applying Theorem \ref{notthesecond},
$$
\mathcal{Q}_\epsilon\leq|\mathcal{N}_\epsilon|(2^{19}L\epsilon)^{2n-1}\leq \left(\frac{C}{L^2\epsilon}\right)^{2n-1}(2^{19}L\epsilon)^{2n-1}\leq 2^{-4n}.
$$ In the last inequality we choose $L$ sufficiently large with respect to the universal constant $C$. For the last step, we note that
$$
\log 1/\epsilon\leq \log 1/\eta=2c_\Sigma n\leq \frac{1}{4}c_0^3nL^{-32/c_0^2}
$$ when $c_\Sigma>0$ is chosen sufficiently small compared to $L^{-1}$ for the last inequality to hold,
so we can apply Theorem \ref{notthesecond}. 

\textbf{Step 3. Upper bound on $\mathcal{Q}_\epsilon'$}. 

We will apply Corollary \ref{netcorollarys} with $K=\kappa_1$ so that for each $v\in\Sigma_\epsilon'$ we have $u\in G_\epsilon'\subset\Sigma_\epsilon'$ and $\|v-u\|_2\leq 32\epsilon\sqrt{\alpha n}$. Then 
$$
\mathcal{Q}_\epsilon'\leq \left(\frac{32\kappa_1}{\epsilon\sqrt{n}}\right)^{2n-1}\alpha^{-n/2}\sup_{u\in G_\epsilon'}\mathcal{L}(\mathcal{L}_Au,2^9\epsilon\sqrt{\alpha n})\leq (2^{20}L\kappa_1)^{2n-1}\alpha^{(n-1)/2}.
$$  Now we set $\alpha>0$ sufficiently small compared to $L\kappa_1$ so that the right hand side is at most $2^{-6n}$. Combining Step 2 and Step 3, we complete the proof of Proposition \ref{prop2.40final}.

\end{proof}

Finally, we can complete the proof of Theorem \ref{theorem2.5s}.

\begin{proof}[\proofname\ of Theorem \ref{theorem2.5s}]
Recall that the theorem states all eigenvectors $v=(v^1,v^2)$ associated to a nonzero eigenvalue of $\mathcal{L_A}$ satisfy that both $v^1$ and $v^2$ are incompressible and have large LCD: all these claims are proven in Proposition \ref{prop2.40final}. The theorem also claims that the eigenvector to the zero eigenvalue of $\mathcal{L}_A$ is of the form $(0,v^0)$ where $v_0$ is incompressible with large LCD. This can be checked combining two facts: first, an $n\times (n-1)$ matrix with i.i.d. entry is invertible with probability $1-2\exp(-\Omega(n))$ using the result of \cite{rudelson2009smallest}. Second, any unit vector lying in the kernel of an $n-1\times n$ matrix with i.i.d. entry is incompressible and has LCD at least $\exp(c'n)$ for some $c'>0$ with probability $1-e^{-\Omega(n)}$, and this fact can be read off from \cite{rudelson2008littlewood}. This justifies all the claims.
    
\end{proof}

\subsection{Generalization to the rectangular case}\label{rectangeneral}
We outline here a proof of Theorem \ref{theoremrectan}. As this theorem was added after completion of the manuscript, we mainly outline how we can modify the proof of Theorem \ref{theorem1} to get Theorem \ref{theoremrectan}.

We similarly define 
$$L_Z:=\begin{pmatrix}
    0&Z\\Z^*&0
\end{pmatrix},\quad \mathcal{L}_Z:=\begin{pmatrix}
        0&Z_{/[1]}\\Z_{/[1]}^T&0
    \end{pmatrix} .$$
where $Z_{/[1]}$ is an $(N-1)\times n$ matrix with i.i.d. entries of law $\xi$.

Let $v=(v^1,v^2)$ where $v^1\in\mathbb{R}^{N-1}$ and $v^2\in\mathbb{R}^n$ be an eigenvector of $\mathcal{L}_Z$ with nonzero eigenvalue. Then similarly to Fact \ref{fact2.43}, we can assume that $\|v^1\|_2=\|v^2\|_2=1$.

We shall always consider nonzero eigenvalues of $\mathcal{L}_Z$ in the proof because, using the fact that $\sigma_{min}(Z)>0$ with probability $1-\exp(-\Omega(n))$ by \cite{rudelson2009smallest}, so that by Cauchy interlacing theorem, repeated (or close enough) singular values of $Z$ should only be associated to a nonzero eigenvalue of $\mathcal{L}_Z$ when applying Fact \ref{fact2.2}.

Denote by $\mathcal{K}':=\{\|Z\|\leq 2(1+\mathbf{a})\sqrt{n}\}$. Then we again have $\mathbb{P}(\mathcal{K}')\geq 1-\exp(-\Omega(n))$ (see \cite{feldheim2010universality}). We can rule out compressible vectors similarly as in Proposition \ref{proposition2.15} and get
\begin{Proposition}We can find $c,\delta,\rho>0$ depending only on $\xi$ and $\mathbf{a}$ so that 
$$\begin{aligned}
\mathbb{P}^\mathcal{K'}&(\text{There exists }\lambda\in[-2(1+\mathbf{a})\sqrt{n},2(1+\mathbf{a})\sqrt{n}])\text{ and } v=(v^1,v^2)\in\mathbb{S}^{N-2}\times\mathbb{S}^{n-1}\\&\text{such that } \mathcal{L}_Zv=\lambda v,\text{and either }v^1\in\operatorname{Comp}_{N-1}(\delta,\rho) \text{ or }v^2\in\operatorname{Comp}_n(\delta,\rho))\leq e^{-cn}.
\end{aligned}$$
    
\end{Proposition}

For two constants $0<\kappa_0<\kappa_1$ and $D_1\subset[N-2],D_2\subset[N,N+n-1]$ we define 
$$\begin{aligned}
\mathcal{I}(D_1,D_2)&=\{v=(v^1,v^2)\in\mathbb{S}^{N-2}\times\mathbb{S}^{n-1}:(\kappa_0+\kappa_0/2)n^{-1/2}\leq |(v^1)_i|\leq (\kappa_1-\kappa_0/2)n^{-1/2}\\& \text{ for all }i\in D_1,(\kappa_0+\kappa_0/2)n^{-1/2}\leq |(v^2)_i|\leq (\kappa_1-\kappa_0/2)n^{-1/2} \text{ for all }i\in D_2\},\end{aligned}
$$   and for $d\in\mathbb{N}_+$ we define 
$
\mathcal{I}=\mathcal{I}_d:\cup_{D_1\subset[N-1],D_2\subset[N,N+n-1]:|D_1|=|D_2|=d}\mathcal{I}(D_1,D_2).$ Then 

\begin{Proposition}\label{ll22233}
    We can find positive constants $c,\kappa_0,\kappa_1,d_{\ref{ll22233}}$ depending only on $\xi$ and $\mathbf{a}$ such that for any $d\leq d_{\ref{ll22233}}n$, 
    $$\mathbb{P}^{\mathcal{K}'}(\mathcal{L}_Z\text{ has an eigenvector in }(\mathbb{S}^{N-2}\times\mathbb{S}^{n-1})\setminus \mathcal{I})\leq\exp(-cn).$$
\end{Proposition}
We now fix the value of $\kappa_0,\kappa_1$. For fixed $D_1\in[N-2],D_2\subset[N,N+n-1]$ define
$$
\mathcal{I}'(D_1,D_2):=\{v\in\mathbb{R}^{N+n-1}: \kappa_0 n^{-1/2}\leq|v_i|\leq \kappa_1 n^{-1/2} \text{ for all }i\in D_1\cup D_2
\},
$$ 
$$
\Lambda_\epsilon:=(B_{N-2}(0,2)\times B_{n-1}(0,2))\cap (4\epsilon n^{-1/2}\cdot\mathbb{Z}^{N+n-1})\cap \mathcal{I}'(D_1,D_2).
$$ 
For fixed $D_1,D_2$, we again define the zeroed out matrix $M$ associated to $\mathcal{L}_Z$ via
$$
\begin{bmatrix}
0&\begin{bmatrix}0&H_2^T\\H_1&0\end{bmatrix}\\\begin{bmatrix}0&H_1^T\\H_2&0\end{bmatrix}&0
    \end{bmatrix}
$$
where $H_2$ occupies the columns with label in $D_1$ and rows with label in $[N,N+n-1]\setminus D_2$, and $H_1$ occupies the column with label in $D_2$ and rows with label $[N-1]\setminus D_1$. $H_1$ and $H_2$ has i.i.d. entries of distribution $\xi_\nu$. (Note there is a tiny difference from previous treatment: we previously defined $D=D_1\cup(D_2-n+1)$ and let $H_1,H_2$ have $|D|$ columns. This definition does not work out here, but the definition here works in every rectangular case).

We can prove a similar result as in Lemma \ref{smalllemmas}. There is again an abuse of notation where $N$ stands for both the number of rows in $Z$ and the size of the box $\mathcal{B}$, so we use shall use $N'$ here for the latter quantity. We define an $(N',\kappa,D_1,D_2)$-box just as in Definition \ref{definitionbox} where we straightforwardly change $2n-1$ there by $N+n-1$ and $N$ there by $N'$.

\begin{lemma}\label{}
    For any $L\geq 2$ and $0<c_0<2^{-50}B^{-4}$, consider $n>L^{64(1+\mathbf{a})^2/c_0^2}$ with $\frac{1}{4}c_0^2n\leq d\leq c_0^2n$. Given $N'\geq 2$ that satisfies $\log N'\leq c_0L^{-8(1+\mathbf{a})n/d}d$ and any constant $\kappa\geq 2$.  Consider $\mathcal{B}$ a $(N',\kappa,D_1,D_2)$-box and let $X$ be uniformly generated from $\mathcal{B}$. Then 
    $$
\mathbb{P}_X\left(\mathbb{P}_M(\|MX\|_2\leq n)\geq \left(\frac{L}{N'}\right)^{N+n-1}
\right)\leq \left(\frac{R}{L}\right)^{2(N+n-1)},
    $$
    with the constant $R:=C(1+\mathbf{a})c_0^{-3}$ and where $C>0$ is a universal constant.
\end{lemma}
To prove this, we first show a variant of Fact \ref{variantof694}, then apply Theorem \ref{essentiallcd2d} to each $H_1,H_2$. The choice of $R$ and $n$ depending on $\mathbf{a}$ arises from the conditions of Theorem \ref{essentiallcd2d} applied to a matrix of size $(N-d)\times 2d$.

We also have the threshold function and nets: for a fixed $L>0$, we define
\begin{equation}
\tau_L(v):=\sup\{t\in[0,1]:\mathbb{P}(\|Mv\|_2\leq t\sqrt{n})\geq (4Lt)^{N+n-1}\},
\end{equation}

\begin{equation}
\mathcal{N}_\epsilon:=\{v\in\Lambda_\epsilon:(L\epsilon)^{N+n-1}\leq\mathbb{P}(\|Mv\|_2\leq4\epsilon\sqrt{n}),\mathcal{L}_{\mathcal{L}_Z,op}(v,\epsilon\sqrt{n})\leq (2^8L\epsilon)^{N+n-1}\}.
\end{equation} Then we can prove, similarly to Theorem \ref{notthesecond}, that 
\begin{theorem}
    For any $L\geq 2$, $0<c_0\leq 2^{-50}B^{-4}$, $n\geq L^{64(1+\mathbf{a})^2/c_0^2}$ and $d\in[c_0^2n/4,c_0^2n]$. Take $\epsilon>0$ with $\log\epsilon^{-1}\leq\frac{1}{4}c_0^3 nL^{-16(1+\mathbf{a})/c_0^2}$. Then for a universal constant $C>0$ depending only on $\mathbf{a}$, we have
$$
|\mathcal{N}_\epsilon|\leq \left(\frac{C}{c_0^6 L^2\epsilon}\right)^{N+n-1}.
$$
\end{theorem}
A result analogous to Lemma \ref{approximation} can be proven showing that whenever $\epsilon\leq\kappa_0/8$ and $d\leq 2^{-10}B^{-4}n(1+\mathbf{a})^{-1}$, then for any $v\in\Sigma_\epsilon$ we can find $u\in\mathcal{N}_\epsilon$ with $\|u-v\|_\infty\leq 4\epsilon n^{-1/2}$.

Finally, defining $\Sigma_\epsilon'$ as in \eqref{sigmaepsilon'}, we can prove as in Corollary \ref{netcorollarys} that $\Sigma_\epsilon'$ has a $32\epsilon(1+\mathbf{a})\sqrt{\alpha n}$-net $G_\epsilon'$ of cardinality at most 
$(\frac{32K}{\epsilon\sqrt{n}})^{N+n-1}\alpha^{-N/2}$.

Now we have collected all the technical components for the rectangular case, and following the same proof as Theorem \ref{theorem1} concludes the proof of Theorem \ref{theoremrectan}. Our assumption that $n\leq N\leq\mathbf{a}n$ for a fixed $\mathbf{a}<\infty$ ensures that two components of $L_Z$ have comparable size, so no significant change is needed in generalizing the case $N=n$ to $n\leq N\leq\mathbf{a}n$.

\section{Two-point invertibility: the real case}\label{secgwe333}

This section contains the proof of Theorem \ref{theorem1.3}, where we show a quantitative control of least singular values jointly over two locations $\lambda_1,\lambda_2$.

\subsection{Proof outline of Theorem \ref{theorem1.3}}\label{section3.1} We will begin with an adaptation of the geometric approach to quantitative invertibility by Rudelson and Vershynin \cite{rudelson2008littlewood}. However, on a high level, three difficulties will arise in our proof:
\begin{enumerate}
\item As we consider joint estimates at two locations simultaneously, when we use a lemma in \cite{rudelson2008littlewood} that reduces invertibility over incompressible vectors to the inner product of a column with another normal vector, we face a difficulty as the singular vector at the two different locations may not have an overlapping support of good property. More precisely, through a standard procedure we can prove that both singular vectors are incompressible, but that only means they have $c_{\rho,\delta}n$ indices of absolute value within $[\kappa_0 n^{-1/2},\kappa_1 n^{-1/2}]$. When $c_{\rho,\delta}<\frac{1}{2}$ we cannot conclude that both singular vectors at two locations have a common support on which the absolute value of vector coordinates lie in this range. Fortunately, we have an additional input of no-gaps delocalization from \cite{rudelson2016no}, which confirms in a strong form that this is true. Therefore we are now able to generalize \cite{rudelson2008littlewood} and reduce two-location singular value estimates to the joint small ball probability of a column with two different normal vectors.
\item After the reduction in (1), we wish to apply Littlewood-Offord type inequalities to eliminate normal vectors having a rigid arithmetic structure. But this time we have two normal vectors, and we need to show the vector pair has large LCD with high possibility. The difficulty here is that linear span of two normal vectors is not lying in the kernel of one linear equation, say it is not annihilated by either $A-\lambda_1 I_n$ or $A-\lambda_2 I_n$ (with one column or row removed), or by any linear combinations of them. (There is an exception when $|\lambda_1-\lambda_2|\to 0$ so that we could still use the standard strategy, as is explored in \cite{ge2017eigenvalue} and \cite{luh2021eigenvectors}, but this is not working in our general case.) This linear span can be annihilated by a quadratic equation of $A$, but Littlewood-Offord type theorems are difficult to apply for quadratic equations. This is the major technical difficulty of our proof, and we will solve it by taking a linearization reduction and consider a $2n$-dimensional system instead. To eliminate all vectors in the linear span, we need to consider a continuously parametrized family of the linearization matrix all at once, and show with high possibility the kernel of any of them has no vector with small LCD.

\item A further technical difficulty arises when $|\lambda_1-\lambda_2|$ is very small but still larger than $e^{-cn}$. In this case it is hard to control the inner product of the two normal vectors to $A-\lambda_1 I_n$, $A-\lambda_2 I_n$ (with a column removed) and their inner product could be arbitrarily close to 1. Here we will take a stratification of unit vectors in $\mathbb{R}^d\times\mathbb{R}^d$ by constructing nets of vectors with small LCD depending on the inner product, such that when the inner product is close to 1, the anticoncentration bound from Littlewood-Offord theorem is weaker but we also have a smaller net, and everything gets balanced out. This is philosophically similar to the real-complex correlation considered in \cite{rudelson2016no}, and then in \cite{ge2017eigenvalue} and \cite{luh2021eigenvectors}. But in our case, we feel there is a technical difficulty to construct vector pairs with intermediate LCD. In \cite{luh2021eigenvectors}, the author can multiply the eigenvector by a phase so the LCD of the vector pair of the real and complex parts for the singular vector is achieved by its real part. We cannot do this here, so we prefer to use a random method to generate the net with intermediate LCD, in the same way as in Section \ref{section222}. Although the proof is longer, it becomes more transparent and unified, and also works efficiently when $|\lambda_1-\lambda_2|$ is very very small. Note that in contrast to Section \ref{section222}, it is not necessary to use the idea of inversion of randomness, or the conditioned Littlewood-Offord theorem \ref{mainlittlewoodofford} here. But unlike in Section \ref{section222}, we face new challenge around the LCD of vector pairs and arbitrarily large overlap here. 

\end{enumerate}

Now we outline the main steps to the proof of Theorem \ref{theorem1.3}. First, we need the powerful notation of no-gaps delocalization from Rudelson and Vershynin \cite{rudelson2016no}. More precisely, the following theorem can be read off from the proof in \cite{rudelson2016no}:

\begin{theorem} \label{nogaps3.1}
    Let $A$ satisfy the assumptions in Theorem \ref{theorem1.3}. We can find $c>0$, $\tau>0$, $\vartheta>0$ depending only on $\xi$ such that for any $\lambda\in\mathbb{R}$ with $|\lambda|\leq 4\sqrt{n}$, the following is true with probability at least $1-e^{-cn}$: 
    
    All unit vectors $w\in\mathbb{R}^n$ satisfying $$
\|(A-\lambda I_n)w\|\leq\tau n^{1/2}
    $$ must satisfy that 
    $$
|\{i\in[n]:|w_i|\geq \vartheta n^{-1/2}\}|\geq \frac{3}{4}n.
    $$
\end{theorem}

We will also need an upper bound for the entry. Since $w$ is a unit vector, we can find some $\Theta>0$ (say, $\Theta^2=20$) such that 
$$
|\{i\in[n]:|w_i|\leq \Theta n^{-1/2}\}|\geq \frac{19n}{20}.
$$ Combining the above two facts, we deduce
\begin{corollary}\label{corollary3.22}
    Let $w[1]$, $w[2]$ be respectively the singular vector associated to the smallest singular value of $A-\lambda_1 I_n$, $A-\lambda_2I_n$. Then we can find $c>0$ depending only on $\xi$ such that, with probability at least $1-e^{-cn}$, the following is true:

    For any fixed $\epsilon>0$, on the event that $\sigma_{min}(A-\lambda_1I)\leq\epsilon n^{-1/2},\sigma_{min}(A-\lambda_2I)\leq\epsilon n^{-1/2}$, 
    \begin{equation}\label{conditionsonj}
|\{i\in[n]:\Theta n^{-1/2}\geq ||w[1]_i|\geq \vartheta n^{-1/2},\quad\Theta n^{-1/2}\geq |w[2]_i|\geq \vartheta n^{-1/2}\}|\geq \frac{1}{4}n
    \end{equation} where we denote by $w[1]_i$, $w[2]_i$ the $i$-th coordinate of $w[1]$ and $w[2]$.
\end{corollary}

This corollary implies the following invertibility-via-distance estimate simultaneously for two locations. This estimate is inspired by Rudelson and Vershynin \cite{rudelson2008littlewood}, Lemma 3.5 but here we consider two locations simultaneously.

\begin{Proposition}\label{invertibilitydistance}
    For each $i=1,2$ and $j\in[n]$ let $X[i]_j$ denote the $j$-th column of $X-\lambda_i I_n$ and $H[i]_j$ denote the linear span of all the columns of $X-\lambda_i I_n$ except the $j$-th. Then let $A$ be the random matrix as in Theorem \ref{theorem2}, we have that for any $\epsilon>0$, 
\begin{equation}\label{line9670}\begin{aligned}
    \mathbb{P}&(\sigma_{min}(A-\lambda_1 I_n)\leq\epsilon n^{-1/2},\sigma_{min}(A-\lambda_2I_n)\leq \epsilon n^{-1/2})\\&\leq\frac{5}{n}\sum_{j=1}^n\mathbb{P}(\operatorname{dist}(X[1]_j,H[1]_j)\leq\epsilon/\vartheta,\operatorname{dist}(X[2]_j,H[2]_j)\leq\epsilon/\vartheta) +e^{-cn},\end{aligned}
\end{equation} where $\vartheta>0$ was determined in Proposition \ref{nogaps3.1} and $c>0$ depends only on $\xi$.
\end{Proposition}
We note that the consideration of compressible singular vectors has already been absorbed in  Corollary \ref{corollary3.22}, so we do not do that again.
\begin{proof} We use the notation $\mathcal{E}_{\ref{invertibilitydistance}}:=\{\sigma_{min}(A-\lambda_1 I_n)\leq\epsilon n^{-1/2}, \sigma_{min}(A-\lambda_2I_n)\leq \epsilon n^{-1/2}\}.$
    Recall that we use $w[1]$, $w[2]$ to denote the least singular vectors. Then 
    \begin{equation}\label{definiitonmin}\sigma_{min}(A-\lambda_1 I_n)=
\|(A-\lambda_1 I_n)w[1]\|\geq\sup_{j\in[n]} |w[1]_j|\operatorname{dist}(X[1]_j,H[1]_j).
    \end{equation}
    By Corollary \ref{corollary3.22}, there exists an event $\mathcal{D}_{\ref{invertibilitydistance}}$
 with probability at least $1-e^{-cn}$, such that on $\mathcal{D}_{\ref{invertibilitydistance}}\cap\mathcal{E}_{\ref{invertibilitydistance}}$, there are at least $\frac{1}{4}n$ indices $j\in[n]$ such that $\min(|w[1]_j|,|w[2]_j|)\geq \vartheta n^{-1/2}$.

    Now we denote by $$p_k:=\mathbb{P}(\operatorname{dist}(X[1]_k,H[1]_k)\leq\epsilon/\vartheta,\text{ and }\operatorname{dist}(X[2]_k,H[2]_k)\leq\epsilon/\vartheta),$$ where $\operatorname{dist}(X[1]_k,H[1]_k)$ is the distance of the vector $X[1]_k$ to the subspace $H[1]_k\subset\mathbb{R}^n$.
    Then we have
    $$
\mathbb{E}|\{k:\operatorname{dist}(X[1]_k,H[1]_k)\leq\epsilon/\vartheta, \text{ and }\operatorname{dist}(X[2]_k,H[2]_k)\leq\epsilon/\vartheta\}|=\sum_{k=1}^n p_k.
    $$
Now we let $U$ denote the event that
$$\sigma_1:=\{k\in[n]:\operatorname{dist}(X[1]_k,H[1]_k)\geq\epsilon/\vartheta\text{ or} \operatorname{dist}(X[2]_k,H[2]_k)\geq\epsilon/\vartheta\}$$
has at least $\frac{4}{5}n$ elements. Then from Chebyshev's inequality we have $$\mathbb{P}(U^c)\leq \frac{5}{n}\sum_{k=1}^n p_k.$$

On the event $U\cap\mathcal{D}_{\ref{invertibilitydistance}}\cap\mathcal{E}_{\ref{invertibilitydistance}}$, by Pigeonhole principle there exists some $j\in[n]$ such that $|w[1]_j|\geq\vartheta n^{-1/2},|w[2]_j|\geq\vartheta n^{-1/2}$ and at least one of the following two claims holds: $$\operatorname{dist}(X[1]_j,H[1]_j)\geq\epsilon/\vartheta \text{ or } \operatorname{dist}(X[2]_j,H[2]_j)\geq\epsilon/\vartheta.$$ This implies by \eqref{definiitonmin} that either $\sigma_{min}(A-\lambda_1 I_n)\geq\epsilon n^{-1/2}$ or $\sigma_{min}(A-\lambda_2 I_n)\geq\epsilon n^{-1/2}$ on $U\cap\mathcal{D}_{\ref{invertibilitydistance}}\cap\mathcal{E}_{\ref{invertibilitydistance}}$, which cannot happen. Thus $\mathbb{P}(\mathcal{E}\cap U_{\ref{invertibilitydistance}}\cap\mathcal{D}_{\ref{invertibilitydistance}})=0$, so that  $\mathbb{P}(\mathcal{E})\leq \mathbb{P}(U_{\ref{invertibilitydistance}}^c\cup \mathcal{D}_{\ref{invertibilitydistance}}^c)$. The result follows.
    
\end{proof}

In Proposition \ref{invertibilitydistance}, note that the columns $X[1]_j$ and $X[2]_j$ are essentially the same random vector except a constant shift in its $j$-th coordinate. The normal vectors to the subspace $H[1]_j$ and $H[2]_j$ are however highly dependent, as they also use the same set of random variables modulo a deterministic shift.

We will therefore need a Littlewood-Offord theory in two dimensions to deduce a small ball probability bound for the right hand side of \eqref{line9670}. In this theory we need to estimate the essential LCD of the linear subspace spanned by both the normal vector of $H[1]_j$ and the normal vector of $H[2]_j$.

\subsubsection{Overlap of normal vectors}
Let $X^*[1]_j$ be a unit normal vector of $H[1]_j$ and $X^*[2]_j$ a unit normal vector of $H[2]_j$, for each $j\in[n]$. We first compute the inner product of $X^*[1]_j$ and $X^*[2]_j$ and show it is bounded away from $\pm1$, in a quantitative sense depending on $|\lambda_1-\lambda_2|$.

We first need to condition on a high probability event where normal vectors are incompressible. The following lemma can be found  in \cite{rudelson2008littlewood}.
\begin{lemma}\label{normalincomp}
   There exists constants $c_\eqref{normalincomp},\rho_\eqref{normalincomp},\delta_\eqref{normalincomp}>0$  and an event $\Omega_\eqref{normalincomp}$ 
   such that $\mathbb{P}(\Omega_\eqref{normalincomp})\geq 1-e^{-c_\eqref{normalincomp}n}$, and that on $\Omega_\eqref{normalincomp}$, all the normal vectors $X^*[1]_j,{j\in[n]}$ and $X^*[2]_j,{j\in[n]}$ are $(\delta_\eqref{normalincomp},\rho_\eqref{normalincomp})$-incompressible.
\end{lemma}

\begin{lemma}\label{realcomplexcorrelate}
    For each $j\in[n]$ let $V_j$ be an $2\times n$ matrix defined as
\begin{equation}\label{line1057}
V_j=\begin{bmatrix}
    (X^*[1]_j)^T\\(X^*[2]_j)^T
\end{bmatrix}.
\end{equation} Then on the event $\Omega_\eqref{normalincomp}\cap\mathcal{K}$, we have 
$$
\det(V_j V_j^T)\geq C|\lambda_1-\lambda_2|^2n^{-1} 
$$ for some $C$ depending only on $\xi$ (and independent of $\lambda_1,\lambda_2$).
\end{lemma}

\begin{proof}
    We take an orthogonal projection of the vector $X^*[2]_j$ onto $X^*[1]_j$ and write 
    \begin{equation}\label{orthogonalforw}
X^*[2]_j=\alpha X^*[1]_j+\beta w_j
    \end{equation} where real numbers $\alpha,\beta$ satisfy $\alpha^2+\beta^2=1$ and $\langle w_j,X^*[1]_j\rangle=0$.

    Since $X^*[2]_j$ is orthogonal to $H[2]_j$, we write 
    $$
(A^T-\lambda_2 I_n)_{/[j]}\cdot (\alpha X^*[1]_j+\beta w_j)=0, 
    $$ where for an $n\times n$ matrix $M$ we use the notation $M_{/[j]}$ to mean the matrix $M$ with its $j$-th row set to be identically zero.
    But we also have $
(A^T-\lambda_1 I_n)_{/[j]}\cdot (\alpha X^*[1]_j)=0, 
    $ so that 
    \begin{equation}\label{comparetwosides}
\alpha(\lambda_1-\lambda_2)(I_n)_{/[j]}\cdot X^*[1]_j
    =-(A^T-\lambda_2 I_n)_{/[j]}\cdot \beta w_j.\end{equation}

As we are working on the event $\Omega_c\cap\mathcal{K}$, we have that $\|(I_n)_{/[j]}X^*[1]_j\|\geq C$ where we use $X^*[1]_j$ is incompressible from Lemma \ref{normalincomp} (and we let $C\in(0,1)$ be a constant that depends only on the constants in Lemma  \ref{normalincomp} (so it depends on $\xi$ but not on $\lambda_1,\lambda_2$)). Also, we have $\|(A^T-\lambda_2 I_n)_{/[j]}\cdot w_j\|\leq 8\sqrt{n}$ on the event $\mathcal{K}$ and using $|\lambda_2|\leq 4\sqrt{n}$. Then we deduce that 
\begin{equation}\label{whatisbeta?}
|\beta|\geq C|\lambda_2-\lambda_1|n^{-1/2}/32.\end{equation}
 Indeed, if this is not the case, then since $|\lambda_1-\lambda_2|\leq 8\sqrt{n}$ we have $\beta\leq 1/4$ so that $|\alpha|\geq 3/4$, and a brief comparison of the two sides of \eqref{comparetwosides} shows that the equality does not hold.

Therefore we can compute $|\langle x^*[1]_j,x^*[2]_j\rangle|=|\alpha|$ and compute $\det(V_jV_j^T)=\beta^2$,
which justifies the claim. 
\end{proof}

\subsubsection{LCD of normal vector pairs}
Now we introduce a notion of least common denominator for vector pairs, which is a special case of the notion in \cite{rudelson2009smallest}.
\begin{Definition} Fix $m\in\mathbb{N}_+.$ Consider $a=(a_1,\cdots,a_n)$ a sequence of vectors $a_k\in\mathbb{R}^m$, for any vector $\theta\in\mathbb{R}^m$ we may define the inner product $\theta\cdot a$ via
$$
\theta\cdot a=(\langle\theta,a_1\rangle,\cdots,\langle\theta,a_n\rangle)\in\mathbb{R}^n. 
$$ Then we define, for $\alpha>0,\gamma\in(0,1),$
\begin{equation}\label{whatissmalllcd?}
D_{\alpha,\gamma}(a):=\inf\{\|\theta\|:\theta\in\mathbb{R}^2,\|\theta\cdot a\|_\mathbb{Z}\leq \min(\sqrt{\alpha n},\gamma\|\theta\cdot a\|_2)\}.
\end{equation}
\end{Definition}
In this section we always take $m=2$. For two $n$-dimensional vectors $X,Y$ with coordinates $(X_k),(Y_k)_{k\in[n]}$ we can record the coordinates of $X,Y$ into the vector $a=a(X,Y)=(a_1,\cdots,a_n)$ via setting $a_k=(X_k,Y_k)^T$ for each $k\in[n]$ and define
\begin{equation}
    D_{\alpha,\gamma}(X,Y):=D_{\alpha,\gamma}(a(X,Y)).
\end{equation}

We will use the following version of Littlewood-Offord theorem which is adapted from \cite{rudelson2009smallest}, Theorem 3.3. In this theorem, we consider a group of almost orthogonal vectors, one of which has length one and the other vector has arbitrary length no more than one.  While there are other Littlewood-Offord theorems (see \cite{rudelson2016no}) serving a similar purpose, we find that proving a new one is more convenient as the notion of LCD used here is different from \cite{rudelson2016no}.

\begin{Proposition}\label{proponewlittlewood}
Given a random vector $\xi=(\xi_1,\cdots,\xi_n)$ with i.i.d. coordinates $\xi_k$, where each $\xi_k$ is a mean 0, variance 1, subgaussian random variable satisfying $\mathcal{L}(\xi,1)\leq 1-p$ for some $p\in(0,1)$. Recall that the Lévy concentration function $\mathcal{L}(\xi,\cdot)$ was defined in \eqref{anticoncentrationlevy}.

Consider two vectors $\mathbf{c}=(c_1,\cdots,c_n)$ and $\mathbf{d}=(d_1,\cdots,d_n)$ in $\mathbb{R}^n$ satisfying $\|\mathbf{c}\|_2=1$ and $\|\mathbf{d}\|_2=\omega_n$ for some $w_n\in(0,1]$, and assume that $|\langle\mathbf{c},\mathbf{d}/\|\mathbf{d}\|\rangle|\leq 0.01$. Define a class of two dimensional vectors $a=(a_1,\cdots,a_n)$ where each $a_k=\begin{pmatrix}
    c_k\\d_k
\end{pmatrix}$. 

Consider the random sum $S=\sum_{k=1}^n a_k\xi_k$. Then for any $\alpha>0,\gamma\in(0,1)$ and 
$$
\epsilon\geq\frac{\sqrt{2}}{D_{\alpha,\gamma}(a)}
$$ we have
\begin{equation}\label{levytwoball}
\mathcal{L}(S,\epsilon\sqrt{2})\leq(\omega_n)^{-1}(\frac{C\epsilon}{\gamma\sqrt{p}})^2 +C^2e^{-2p\alpha n}
\end{equation} for some constant $C>0$.    
\end{Proposition} We consider a small inner product between $c$ and $d$ here but not strict orthogonality because in our construction of $\epsilon$-nets, we do not impose strict orthogonality condition. 
We have carefully designed this Proposition so that it captures the $(w_n)^{-1}$ leading factor and the exponent $\epsilon^2$.
We will outline the proof in Section \ref{prooflittlewoodofford}.

The main quasi-randomness condition of this section, which is the key technical difficulty in proving Theorem \ref{theorem2}, is outlined in the following theorem.

\begin{theorem}\label{quasirandomtheorem3}
    Recall the definition of $V_j,j\in[n]$ in \eqref{line1057}. Then for some fixed choices of $\alpha_{\ref{quasirandomtheorem3}}>0$ and $\gamma_{\ref{quasirandomtheorem3}}>0$, we can find constants $C_{\ref{quasirandomtheorem3}},c_{\ref{quasirandomtheorem3}}>0$ and an event $\Omega_{\ref{quasirandomtheorem3}}$ with $\mathbb{P}(\Omega_{\ref{quasirandomtheorem3}})\geq 1-e^{-c_{\ref{quasirandomtheorem3}n}}$ such that on $\Omega_{\ref{quasirandomtheorem3}}$ we have
    $D_{\alpha_{\ref{quasirandomtheorem3}},\gamma_{\ref{quasirandomtheorem3}}}(V_j)\geq \exp(C_{\ref{quasirandomtheorem3}} n)$ for all $j\in[n]$.

    All the constants $\alpha_{\ref{quasirandomtheorem3}},\gamma_{\ref{quasirandomtheorem3}},c_{\ref{quasirandomtheorem3}},C_{\ref{quasirandomtheorem3}}$ depend on $\xi$ but not on $\lambda_1,\lambda_2$.
\end{theorem}

\subsubsection{Proof of main results assuming Theorem \ref{quasirandomtheorem3}}
We now complete the proof of Theorem \ref{theorem2} and Theorem \ref{theorem1.3} assuming Theorem \ref{quasirandomtheorem3}:

\begin{proof}[\proofname\ of Theorem \ref{theorem2} and Theorem \ref{theorem1.3}]
By Proposition \ref{invertibilitydistance} we need to estimate, for each $j\in[n]$ and any $\epsilon>0$, 
    $$\mathbb{P}(\operatorname{dist}(X[1]_j,H[1]_j)\leq\epsilon, \operatorname{dist}(X[1]_j,H[1]_j)\leq\epsilon).
    $$

We shall work with the conditioned version    
    $$\mathbb{P}^{\mathcal{K}\cap \Omega_{\ref{quasirandomtheorem3}}\cap\Omega_{\ref{realcomplexcorrelate}}}(\operatorname{dist}(X[1]_j,H[1]_j)\leq\epsilon, \operatorname{dist}(X[1]_j,H[1]_j)\leq\epsilon)
    ,$$ 
    and this quantity is bounded by 
    $$
\mathcal{L}(V_j\xi,\sqrt{2}\epsilon)
    $$ where $\xi=(\xi_1,\cdots,\xi_n)$ is an i.i.d. copy of a column of $A$. 
    We take the orthogonal projection of $X^*[2]_j$ onto $X^*[1]_j$ as in \eqref{realcomplexcorrelate} and define a new matrix $$\overline{V}_j=\begin{bmatrix}
    (X^*[1]_j)^T\\\beta\cdot w_j^T
\end{bmatrix}.$$ We have $|\beta|\geq C|\lambda_1-\lambda_2|n^{-1/2}/32$ on the conditioned event by Lemma \ref{realcomplexcorrelate} and it is elementary to check that $\mathcal{L}(V_j\xi,\sqrt{2}\epsilon)\leq \mathcal{L}(\overline{V}_j\xi,4\epsilon)$. On the conditioned event, we also have $D_{\alpha_{\ref{quasirandomtheorem3}},\gamma_{\ref{quasirandomtheorem3}}}(\overline{V}_j)\geq D_{\alpha_{\ref{quasirandomtheorem3}},\gamma_{\ref{quasirandomtheorem3}}}(V_j)\geq \exp(C_{\ref{quasirandomtheorem3}} n)$ thanks to Theorem \ref{quasirandomtheorem3}.
    
We can always multiply $\xi$ by a constant so as to have (say) $\mathcal{L}(\xi,1)\leq\frac{3}{4}$, and we make this assumption throughout. Now we apply Theorem \ref{proponewlittlewood} to each $\overline{V}_j$ and estimate $
\mathcal{L}(\overline{V}_j\xi,4\epsilon)
    $:
    $$
\mathcal{L}(\overline{V}_j\xi,4\epsilon)\leq |\beta|^{-1}(\frac{C\epsilon}{\gamma_{\ref{quasirandomtheorem3}}})^2+C^2e^{-c'n} 
    $$ for any $\epsilon>0$, where $C>0$ is a universal constant and $c'>0$ depends on $\alpha_{\ref{quasirandomtheorem3}},\gamma_{\ref{quasirandomtheorem3}},c_{\ref{quasirandomtheorem3}}$.
    This concludes the proof of the whole theorem thanks to Proposition \ref{invertibilitydistance} and the fact that $\mathbb{P}((\mathcal{K}\cap \Omega_{\ref{quasirandomtheorem3}}\cap\Omega_{\ref{realcomplexcorrelate}})^c)\leq\exp(-\Omega(n)).$
\end{proof}

\subsubsection{A brief outline of the rest of the section}

The remaining parts of this section are highly technical, so we outline here the main technical steps, and how our method is different from other methods. First in Section \ref{sect3.2} and \ref{sect3.3} we rule out compressible vectors in the linear span of $X^*[1]_j$ and $X^*[2]_j$. Even this step is nontrivial: the method we take is to note that any vector in the linear span is annihilated by a quadratic polynomial of $A$ (but up to a certain error). This error term will require us to study how large is the $j$-th entry of the vector. Thus we first show that this $j$-th entry cannot be too large, and then show all vectors in the linear span should be incompressible with high probability.

Then we rule out vectors in the linear span with small LCD. This would typically follow from an $\epsilon$-net argument via LCD, but here we have a new difficulty when constructing the net as we have to know for which $\theta\in[0,2\pi]$ we have $\cos(\theta)X^*[1]_j+\sin(\theta)X^*[2]_j$ has the least LCD (in other cases such as \cite{luh2021eigenvectors}, we can multiply the vector by a complex phase and assume that the smallest LCD is achieved, say by $X^*[1]_j$, but not here). The setting further complicates when $|\lambda_1-\lambda_2|$ is small and $X^*[1]_j,X^*[2]_j$ could have large overlap. We take the random generation method from Section \ref{section222} to avoid these complications, and to keep a coherent style throughout the whole paper (despite a bit lengthy argument).

Specifically, Section \ref{sect3.4} is a technical preparation for the random generation method and Section \ref{sect3.5}
defines a family of nets $\mathcal{N}_{\epsilon,\epsilon_1}(c)$ in terms of small ball probability of a certain truncated matrix $M_{\underline{A}}$. Here $\epsilon$ is the scale of small ball probability and $\epsilon_1$ roughly measures the overlap between $X^*[1]_j$ and $X^*[2]_j$. Section \ref{sect3.5} also defines a threshold function $\tau_{L,\epsilon_1}$ to measure the arithmetic structure of the vector with respect to $M_{\underline{A}}$, in a way parameterized by $\epsilon_1>0$. We construct discrete approximations for the vector pair but with different size of the mesh grid in its two components, taking into account $\epsilon_1>0$.
Then in Section \ref{sec3.6} we upper bound the cardinality of $\mathcal{N}_{\epsilon,\epsilon_1}(c)$ via Littlewood-Offord theorems and sow $\mathcal{N}_{\epsilon,\epsilon_1}(c)$ is a net via a random rounding technique. Here we use the randomness in the base vector as we exhaust all of them in an integral lattice: we can show that it has large two-dimensional essential LCD with super-exponential probability.  Finally, in Section \ref{sec3.7} we move back from the net $\mathcal{N}_{\epsilon,\epsilon_1}(c)$ to the net constructed from essential LCD, and obtain Theorem \ref{quasirandomtheorem3}.

Overall, we can successfully deal with all possible cases of overlaps of $X^*[1]_j,X^*[2]_j$ in a unified way and avoid certain difficulties in constructing nets of vector pairs based on LCD when we cannot rotate the direction of smallest LCD to one axis. We will generalize this construction in Section \ref{secgwe444} to complex i.i.d. matrices.

\subsection{Arithmetic structures: first reduction.}\label{sect3.2}

Now we initiate the proof of Theorem \ref{quasirandomtheorem3}. We need to consider the least common denominator of \begin{equation}\label{unitvectors}\theta_1 X^*[1]_j+\theta_2 X^*[2]_j\end{equation} for any $\theta_1,\theta_2\in\mathbb{R}$.

Arguing non-rigorously, this vector is annihilated by a quadratic equation of $A$, that is to say we have something like  \begin{equation}\label{productrules}\left((A^T-\lambda_1 I_n)_{/[j]}(A^T-\lambda_2 I_n)_{/[j]}\right)\cdot\left(\theta_1 X^*[1]_j+\theta_2 X^*[2]_j\right)\stackrel{?}{=}0.\end{equation}
This is however not exactly true, as 
$(A^T-\lambda_1 I_n)_{/[j]}$ and $(A^T-\lambda_2 I_n)_{/[j]}$ do not commute because of a removed row. Nonetheless, the equation is true up to a small error: we have
\begin{equation}\label{righthanddsidebound}\begin{aligned}
    \operatorname{LHS}\text{ of } \eqref{productrules}&=\theta_1(A^T-\lambda_1 I_n)_{/[j]}(\lambda_1-\lambda_2){I_n}_{/[j]}X^*[1]_j\\&=\theta_1(\lambda_2-\lambda_1)(A^T-\lambda_1 I_n)_{[j]}(X^*[1]_j)_{[j]},
\end{aligned}\end{equation}
where $(X^*[1]_j)_{[j]}$ denotes the $j$-th coordinate of the vector $X^*[1]_j$ and  $(A^T-\lambda_1 I_n)_{[j]}$ is the $j-th$ column of $(A^T-\lambda_1 I_n)$ with the $j$-th entry removed. We do not get zero in the last line of \eqref{righthanddsidebound} because ${I_n}_{/[j]}X^*[1]_j$ sets the $j$-th coordinate of $X^*[1]_j$ to be zero but keeps its other coordinates fixed. 

In words, the RHS of \eqref{righthanddsidebound} means the vector \eqref{unitvectors} is annihilated by a quadratic equation of $A^T$, up to an error given by a multiple of the $j$-th column of $A^T$. 

From this first computation, we shall be adopting the following strategy; we will condition on the realization of the $j$-th column of $A^T$ and use randomness from the other $n-1$ columns of $A^T$. Enough randomness is maintained as only one column is conditioned on.
\subsubsection{Ruling out degenerate case}
The conditioning procedure proposed in the last paragraph will break down in one degenerate case: the case where the vector \eqref{unitvectors} is supported only on the $j$-th coordinate or has a mass close to 1 on the $j$-th coordinate and a mass close to 0 elsewhere. This degenerate case does not happen with high possibility, as shown in the following lemma:

\begin{lemma}\label{lemmadegenerate} We can find constants $M_{\ref{lemmadegenerate}}\in(0,1)$ and $c_{\ref{lemmadegenerate}}>0$ such that, on an event $\Omega_{\ref{lemmadegenerate}}$ which holds with probability $1-\exp(-c_{\ref{lemmadegenerate}}n)$, the following two statements are true: \begin{enumerate}\item For any $\theta_1,\theta_2\in\mathbb{R}$ such that $\theta_1X^*[1]_j+\theta_2X^*[2]_j$ is a vector of unit $\ell^2$ norm in $\mathbb{R}^n$, we must have that 
\begin{equation}\label{massaway}
\|\theta_1X^*[1]_j+\theta_2X^*[2]_j\|_{[n]\setminus\{j\}}>M_{\ref{lemmadegenerate}}>0, 
\end{equation} where $\|\cdot\|_{[n]\setminus\{j\}}$ denotes the $\ell^2$-norm of the vector when its $j$-th coordinate is set to be zero. \item Also, for any $\theta_0,\theta_1,\theta_2\in\mathbb{R}$ such that $\theta_1X^*[1]_j+\theta_2X^*[2]_j-\theta_0(X^*[1]_j)_{[j]}e_j$ is a unit vector, we also have $\|\theta_1X^*[1]_j+\theta_2X^*[2]_j-\theta_0(X^*[1]_j)_{[j]}e_j\|_{[n]\setminus\{j\}}> M_{\ref{lemmadegenerate}}>0.$\end{enumerate}
\end{lemma}

\begin{proof}
    We first note that such a pair $(\theta_1,\theta_2)$ must satisfy 
    $\theta_1^2+\theta_2^2+2\theta_1\theta_2\alpha=1$ where $\alpha$ was defined in \eqref{orthogonalforw}. This implies $\theta_1^2(1-\alpha^2)=\theta_1^2\beta^2\leq 1$. But we already have on the event $\mathcal{K}$ that \eqref{whatisbeta?} holds, so that $\theta_1^2|\lambda_1-\lambda_2|^2\leq Cn$ for some $C>0$ depending only on $\xi$.

    Therefore, in the equation \eqref{righthanddsidebound} we will take an $n^{-1/2}$-net for $\theta_1(\lambda_2-\lambda_1)$ which has cardinality $O(n)$ and we also take a $n^{-1/2}$-net for the possible values of $(X^*[1]_j)_j$ which has cardinality $O(n^{1/2})$.

    Having obtained these nets, we turn back to the equation \eqref{righthanddsidebound} and assume that \eqref{massaway} does not hold for some $M<1/4$. Then \eqref{righthanddsidebound} implies that, on the event $\mathcal{K}$, 
    $$\|(A^T-\lambda_1 I_n)_{/[j]}(A^T-\lambda_2 I_n)_{/[j]}(M'e_j)-\theta_1(\lambda_2-\lambda_1)(A^T-\lambda_1 I_n)_{[j]}(X^*[1]_j)_{[j]}\|\leq CM\sqrt{n}
$$ for a $C>0$ depending only on $\xi$. Here $M'\in(1/2,1)$ is the absolute value of the $j$-th coordinate of $\theta_1X^*[1]_j+\theta_2X^*[2]_j$, which is fixed for now and that $e_j$ is the unit vector in the $j$-th coordinate in $\mathbb{R}^n$. We take a $n^{-1/2}$-net for $M'$ as well. Conditioning on a realization of the $j$-th column $(A^T)_{[j]}$ such that $\|(A^T)_{[j]}\|\geq c_1\sqrt{n}$ for some $c_1>0$ (this event happens with probability at least $1-e^{-c_2n}$ by a standard application of Tensorization lemma). Then $(A^T-\lambda_2I_n)_{/[j]}e_j$ is a vector with norm at least $c_1\sqrt{n}$ (and its $j$-th coordinate is zero). Then multiplying this vector by $(A^T-\lambda_1 I_n)_{/[j]}$, we can use the remaining randomness of $A^T$ (its randomness away from the $j$-th row and column has not been used before) to deduce that, conditioning on $(A^T)_{[j]}$, $\mathcal{L}((A^T-\lambda_1 I_n)_{/[j]}(A^T-\lambda_2 I_n)_{/[j]}(M'e_j),c_3\sqrt{n})\leq \exp(-c_4n)$ again by the standard tensorization lemma, for some $c_3,c_4>0$. It suffices to take $M>0$ small enough and take a union bound over the various nets we constructed. Throughout the paragraph, $c_1,c_2,c_3,c_4$ depend only on $\xi$.

Exactly the same proof applies to the vector $\theta_1X^*[1]_j+\theta_2X^*[2]_j-\theta_0(X^*[1]_j)_{[j]}e_j$ as we only modified its $j$-th coordinate and therefore this vector still solves the equation \eqref{righthanddsidebound} with an appropriate choice of constant.
\end{proof}
Therefore we can take the following reduction to the original problem:
\begin{lemma}\label{lemma3.9}
    Fix some $\alpha\in(0,1)$, some $\gamma>0$ and some $c>0$. There exists  constants $C_{\ref{lemma3.9}}>0$, $M_{\ref{lemma3.9}}>0$ and $c_{\ref{lemma3.9}}>0$ depending only on $\xi$ such that the following holds. Consider the following two probabilities 
    \begin{equation}\mathbb{P}_1:=\mathbb{P}^\mathcal{K}(\text{There exists a unit  vector $v$ of the form } \eqref{unitvectors} \text{with } D_{\alpha,\gamma}(v)\leq e^{cn})\end{equation}
        and
        \begin{equation}\begin{aligned}\label{firstdefinitionofp2}\mathbb{P}_2:&=\mathbb{P}^\mathcal{K}(\text{There exists some $\lambda\in[-C_{\ref{lemma3.9}}\sqrt{n},C_{\ref{lemma3.9}}\sqrt{n}]$} \\&\text{ and some unit vector $v$, $\|v\|_{[n]\setminus\{j\}}\geq M_{\ref{lemmadegenerate}}>0$ with $D_{\alpha,\gamma}(v)\leq e^{cn}$ such that}\\& (A^T-\lambda_1 I_n)_{/[j]} (A^T-\lambda_2 I_n)_{/[j]}v=\lambda (A^T-\lambda_1 I_n)_{[j]}.\\& \text{Denote by } w=(A^T-\lambda_2I_n)_{/[j]}v,\text{then for any $\theta_1,\theta_2$ such that $\|\theta_1v+\theta_2w\|=1$},\\&\text{we have } \|\theta_1v+\theta_2w\|_{[n]\setminus\{j\}}> M_{\ref{lemmadegenerate}}>0).   \end{aligned}\end{equation}
        Then we have
        $\mathbb{P}_1\leq\mathbb{P}_2+\exp(-c_{\ref{lemma3.9}}n)$.
\end{lemma}

\begin{proof}
   This is a restatement of the equality \eqref{productrules} and \eqref{righthanddsidebound}, where we rewrite $\lambda$ for the quantity $\theta_1(\lambda_2-\lambda_1)(X^*[1]_j)_{[j]}$ in \eqref{righthanddsidebound}. The range of $\lambda$ was computed in the proof of Lemma \ref{lemmadegenerate} under the event $\mathcal{K}$.
  The norm condition on $\|\theta_1v+\theta_2w\|_{[n]\setminus\{j\}}$ is also guaranteed by Lemma \ref{lemmadegenerate}, (2) because in our reduction process any linear sum of $u$ and $v$ is exactly of the form $\theta_1X^*[1]_j+\theta_2X^*[2]_j-\theta_0(X^*[1]_j)_{[j]}e_j$ for some $\theta_0,\theta_1,\theta_2\in\mathbb{R}$. We are working on the event $\Omega_\eqref{normalincomp}\cap \Omega_{\ref{lemmadegenerate}}$, which has probability $1-\exp(-\Omega(n)).$
\end{proof}

\subsubsection{Reduction via linearization}

Next we take a linearization procedure and reduce estimating the quadratic equations in probability $\mathbb{P}_2$ to that of a linear equation.

We shall take $j=1$ in the following construction without loss of generality, so we will no longer need to use the notations $_{[j]}$ and $_{/[j]}$.

\begin{Definition}
    For any $\hat{\lambda}\in\mathbb{R}$ consider a $(2n+1)\times (2n+1)$ matrix $P_A(\hat{\lambda})$ defined via
    \begin{equation}\label{palambda}
    P_A(\hat{\lambda})=\begin{bmatrix}
        0&0&0\\0&-\hat{\lambda}I_n&A_{/[1]}-\lambda_2 {I_n}_{/[1]}\\A_{[1]}-\lambda_1e_1&A_{/[1]}-\lambda_1 {I_n}_{/[1]}&0
    \end{bmatrix}.\end{equation}
    where $A_{/[1]}$ is an $n\times n$ matrix with i.i.d. entries of distribution $\xi$ but with its first row set identically zero, and ${I_n}_{/[1]}$ is the identity matrix with its first row set to be zero. The notation $A_{[1]}$ denotes the first column of $A_{/[1]}$, and $e_1$ is the unit coordinate vector in $\mathbb{R}^n$.

\end{Definition}

We can prove the following reduction from Lemma \ref{lemma3.9}:

\begin{lemma}\label{fact3.11} We can find a constant $c_{\ref{fact3.11}}>0$ depending only on $\xi$ such that
  $$  \mathbb{P}_2\leq \mathbb{P}_3+\mathbb{P}_4+\exp(-c_{\ref{fact3.11}}n),$$
where we further define $\mathbb{P}_3$ and $\mathbb{P}_4$ as follows: for any $c>0$,
\begin{equation}
   \mathbb{P}_3= \mathbb{P}^\mathcal{K}(\text{There exists a unit vector }v\in\mathbb{R}^n: D_{\alpha,\gamma}(v)\leq e^{cn}, (A_{/[1]}-\lambda_2 {I_n}_{/[1]})v=0)
\end{equation}
\begin{equation}\label{propp4}\begin{aligned}\mathbb{P}_4=\mathbb{P}^\mathcal{K}(&\text{There exists } \lambda\in[-C_{\ref{lemma3.9}}\sqrt{n},C_{\ref{lemma3.9}}\sqrt{n}],\hat{\lambda}\in[-8\sqrt{n},8\sqrt{n}]\\&\text{and exists unit vectors }v,w\in\mathbb{R}^n,D_{\alpha,\gamma}(v)\leq e^{cn}\\&\text{such that }(\lambda,w,v) \text{ is in the kernel of } P_A(\hat{\lambda})\text{ , and whenever $\theta_1,\theta_2\in\mathbb{R}$ satisfies }\\& \|\theta_1v+\theta_2w\|=1,\text{ then } \|\theta_1v+\theta_2w\|_{[n]\setminus\{1\}}\geq M_{\ref{lemmadegenerate}}>0).
    \end{aligned}
\end{equation}
Here we use the notation $(\lambda,w,v)$ to denote by a vector in $\mathbb{R}^{2n+1}$ whose first coordinate is $\lambda$, whose second to the $n+1$-th coordinates are given by the vector $w$ (in the order of the coordinates of $w$) and its $n+2$-th to $2n+1$-th coordinates are given by the vector $v$.
\end{lemma}

\begin{proof}    We first note a change of notation: we have used the matrix $A$ in the definition of $P_A(\hat{\lambda})$ but used $A^T$ in defining $\mathbb{P}_2$. We shall replace all $A^T$ in definition of $\mathbb{P}_2$ to the matrix $A$. This is minor change as the distribution of $A$ is not changed after taking a transpose.

Also note that the definition of $\mathbb{P}_2$ involves $j\in[n]$, but the value of $\mathbb{P}_2$ is the same for any $j\in[n]$. In defining $\mathbb{P}_3,\mathbb{P}_4$ we simply take $j=1$.

    Recall that $\mathbb{P}_2$ was defined in \eqref{unitvectors}. We first consider two separate cases: either $(A_{/[1]}-\lambda_2{I_n}_{/[1]})v=0$ or it is nonzero. In the first case we are reduced to the probability $\mathbb{P}_3$. In the second case, we can set $w=(A-\lambda_2 I_n)_{/[1]}v/\|(A-\lambda_2 I_n)_{/[1]}v\|$ and take $\hat{\lambda}=\|(A-\lambda_2 I_n)_{/[1]}v\|$. Then we can find some $\lambda\in\mathbb{R}$ such that $(\lambda,w,v)\in\operatorname{Ker}(P_A(\hat{\lambda}))$. 
    Finally, we use the operator norm bound on $A$ to deduce that we must take $\hat{\lambda}\in[-8\sqrt{n},8\sqrt{n}]$ on the event $\mathcal{K}$. The restriction on $\|\theta_1v+\theta_2w\|_{[n]\setminus\{1\}}$ comes from Lemma \ref{lemma3.9}.
\end{proof}

\subsection{Eliminating compressible vectors}
\label{sect3.3}

We first show that the kernel to the matrix $P_A(\hat{\lambda})$ has no compressible vector with very high probability, in the sense that if $(\lambda,w,v)\in \operatorname{Ker}(P_A(\hat{\lambda}))$ then neither $v$ nor $w$ can be a compressible vector.

\begin{Proposition}\label{proposition1199}
    There exist constants $\rho_\eqref{proposition1199},\delta_\eqref{proposition1199}>0$ and $c_\eqref{proposition1199}>0$ depending only on $\xi$ such that with probability $1-e^{-c_\eqref{proposition1199}n}$, the following statement is true: for any fixed $\lambda,\hat{\lambda}\in[-4\sqrt{n},4\sqrt{n}]$, any unit vector $v,w\in\mathbb{R}^n$ with $\|v\|_{[n]\setminus\{1\}},\|w\|_{[n]\setminus\{1\}}\geq M_{\ref{lemmadegenerate}}$ such that $(\lambda,w,v)$ is in the kernel of $P_A(\hat{\lambda})$, we must have $v\in\operatorname{Incomp}(\delta_\eqref{proposition1199},\rho_\eqref{proposition1199})$ and $w\in\operatorname{Incomp}(\delta_\eqref{proposition1199},\rho_\eqref{proposition1199})$.
\end{Proposition}

\begin{proof}
    We first fix the value $\lambda,\hat{\lambda}$ and fix some $\delta>0$ sufficiently small. First consider the vector $w$ (observe that the first coordinate of $w$ must be 0 as a solution to \eqref{palambda}), which solves the equation 
    $$
(A_{[1]}-\lambda_1 e_1)\lambda+(A_{/[1]}-\lambda_1 {I_n}_{/[1]})w=0,
    $$ Then we condition on the column $A_{[1]}$ and use the fact that the other coordinates of $w$ solve a linear equation with coefficients in $A_{[2,n]\times[2,n]}$. Then a standard argument via tensorization (see \cite{rudelson2008littlewood}) shows that $w_{[2,n]}$ cannot be a $(\delta,\rho)$- compressible vector with probability $1-e^{-cn}$ for some $\rho,\delta,c>0$ depending on $\xi$. This completes the proof.

    Next we prove that $v$ must be incompressible with high possibility: this is the harder part of the proof. Indeed, $v$ solves the quadratic equation    \begin{equation}\label{quadratic}\begin{aligned}(A-\lambda_1 I_n)_{/[1]} (A-\lambda_2 I_n)_{/[1]}v=\lambda (A-\lambda_1 I_n)_{[1]}.     \end{aligned}\end{equation} By assumption on $\|v\|_{[n]\setminus\{1\}}$, $v$ is not dominated by its first coordinate. Suppose $v_{[2,n]}$ is supported on only $\delta n$ coordinates, again we first condition on the first column of $A$. We may assume its support is on the last $\delta n$ coordinates and decompose $(A-\lambda_i I_n)_{/[1]}$ as $(A-\lambda_i I_n)_{/[1]}=\begin{bmatrix}E_i&F_i\\G_i&H_i\end{bmatrix}$ for both $i=1,2$, where $H_i$ is a square matrix of size $\delta n\times\delta n$. We use $A_{[1]}^1,A_{[1]}^2$ to denote the sub-vector formed by the first $n-\delta n$ (resp. the last $\delta n$) coordinates of $A_{[1]}$.
    Then the matrix equation solved by \eqref{quadratic}, projected onto the first $n-\delta n$ coordinate, yields that $$(E_1F_2+F_1H_2)v_{[2,n]}+E_1A_{[1]}^1v_{[1]}+F_1A_{[1]}^2v_{[1]}=\lambda(A_{[1]}^1-\lambda_1 e_1).$$
We take the following steps of conditioning:
\begin{enumerate} \item   We first condition on $A_{[1]}$ (which is the first column of $E_i$). \item As $F_2$ is a $(n-\delta n)\times\delta n$ rectangular matrix with first row zero and all the other entries are i.i.d., we may apply anti-concentration and tensorization to get that for some $c_9,c_{10}>0$ (that do not depend on $\delta\in(0,\frac{1}{2})$), we have $\|F_2v_{[2,n]}+A_{[1]}^1v_{[1]}\|\geq c_9\sqrt{n}M_{\ref{lemmadegenerate}}$ with probability at least $1-\exp(-c_{10}n)$. 
    
   We condition on this high-probability event of $F$. \item We condition on any realization of $H_1,H_2$. 
   \item Now we use the randomness in $E_1$ (without its first column) and the conditioning event in (2) to obtain that $$\mathcal{L}(E_1F_2v_{[2,n]}+E_1A_{[1]}^1V_{[1]},c_{11}n)\leq \exp(-c_{12}n)$$ for some $c_{11},c_{12}>0$ (that do not depend on $\delta\in(0,\frac{1}{2})$) via tensorization lemma \cite{rudelson2008littlewood}.
(To be more precise, we can write $E_1^0$ as $E_1$ with first column set zero, so that $E_1F_2v_{[2,n]}+E_1A_{[1]}^1V_{[1]}=E_1^0F_2v_{[2,n]}+E_1^0A_{[1]}^1V_{[1]}$ and $E_1^0$ is independent of all the other matrices. Then use tensorization). 
\end{enumerate}
Therefore, we have obtained that for this vector $v$ with $v_{[2,n]}$ supported on its last $\delta n$ coordinates, we have 
$$\mathbb{P}(\|(A-\lambda_1 I_n)_{/[1]} (A-\lambda_2 I_n)_{/[1]}v-\lambda (A-\lambda_1 I_n)_{[1]}\|\leq c_{11}n)\leq\exp(-c_{12}n).
$$
    Next, we can use a standard procedure to show that when $\delta>0$ is sufficiently small, then with possibility $1-\exp(-\Omega(n))$, \eqref{quadratic} cannot hold for any $\lambda$ and any $v$ supported on at most $\delta n$ indices. Via another standard procedure we can check that for some small $\delta>0$, any $(\delta,\rho)$-compressible vector cannot solve \eqref{quadratic} for any $\lambda$ with probability $1-\exp(-\Omega(n))$. The details are very similar to the proof of Proposition \ref{proposition2.15} or to \cite{rudelson2008littlewood} and is therefore omitted.
 \end{proof}

\begin{remark}
    We have used a somewhat complicated matrix $P_A(\hat{\lambda})$ \eqref{palambda} to show that $\theta_1X^*[1]_1+\theta_2X^*[2]_1$ is incompressible with high probability. This approach seems necessary to the author as we cannot find a simpler way. When we later try to prove that $\theta_1X^*[1]_1+\theta_2X^*[2]_1$ has a large LCD, it may be possible to avoid using the $2n+1$-dimensional matrix $P_A(\hat{\lambda})$ but simply use the following\begin{equation} 
\begin{pmatrix}
    0&A_{/[1]}-\lambda_2{I_n}_{/[1]}\\A_{/[1]}-\lambda_1{I_n}_{/[1]}&0
\end{pmatrix}. \end{equation} However, this reduction will lead to other technical difficulties (now we need to show the essential LCD of a vector pair is large, which is a bit complicated for several reasons; whereas for $P_A(\hat{\lambda})$ we only need to show the LCD of one component of the vector is large and a random generating method is applicable for this purpose) and does not make the proof much simpler, so we do not pursue it.
     \end{remark}

We also need a slight strengthening of Proposition \ref{proposition1199}:

\begin{lemma}\label{incompmore} For any $\lambda_1,\lambda_2\in\mathbb{R}$  we can find $c_{\ref{incompmore}}>0,\rho_{\ref{incompmore}}>0,\delta_{\ref{incompmore}}>0$ depending only on $\xi$ such that with probability at least $1-\exp(-c_{\ref{incompmore}}n)$ the following is true: 

For any $\lambda\in[-C_{\ref{lemma3.9}}\sqrt{n},C_{\ref{lemma3.9}}\sqrt{n}],\hat{\lambda}\in[-8\sqrt{n},8\sqrt{n}]$ there exists no solution $(\lambda,v,w)\in\mathbb{R}\times\mathbb{S}^{n-1}\times\mathbb{S}^{n-1}$ to $\ker(P_A(\hat{\lambda}))$ satisfying the assumption in $\mathbb{P}_4$ \eqref{propp4}
such that we have a decomposition $w=c_1v+c_2r$ for some $c_1,c_2\in\mathbb{R}$ and some unit vector $r\in\mathbb{S}^{n-1}$ such that $r$ is  $(\delta_{\ref{incompmore}},\rho_{\ref{incompmore}})$-compressible.
\end{lemma}

This lemma will be used when we subtract a certain copy of $v$ from $w$, and this lemma tells us the remaining part is also incompressible.
The proof is similar to Proposition \ref{proposition1199} and is sketched in Section \ref{appendix4a}.

\subsection{Random base vectors, inner product, and truncation}\label{sect3.4}

In this section we collect a few technical tools. The first result concerns the LCD of randomly generated vector pairs. The second is about the inner product between two random base vectors, and the third is a truncation procedure.
\subsubsection{Random base vectors}
\begin{lemma}\label{randmgenerationoflcd}
    Fix an integer $D>0$ and two subsets of nonnegative integers $D_1,D_2\subset [D]$ with $|D_1|=|D_2|=d\leq D$ and $D_1\cup D_2=[D]$. Fix two integers $N\geq N_1$ and a constant $\kappa>2$. We also fix two sequences $\sigma^1_1,\cdots,\sigma^1_{D}$ and $\sigma^2_1,\cdots,\sigma^2_{D}$ of nonnegative integers. Denote by $J^N=[-\kappa N,-N]\cup[N,\kappa N]$ and also denote by $J^{N_1}$ which is $J^N$ with $N$ replaced by $N_1$. For each $\ell\in\mathbb{N}_+$ we also denote by $I_\ell^N:=[-2^\ell N,2^\ell N]\setminus[-2^{\ell-1}N,2^{\ell-1}N]$ and denote by $I_\ell^{N_1}$ which is $I_\ell^N$ with $N$ replaced by $N_1$. Then we define the following two product sets $$\mathcal{B}_1:=\oplus_{j=1}^{D}\left(J^N1_{j\in D_1}+I^N_{\sigma_j^1}1_{j\notin D_1}\right),\quad \mathcal{B}_2:=\oplus_{j=1}^{D}\left(J^{N_1}1_{j\in D_2}+I^{N_1}_{\sigma_j^2}1_{j\notin D_2}\right).$$ We also impose a condition that there exists some $T_{\ref{randmgenerationoflcd}}>0$ such that for any $X\in\mathcal{B}_1$ we have $\|X\|_2\leq T_{\ref{randmgenerationoflcd}}\sqrt{D}N$ and for any $Y\in \mathcal{B}_2$ we have $\|Y\|_2\leq T_{\ref{randmgenerationoflcd}}\sqrt{D}N_1.$ 
    Let $\mathbb{P}_{X,Y}$ denote the probability measure where $X$ is chosen uniformly at random from $\mathcal{B}_1$ and $Y$ is chosen uniformly at random from $\mathcal{B}_2$ and independent of $X$.  Then we have, for any constant $\tau>0$, whenever $\alpha D\geq 4$ and $32T_{\ref{randmgenerationoflcd}}ND^{-1/2}/\alpha\leq 3^D$, we have the following estimate:
   \begin{equation}\label{assumptionmatrix}\begin{aligned}
        \mathbb{P}_{X,Y}&(\text{There exists $(\phi_1,\phi_2)\in\mathbb{R}^2: |\phi_1|\geq \tau(2T_{\ref{randmgenerationoflcd}}N)^{-1}$ or $|\phi_2|\geq \tau(2T_{\ref{randmgenerationoflcd}}N_1)^{-1}$}  \\&\text{ and $\|\phi\|\leq 16D^{-1/2}$ such that }\|\phi_1 X+\phi_2Y\|_\mathbb{T}\leq \sqrt{\alpha D})\leq (10^{12}(2+S)^2\alpha)^{d/4}, 
    \end{aligned}\end{equation} where $S=\frac{2T_{\ref{randmgenerationoflcd}}}{\tau\kappa}$
\end{lemma}

The proof of Lemma \ref{randmgenerationoflcd} is technical and deferred to Section \ref{appendix4a}. We briefly explain the auxiliary constants here: the constant $T_{\ref{randmgenerationoflcd}}>0$ will be fixed in advance as we generate boxes $\mathcal{B}_1,\mathcal{B}_2$ approximating unit vectors. Then $D_1,D_2$ are the subsets of coordinates where $v$ and $w$ take value in $[\kappa_0n^{-1/2},\kappa_1n^{-1/2}]$ for some $\kappa_0,\kappa_1$: this is ensured by the incompressibility condition. The constant $\tau>0$ will be used later as a restriction to the inner product of $X$ and $Y$: we need $X$ and $Y$ to be very close to being orthogonal when applying Littlewood-Offord theorems. Thus the only free parameter is $\alpha>0$, which we will set to be very small, so we get a super-exponential bound. Although this random generation method can show that $X$ and $Y$ are almost orthogonal with exponentially good probability, this is not strong enough and we will later need to restrict to the subset of vectors $X,Y$ that are close to orthogonal. In the final estimate, we only use the coordinates on which we have the interval $J^N$ and $J^{N_1}$, so the final upper bound does not depend on the constant $\sigma^1_j$'s.

In practical use we will choose $N_1$ at many different magnitudes that may be much smaller than $N$. This is the case when $|\langle v,w\rangle|\to 1$, where we will write $w=\alpha v+\beta r$ and generate a random vector pair $(v,r)$ from this lemma. As $|\beta|$ is very small, we only need a very loose net for $r$ (so that a small $N_1$) whereas a dense net (a large $N$) is needed for $v$.

\subsubsection{Inner product of random vectors}

We first define the notion of an angle for two high-dimensional vectors.
\begin{Definition}
    Fix an integer $d\in\mathbb{N}_+$. For two non-zero vectors $X,Y\in\mathbb{R}^d$ we define 
    \begin{equation}
        \cos(X,Y)=\frac{\langle X,Y\rangle}{\|X\|_2\|Y\|_2}
    \end{equation} where $\langle X,Y\rangle$ is the inner product in $\mathbb{R}^d$.
\end{Definition}
By the Cauchy-Schwarz inequality we have $|\cos(X,Y)|\leq 1$, and the closer this quantity is to 1, the stronger the overlap the two vectors have.

We verify that the random generation method of Lemma \ref{randmgenerationoflcd} typically produces two vectors with $|\cos(X,Y)|$ very close to 0.
Whereas we can show the possibility for the inner product to be bounded away from 0 is exponentially small, we only need the following weak quantitative estimate for our purpose.

\begin{fact}\label{newoldfact} Consider the random sampling of vectors $X$ and $Y$ in Lemma \ref{randmgenerationoflcd}. Then we have that 
$$
\mathbb{P}(|\operatorname{cos}(X,Y)|<32\kappa^2T_{\ref{randmgenerationoflcd}}^2 D^{-1})\geq \frac{3}{4}
.$$ 
    
\end{fact}

\begin{proof} We reuse the constant $T_{\ref{randmgenerationoflcd}}>0$ in Lemma \ref{randmgenerationoflcd} so that $\|X\|\leq T_{\ref{randmgenerationoflcd}}\sqrt{D}N$ and $\|Y\|\leq T_{\ref{randmgenerationoflcd}}\sqrt{D}N_1$.
We compute the second moment of the inner product thanks to independence of coordinates:
$$
\mathbb{E}\langle X,Y\rangle^2\leq\sum_{i\in D_1} \mathbb{E}|X_i|^2\mathbb{E}|Y_i|^2+\sum_{i\in D_2} \mathbb{E}|X_i|^2\mathbb{E}|Y_i|^2\leq 2 \kappa^2T_{\ref{randmgenerationoflcd}}^2N^2N_1^2D,
$$ but we have deterministically 
$$
\|X\|_2^2\|Y\|_2^2\geq N^2 N_1^2d^2\geq N^2N_1^2D^2/4,
$$
The claimed result then follows from a standard second moment computation.
\end{proof}

\subsubsection{Truncated matrix}

Next we introduce a truncation of the matrix $P_A(\hat{\lambda})$ which has Hilbert-Schmidt norm $O(n^2)$ almost surely. This will be useful when proving Lemma \ref{1712}.

Let $\xi'$ be an independent copy of $\xi$, and denote by $\tilde{\xi}=\xi-\xi'$ the symmetric difference. Denote by $I_B:=(1,16B^2),$ where we recall that $\xi\in\Gamma_B$. Denote by $p:=\mathbb{P}(\tilde{\xi}\in I_B)$, then we have $p\geq\frac{1}{2^7B^4}$ by Lemma \ref{lemma2.11}.

For some parameter $\nu\in(0,1)$ define $\xi_\nu$ via 
\begin{equation}\label{truncationrule}
\xi_\nu:=1\{|\tilde{\xi}\in I_B\}\tilde{\xi}Z_\nu,
\end{equation} where $Z_\nu$ is an independent Bernoulli variable with parameter $\nu$. 

\begin{Definition} Fix a $\nu\in(0, 2^{-8})$.
    We use the notation $\underline{A}$ to denote a random matrix with i.i.d. entry having distribution $\xi_\nu$, and we  denote by $\mathbb{P}_{\underline{A}}$ its probability measure. We also define the following matrix
     \begin{equation}\label{hilbertmas}M_{\underline{A}}=\begin{bmatrix}
0&0&0\\0&0&\underline{A}_{/[1]}\\\underline{A}_{[1]}&\underline{A}_{/[1]}&0
    \end{bmatrix}.\end{equation} 
    where as before, $\underline{A}_{/[1]}$ is $\underline{A}$ with its first row set to zero and $\underline{A}_{[1]}$ is the first column of $\underline{A}_{/[1]}$.
\end{Definition}

\subsection{Constructing nets based on inner product, and its cardinality}\label{sect3.5}

Thanks to Proposition \ref{proposition1199}, we have now ruled out compressible vectors in the kernel of $P_A(\hat{\lambda})$. We will now rule out incompressible vector pairs $(v,w)$ with intermediate LCD. For technical reasons, we will now work with the truncated matrix $M_{\underline{A}}$ \eqref{hilbertmas} and work with a certain notion of net which differs from the net constructed via LCD. More importantly, we will consider a stratification of $(v,w)\in\mathbb{R}^{2n}$ depending on the inner product $\langle v,w\rangle$, and then construct $\epsilon$- nets in a way tailor-made to this inner product.

First, we need to fix the support of the incompressible part of the vectors $v$ and $w$. \begin{notation}\label{1728notation}Thanks to Lemma \ref{incompmore} we may assume that in any decomposition $w=c_1v+c_2r$ where $\|r\|=1$ we must have both $v$ and $w$ are ($\delta_{\ref{incompmore}},\rho_{\ref{incompmore}}$)-incompressible. Then we can find some ${\kappa_0}_{\ref{incompmore}}$ and ${\kappa_1}_{\ref{incompmore}}$, some $d\geq C_{\ref{incompmore}}n$ and two subsets $D_1,D_2\subset[n]$ with $|D_1|=|D_2|=d$, such that (omitting the subscript of $\kappa_0,\kappa_1$) 
$$
(\kappa_0+\kappa_0/2)n^{-1/2}\leq |v_i|,|r_j|\leq (\kappa_1-\kappa_0/2)n^{-1/2},\quad \forall i\in D_1,j\in D_2.
$$
Throughout the rest of the chapter, we will use this value of $\kappa_0,\kappa_1$ and sometimes omit the subscript $\cdot_{\ref{incompmore}}$. This introduces a slight abuse of notation as $\kappa_0,\kappa_1$ here may have different values from the same notation used in Section \ref{section222}, but as they have the same function in the two distinct proofs this should not lead to any confusion.
\end{notation}

We then define a family of sets. Note that since we will be conditioning the first column of $A$, we take the inner product of $v$ and $w$ restricted to the subset $D\setminus\{1\}$.\begin{notation} Recall our assumption that $D_1\cup D_2=[D]\subset [n]$, so the inner product is over indices in $[2,D]$. We also recall that for any $n$-dimensional vector $v$, $v_{[2,D]}$ is the $D-1$-dimensional vector composed by the 2nd to the $D$-th entries of $v$ in this order.\end{notation}

\begin{Definition}\label{netsandlevels}
For two constants $\epsilon,\epsilon_1>0$ we define the following sets and notions:
\begin{enumerate}
    \item 
Subsets indicating the value of $v$ and $r$ on $D_1$ and $D_2$:
$$\begin{aligned}\mathcal{I}(D_1,D_2):=&\{(v,r)\in \mathbb{S}^{n-1}\times\mathbb{S}^{n-1}:\\&(\kappa_0+\kappa_0/2)n^{-1/2}\leq |v_i|,|r_j|\leq(\kappa_1-\kappa_0/2)n^{-1/2},\forall i\in D_1,\forall j\in D_2\}.\end{aligned}$$
$$\begin{aligned}\mathcal{I}'(D_1,D_2):=&\{(v,r)\in \mathbb{R}^{n+n}:\kappa_0n^{-1/2}\leq |v_i|,|r_j|\leq \kappa_1n^{-1/2},\forall i\in D_1,\forall j\in D_2\}.\end{aligned}$$

\item For this fixed $D_1,D_2,\epsilon\leq\epsilon_1$, and any $c\in\mathbb{R}$ define the following sets
$$\begin{aligned}
\mathcal{P}_{\epsilon_1}(c)&:=\{(v,w)\in\mathbb{R}^n\times\mathbb{R}^n:\|v\|_2=1,\frac{1}{2}\leq\|w\|_2\leq\frac{3}{2},\exists r\in\mathbb{S}^{n-1}\text{ such that }\\&w=cv+\epsilon_1r,\quad \langle v_{[2,D]},r_{[2,D]}\rangle=0,(v,r)\in\mathcal{I}(D_1,D_2)
\},\end{aligned}$$  
$$\begin{aligned}\mathcal{P}^0_\epsilon(c)&:=\{(v,w)\in\mathbb{R}^{n}\times\mathbb{R}^{n}:\|v\|_2=1,\frac{1}{2}\leq\|w\|_2\leq\frac{3}{2},\exists r\in\mathbb{S}^{n-1},\exists t\in\mathbb{R},|t|\leq\epsilon\\&\text{such that }w=cv+tr,\langle v_{[2,D]},r_{[2,D]}\rangle|=0,(v,r)\in\mathcal{I}(D_1,D_2)\}.\end{aligned}$$

\item Discrete approximation of $\mathcal{P}_{\epsilon_1}$ and for $\mathcal{P}_\epsilon^0$ at scale $\epsilon:$ when $\epsilon_1\geq \epsilon$,
$$\begin{aligned}
\Lambda_{\epsilon,\epsilon_1}(c):=&\{
(v,w)\in B_n(0,2)\times B_n(0,2): v\in 4\epsilon n^{-1/2}\cdot\mathbb{Z}^n,\\&\text{ there exists } r\in B_n(0,2)\cap \frac{4\epsilon}{ \epsilon_1}n^{-1/2}\cdot\mathbb{Z}^n\text { such that }
w=cv+\epsilon_1r,\\&(v,r)\in \mathcal{I}'(D_1,D_2), \text{ and that }|\cos (v_{[2,D]},r_{[2,D]})|\leq 0.01\}\end{aligned}
$$

$$
\Lambda_\epsilon^0:=\{(v,v)\in B_n(0,2)\times B_n(0,2):v\in 4\epsilon n^{-1/2}\cdot\mathbb{Z}^n,(v,v)\in\mathcal{I}'(D_1,D_2)\}.
$$
\item Threshold: given $L>0$, $\epsilon_1>0$, for $(v,w)\in\mathcal{P}_{\epsilon_1}$(c) or $(v,w)\in{\Lambda_{\epsilon,\epsilon_1}}(c)$ define
$$
\tau_{L,\epsilon_1}(v,w)(\lambda):=\sup\{t\in[0,1]:\mathbb{P}(\|M_{\underline{A}}(\lambda,w,v)\|_2\leq t\sqrt{n})\geq(4Lt^2/\epsilon_1)^{n-1}\},
$$
and for $(v,w)\in\mathcal{P}_{\epsilon}^0(c)$ or $(v,w)\in \Lambda_\epsilon^0$, define
$$
\tau_{L,0}(v,w)(\lambda):=\sup\{t\in[0,1]:\mathbb{P}(\|M_{\underline{A}}(\lambda,w,v)\|_2\leq t\sqrt{n})\geq(4Lt)^{n-1}\}.
$$
\item $\Lambda_{\epsilon,\epsilon_1}(c)$ stratified by threshold function: we define, for $\epsilon\leq\epsilon_1$, 
$$
\Sigma_{\epsilon,\epsilon_1}(c,\lambda):=\{(v,w)\in\mathcal{P}_{\epsilon_1}(c):\tau_{L,\epsilon_1}(v,w)(\lambda)\in [\epsilon,2\epsilon]\},
$$
$$
\Sigma_{\epsilon,0}(c,\lambda):=\{(v,w)\in\mathcal{P}_\epsilon^0(c)
:\tau_{L,0}(v,w)(\lambda)\in [\epsilon,2\epsilon]\}.$$
\item  Nets: for fixed $\epsilon_1\geq\epsilon$ and fixed $\lambda\in\mathbb{R}$, we define the following nets: 
\begin{equation}\label{secondmomemtcomp}\begin{aligned}
\mathcal{N}_{\epsilon,\epsilon_1}(c,\lambda)&:=\{(v,w)\in\Lambda_{\epsilon,\epsilon_1}(c):(\frac{L\epsilon^2}{\epsilon_1})^{
n-1}\leq \mathbb{P}((\|M_{\underline{A}}(\lambda,w,v)\|\leq 2H_{\ref{1712}}\epsilon\sqrt{n}),\\&\mathbb{P}^\mathcal{K}(\|P_A(\hat{\lambda})(\lambda,w,v)\|\leq \epsilon\sqrt{n})\leq (2^{10}\frac{LH_{\ref{1712}}^2\epsilon^2}{\epsilon_1})^{
n-1}\}. \end{aligned}
\end{equation}

\begin{equation}\begin{aligned}
    \mathcal{N}_{\epsilon,0}(c,\lambda):=\{(v,w)\in\Lambda_\epsilon^0:&(L\epsilon)^{n-1}\leq \mathbb{P}(\|M_{\underline{A}}(\lambda,w,v)\|\leq 8\epsilon\sqrt{n}),\\&
    \mathbb{P}^\mathcal{K}(\|P_A(\hat{\lambda})(\lambda,w,v)\|\leq \epsilon\sqrt{n})\leq (2^{10}L\epsilon)^{n-1}\}.
    \end{aligned}
\end{equation}
\end{enumerate}
\end{Definition}

The range of $\epsilon>0$ and $c\in\mathbb{R}$ such that the sets $\Sigma_{\epsilon,\epsilon_1}(c)$ and $\Lambda_{\epsilon,\epsilon_1}(c)$ are not empty is explicitly spelled out in Lemma \ref{1712}.

The constant $H_{\ref{1712}}>0$ will be given in Lemma \ref{1712}, 
Here in the definition of threshold and nets, the exponent of $\epsilon$ is $n-1$ instead of $n$ because the first row of the matrix $\underline{A}$ is 0.

    In part (2) we do not fix $\|w\|=1$ in defining $\mathcal{P}_{\epsilon_1}(c)$ but set a possible range for $\|w\|$. This is useful in proving Proposition \ref{propostiaggewgwg}.

\begin{Definition}\label{importantdefns}(Important remark on Definition \ref{netsandlevels})

In Definition \ref{netsandlevels} we have introduced two parameters $\epsilon,\epsilon_1$ and there could be an ambiguity about how to fix them. We propose the following algorithm:

Given a vector $(v,w)\in\mathbb{R}^{n+n}$ that satisfies the conditions in Notation \ref{netsandlevels}, then we can always find some $c,\epsilon_1$ such that $(v,w)\in\mathcal{P}_{\epsilon_1}(c)$. Then whenever $\epsilon_1\neq 0$, we evaluate the threshold function $\tau_{L,\epsilon_1}(v,w)(\lambda)$. There are two cases: \begin{enumerate}\item  Either $\tau_{L,\epsilon_1}(v,w)(\lambda)\leq 2 \epsilon_1$, then we have $(v,w)\in \Sigma_{\epsilon,
\epsilon_1}(c,\lambda)$ for some $\epsilon\leq\epsilon_1$. \item Or we have $\tau_{L,\epsilon_1}(v,w)(\lambda)\geq 2\epsilon_1$, then we must also have $\tau_{L,0}(v,w)(\lambda)\geq2\epsilon_1$. In this scenario let $\epsilon:=\tau_{L,0}(v,w)(\lambda)$, then we assign  $(v,w)\in\mathcal{P}_\epsilon^0(c)$, $(v,w)\in\Sigma_{\epsilon,0}(c,\lambda)$.\end{enumerate}
This explains why in (5), for the definition of $\Sigma_{\epsilon,\epsilon_1}(c,\lambda)$, the threshold value of $(v,w)$ is not exceeding $2\epsilon_1$. This also explains why for $\Sigma_{\epsilon,0}(c,\lambda)$ we only consider the threshold value in one single interval $[\epsilon,2\epsilon]$. From this clarification we get $(\cup_{\epsilon_1\geq\epsilon,c}\Sigma_{\epsilon,\epsilon_1}(c,\lambda))\cup(\cup_\epsilon\Sigma_{\epsilon,0}(c,\lambda))$ exhausts all the incompressible vectors we study.
\end{Definition}

\subsubsection{A few technical preparations}

 we need the following Fourier replacement lemma:

\begin{lemma}\label{lemmafs}
    For any $L>0$, $\epsilon_1>0$ and $t>0$ and any $(\lambda,v,w)\in\mathbb{R}^{2n+1}$ satisfying that 
    $$
\mathbb{P}(\|M_{\underline{A}}(\lambda,w,v)\|_2\leq t\sqrt{n})\leq (4L^2t^2/\epsilon_1)^{n-1},\quad \text{(resp. $(4Lt)^{n-1}$)},
    $$
    we must have, for any $\hat{\lambda}\in\mathbb{R}$,
    $$
\mathcal{L}\left(P_A(\hat{\lambda})(\lambda,w,v),t\sqrt{n}\right)\leq (50L^2t^2/\epsilon_1)^{n-1},\quad\text{(resp. $(50Lt)^{n-1}$)}.
    $$
\end{lemma}
This lemma permits us to extend the small ball probability of $M_{\underline{A}}$ to that of $P_A(\hat{\lambda})$ with error up to a constant. The proof of Lemma \ref{lemmafs} follows almost the same lines as the proof of Lemma \ref{lemmareplacement}. We defer its proof to Section \ref{prooflittlewoodofford}.

We then show that we can use an exponential number of boxes to cover the discrete approximation subsets $\Lambda_{\epsilon,\epsilon_1}(c)$. The boxes we will use is very similar to that in Definition \ref{definitionbox}, with a minor difference in its definition.

\begin{Definition}\label{definitionboxnew} Fix $D_1,D_2\subset[n]$ with $D_1\cup D_2=[D]$, some $\kappa\geq 2$ and integers $N\geq N_1$. An $(N,N_1,\kappa,D_1,D_2)$ box pair is defined as a pair of product sets $\mathcal{B}_1,\mathcal{B}_2$ of the form $$\mathcal{B}_1=B_1^1\times B_2^1\times\cdots\times B_{n}^1\subset\mathbb{Z}^{n},\quad \mathcal{B}_2=B_1^2\times B_2^2\times\cdots\times B_{n}^2\subset\mathbb{Z}^{n}$$ where we have
    \begin{enumerate}
    \item $|B_i^1|\geq N$ for each $i\in[n]$ and $|B_i^2|\geq N_1$ for each $i\in[n]$,\item $B_i^1=[-\kappa N,-N]\cup[N,\kappa N]$, $i\in D_1$, and  $B_i^2=[-\kappa N_1,-N_1]\cup[N_1,\kappa N_1]$, $i\in D_2$. \item $|\mathcal{B}_1|\leq(\kappa N)^{n}$ and $|\mathcal{B}_2|\leq(\kappa N_1)^{n}$.\end{enumerate} 
\end{Definition}

For future applications, we need to modify the definition of $D_1,D_2$ as follows:
\begin{Reduction}\label{assumption3.23}
Given a $(N,N_1,\kappa,D_1,D_2)$ box pair $\mathcal{B}_1,\mathcal{B}_2$, by its definition we can find $T_{\ref{randmgenerationoflcd}}>0$ such that for any $X\in\mathcal{B}_1$, $\|X_{[2,D]}\|\leq T_{\ref{randmgenerationoflcd}}\sqrt{D}N$ and for any $Y\in\mathcal{B}_1,$ $\|Y_{[2,D]}\|\leq T_{\ref{randmgenerationoflcd}}\sqrt{D}N_1$. Then it is easy to see that we can find a constant $Q_{\ref{assumption3.23}}=100T_{\ref{randmgenerationoflcd}}>0$ such that there are at most $\frac{D}{8}$ indices $i\in[2,D]$ such that $|X_i|\geq Q_{\ref{assumption3.23}}N$, and that there are at most $\frac{D}{8}$ indices $i\in[2,D]$ such that $|Y_j|\geq Q_{\ref{assumption3.23}}N_1$. We remove these at most $D/4$ indices from $[2,D]$, so that after removal we have $|X_i|\leq Q_{\ref{assumption3.23}}N$ and $|Y_i|\leq Q_{\ref{assumption3.23}}N_1$ for each $i\in[2,D]$ that was not removed and for each $X,Y$ from the box pair $\mathcal{B}_1,\mathcal{B}_2$. For notational simplicity, we shall still denote the remaining interval by $[2,D]$. In other words, we assume that $D_1,D_2$ are initially chosen such that this uniform bound is satisfied.
\end{Reduction}
Our covering result is stated as follows. Here $\kappa_0,\kappa_1$ refer to ${\kappa_0}_{\ref{incompmore}}$ and ${\kappa_1}_{\ref{incompmore}}$
\begin{lemma} \label{lemma3.23}
    The net $\Lambda_{\epsilon,\epsilon_1}(c)$ can be covered by a family $\mathcal{F}_{\epsilon,\epsilon_1}(c)$ of $(N,N_1,\kappa,D_1,D_2)$ box pairs, where we take $N={\kappa_0}/4\epsilon$ and $N_1={\kappa_0}\epsilon_1/4\epsilon$ and $\kappa=\kappa_{\ref{lemma3.23}}\geq\max({\kappa_1}/{\kappa_0},2^8{\kappa_0}^{-4})$: 
    \begin{equation}\label{productboxdef}
\Lambda_{\epsilon,\epsilon_1}(c)\subset\cup_{i\in\mathcal{F}_{\epsilon,\epsilon_1}(c)}((4\epsilon n^{-1/2})\mathcal{B}_1\times(4\epsilon n^{-1/2}/\epsilon_1)\mathcal{B}_2)_i^{0.01}
    \end{equation} in the sense that for any $(v,w)\in \Lambda_{\epsilon,\epsilon_1}(c)$ with $w=cv+\epsilon_1 r$ then we have $(v,r)\in (4\epsilon n^{-1/2})\mathcal{B}_1\times (4\epsilon n^{-1/2}/\epsilon_1)\mathcal{B}_2$ for one such box $\mathcal{B}_1\times\mathcal{B}_2$ in the family $\mathcal{F}_{\epsilon,
    \epsilon_1}(c)$ and that $|\cos( v_{[2,D},r_{[2,D]})|\leq 0.01$. The superscript $\cdot^{0.01}$ on the right hand side of \eqref{productboxdef} is the subset of the product box consisting of vector pairs $(X,Y)$ with $|\cos(X_{[2,D]},Y_{[2,D]})|\leq 0.01.$

    The index set has cardinality $|\mathcal{F}_{\epsilon,\epsilon_1}(c)|\leq \kappa_{\ref{lemma3.23}}^{2n}$
    and each product box has cardinality $|(\mathcal{B}_1\times\mathcal{B}_2)_i|\leq (16\kappa_{\ref{lemma3.23}}^2\kappa_0^2\epsilon_1/\epsilon^2)^n $ for each $i\in \mathcal{F}_{\epsilon,\epsilon_1}(c)$.

\end{lemma}
More precisely, we shall take $N=\lfloor \kappa_0/4\epsilon\rfloor$, but this difference is negligible and is omitted.
\begin{proof}
    This can be proven in essentially the same way as Lemma \ref{specified850}. Now we only need to construct boxes separately for $v\in 4\epsilon n^{-1/2}\cdot\mathbb{Z}^n$ and $r\in 4\epsilon/\epsilon_1 n^{-1/2}\cdot\mathbb{Z}^n$. The final restriction on $\cos(X_{[2,D]},Y_{[2,D]})$ is inherent in the definition of $\Lambda_{\epsilon,\epsilon_1}(c)$.
\end{proof}

\subsection{Cardinality of net and net verification}\label{sec3.6}

We first prove that $\mathcal{N}_{\epsilon,\epsilon_1}(c,\lambda)$ is an $\epsilon$- net for $\Sigma_{\epsilon,\epsilon_1}(c,\lambda)$. The proof is similar to that of Lemma \ref{approximation}.

\begin{lemma}\label{1712} We can find two constants $H=H_{\ref{1712}}>1$ and $\kappa_{\ref{1712}}>0$ depending only on ${\kappa_0}_{\ref{incompmore}},{\kappa_1}_{\ref{incompmore}}$ such that \begin{enumerate}
   \item Any $c\in\mathbb{R},\epsilon_1>0$ such that $\mathcal{P}_{\epsilon_1}(c)$ is not empty must satisfy that $\max(|c|^2+2,|\epsilon_1|^2+2)\leq (H_{\ref{1712}}-1)^2/256B^4$.
   \item Fix $\epsilon\in(0,\kappa_{\ref{1712}})$ and $\epsilon_1\geq\epsilon$. Then for any $(v,w)\in \Sigma_{\epsilon,\epsilon_1}(c,\lambda)$  we can find $(v',w')\in \mathcal{N}_{\epsilon,\epsilon_1}(c,\lambda)$ such that $\|(v,w)-(v',w')\|_\infty\leq 2H_{\ref{1712}}\epsilon n^{-1/2}$.\end{enumerate}
\end{lemma}

\begin{proof} In this proof only we use $M$ to denote the matrix $M_{\underline{A}}$: this will greatly simplify the notations. We will always take $\epsilon\leq{\kappa_0}_{\ref{incompmore}}/8$.

    For any $(v,w)\in\Sigma_{\epsilon,\epsilon'}(c,\lambda)$ with $w=cv+\epsilon_1r$, we associate two random vectors $t_1=(t_1^1,\cdots,t_{n}^1)$ and $t_2=(t_1^2,\cdots,t_{n}^2)$ with independent, mean 0 coordinates satisfying $|t_i^1|\leq 4\epsilon n^{-1/2}$, $|t_i^2|\leq 4\epsilon/\epsilon_1 n^{-1/2}$ and $v-t_1\in 4\epsilon n^{-1/2}\mathbb{Z}^n$, $r-t_2\in 4\epsilon /\epsilon_1 n^{-1/2}\mathbb{Z}^n$. Consider the random vectors $u_1:=v-t_1$ and $u_2:=r-t_2$. Then by definition, we have $(u_1,u_2)\in\mathcal{I}'(D_1,D_2)$. We show that (for some $H>0$ to be fixed later) we can find a realization of $u_1,u_2$ that satisfies 
    \begin{equation}\label{1stcriterion}
        \mathbb{P}(\|M(\lambda,cu_1+\epsilon_1u_2,u_1)\|_2\leq2H\epsilon\sqrt{n})\geq(L\epsilon^2/\epsilon_1)^{n-1},\text{ and }|\cos({u_1}_{[2,D]},{u_2}_{[2,D]})|\leq 0.01.
    \end{equation}
   We first check that the first condition holds with probability at least $\frac{1}{2}$. Indeed, let $\mathcal{E}$ be the event that $\mathcal{E}:=\{M:\|M(\lambda,w,v)\|_2\leq2\epsilon\sqrt{n}\}$. Then we claim that \begin{equation}\label{tobedetermined}\mathbb{E}_M[\mathbb{P}_{t_1,t_2}(\|M(0,ct_1+\epsilon_1 t_2,t_1)\|_2\geq2(H-1)\epsilon\sqrt{n})\mid\mathcal{E}]\leq 1/2.\end{equation} To check the claim, we exchange the order of expectation and get
$$\mathbb{E}_{t_1,t_2}\|M(0,ct_1+\epsilon_1t_2,t_1)\|_2^2\leq \sum_i(\mathbb{E}(t_i^1)^2+\mathbb{E}(ct_i^1+\epsilon_1t_i^2)^2)\sum_j M_{ij}^2\leq 256B^4(2+|c|^2)\epsilon^2n$$ where we use the fact the entries of $M_{\underline{A}}$ is bounded by $16B^2$. Recall that $c$ is chosen such that $w=cv+\epsilon_1r$ with both $r$ and $v$ being unit vectors and $\langle v_{[2,D]},r_{[2,D]}\rangle=0$ and $(v,r)\in \mathcal{I}'(D_1,D_2)$. Computing the norm of $w_{[2,D]}$ relative to $v_{[2,D]}$ and  $r_{[2,D]}$, we have by definition of $\mathcal{P}_{\epsilon_1}(c)$ that $\frac{3}{2}\geq \|w_{[2,D]}\|_2\geq \max(c\|v_{[2,D]}\|_2,\epsilon_1\|r_{[2,D]}\|_2)$, so that by incompressibility we can find some $H_{\ref{1712}}>0$ depending only on $\kappa_0,\kappa_1$ such that any possible $c$ should satisfy \begin{equation}\label{assumptiononcs}|c|^2+2\leq (H_{\ref{1712}}-1)^2/256B^4,\quad |\epsilon_1|^2+2\leq (H_{\ref{1712}}-1)^2/256B^4.\end{equation} Then \eqref{tobedetermined} follows from a second moment computation.

Now we can compute that
 $$\begin{aligned}
\mathbb{E}_{u_1,u_2}&\mathbb{P}_M(\|M(\lambda,cu_1+\epsilon_1u_2,u_1)\|_2\leq 2H\epsilon\sqrt{n})\geq(1/2)\mathbb{P}_M(\|M(\lambda,cv+\epsilon_1r,v)\|_2\leq2\epsilon\sqrt{n})\\&\geq(1/2)(16\epsilon^2 L/\epsilon_1)^{n-1}\end{aligned} $$
where the last inequality follows from the fact that $(v,w)\in\Sigma_{\epsilon,\epsilon_1}(c,\lambda)
$ and definition of the threshold function. This proves the existence of $u_1,u_2$ satisfying the first criterion of \eqref{1stcriterion}. The second criterion of \eqref{1stcriterion} can be satisfied almost surely whenever we pick $\epsilon>0$ smaller than some constants depending on $\kappa_0,\kappa_1$.

Finally we check that the last condition in the definition of the net \ref{secondmomemtcomp} is satisfied by any such choice of $u_1,u_2$: that is, we have for any $|\hat{\lambda}|\leq 4\sqrt{n}$,
$$
\mathcal{L}_{A,op}(P_A(\hat{\lambda})(\lambda,cu_1+\epsilon_1u_2,u_1),\epsilon\sqrt{n})\leq (2^{10}L\epsilon^2H^2/\epsilon_1)^{n-1}.
$$
This can be checked from the fact that $\tau_{L,\epsilon_1}(v,w)(\lambda)\geq\epsilon,$ and then use Lemma \ref{lemmafs}, together with the fact that $\|P_A(\hat{\lambda})\|_{op}\leq 8\sqrt{n}$ and the a.s. bound of $\|u_1\|_\infty,\|u_2\|_\infty$. We omit the details, but see proof of Lemma \ref{approximation} for a similar argument at this step.

We take $(v',w')=(u_1,cu_1+\epsilon_1u_2)$ for any $(u_1,u_2)$ satisfying \ref{1stcriterion}. We can readily check that $\|(v',w')-(v,w)\|_\infty\leq 2H\epsilon n^{-1/2}$.
This completes the proof.
\end{proof}

Now we compute the cardinality of $\mathcal{N}_{\epsilon,\epsilon_1}(c,\lambda)$ and $\mathcal{N}_{\epsilon,0}(c,\lambda).$ 

Starting from the following result, we shall always implicitly assume that $n$ is larger than some fixed constant depending only on ${\kappa_0}_{\ref{incompmore}},{\kappa_1}_{\ref{incompmore}}$ without explicit specification.

    \begin{Proposition}\label{cardinalityprop} For all  $L\geq R_{\ref{cardinalityprop}} \geq2$  and any $\epsilon>0$ with $\log \epsilon^{-1}\leq R'_{\ref{cardinalityprop}}(nL^{-8n/D})$,  we have the following bound on cardinality of nets:
    $$
|\mathcal{N}_{\epsilon,\epsilon_1}(c,\lambda)|\leq (\frac{C\epsilon_1}{L^2\epsilon^2})^{n-1}\epsilon_1/\epsilon^2,\quad |\mathcal{N}_{\epsilon,0}(c,\lambda)|\leq (\frac{C}{L^2\epsilon})^{n-1}\epsilon^{-1},
    $$ where $C>0$, $R_{\ref{cardinalityprop}},R'_{\ref{cardinalityprop}}>0$ are constants depending only on $\xi$ (and ${\kappa_0}_{\ref{incompmore}}$ and ${\kappa_1}_{\ref{incompmore}}$).
    \end{Proposition}

This result follows from the following first moment computation.

\begin{lemma}\label{lemma3.7whatsoever} There exists two constants $R_{\ref{lemma3.7whatsoever}}$ and $R'_{\ref{lemma3.7whatsoever}}$ depending only on $\xi,\kappa_0,\kappa_1$ such that the following is true for all   $L\geq R_{\ref{lemma3.7whatsoever}},$ all $ N\geq N_1\geq 2$ with $\log N\leq R'_{\ref{lemma3.7whatsoever}}(nL^{-8n/D})$ and $\kappa>2$. Let $\mathcal{B}_1$, $\mathcal{B}_2$ be a $(N,N_1,\kappa,D_1,D_2)$ box pair in Definition \ref{definitionboxnew} with $D_1\cup D_2=[D]$.   We denote by $\mathbb{P}^c_{X,Y}$ this conditioned probability where $X$ is chosen uniformly from $\mathcal{B}_1$ and $Y$ chosen uniformly from $\mathcal{B}_2$, conditioning on the event that $|cos(X_{[2,D]},Y_{[2,D]})|\leq 0.01.$ Then
    $$
\mathbb{P}_{X,Y}^c\left(\mathbb{P}_{\underline{A}}(\|M_{\underline{A}}(\lambda,cX+Y,X)\|_2\leq n)\geq (\frac{L}{NN_1})^{n-1}\right)\leq (\frac{16}{L})^{2n-2}.$$ 
\end{lemma}

 First, we define the set of typical vectors $T(\mathcal{B}_1,\mathcal{B}_2)$ via
\begin{equation}
T=T(\mathcal{B}_1,\mathcal{B}_2):=\{(X,Y)\in \mathcal{B}_1\times\mathcal{B}_2:D_\alpha(D^{-1/2}X_{[2,D]},D^{-1/2}Y_{[2,D]})\geq 16\}.
\end{equation}

We first prove that for $X,Y$ generated from the probability measure $\mathbb{P}_{X,Y}^c$, the probability that $(X,Y)\notin T$ can be made super-exponentially small.

\begin{lemma}\label{lem3.28}
    We have $\mathbb{P}_{X,Y}^c((X,Y)\notin T)\leq (C_{\ref{lem3.28}}\alpha)^{D/4}$ for some constant $C_{\ref{lem3.28}}>0$ depending only on ${\kappa_0}_{\ref{incompmore}}$ and ${\kappa_1}_{\ref{incompmore}}$.
\end{lemma}

\begin{proof} 

By Reduction \ref{lemma3.23} we can assume that $|X_i|\leq Q_{\ref{lemma3.23}}N,|Y_j|\leq Q_{\ref{lemma3.23}} N_1\forall i,j\in[2,D]$. Therefore if $|\theta_1|\leq (4Q_{\ref{lemma3.23}}N)^{-1}$ and $|\theta_2|\leq (4Q_{\ref{lemma3.23}}N_1)^{-1}$ we have $\|\theta_1X_{[2,D]}+\theta_2Y_{[2,D]}\|_\mathbb{T}=\|\theta_1X_{[2,D]}+\theta_2Y_{[2,D]}\|_2$. The case when $\theta_1$ or $\theta_2$ is larger than this threshold but $\|\theta\|\leq 16D^{-1/2}$ has been dealt with in Lemma \ref{randmgenerationoflcd}. Then by definition of LCD, the result (without conditioning) follows from applying Lemma \ref{randmgenerationoflcd}, where we take $\tau=T_{\ref{randmgenerationoflcd}}/2Q_{\ref{lemma3.23}}$ and we absorb everything into a new constant $C_{\ref{lem3.28}}$. The result with conditioning follows from Fact \ref{newoldfact} which shows the conditioned event has high probability whenever $n$ is large enough so that  $32\kappa^2T_{\ref{randmgenerationoflcd}}^2 (C_{\ref{incompmore}}n)^{-1}\leq 0.01$.

\end{proof}

Now we are ready to prove Lemma \ref{lemma3.7whatsoever}.

\begin{proof}[\proofname\ of Lemma \ref{lemma3.7whatsoever}] 
   We first evaluate the probability, for each $(X,Y)\in T$:
     $$
\mathbb{P}_{\underline{A}}(\|M_{\underline{A}}(\lambda,cX+Y,X)\|_2\leq n).$$For this we can use the inverse Littlewood-Offord theorem (Proposition \ref{proponewlittlewood}) for $\underline{A}$. We first condition on the realization of all except the $[2,D]$-th columns of  $\underline{A}$ and only use the randomness in the $[2,D]$-th columns. We denote the remaining part by $\underline{A}_{[2,D]}$. 

Throughout the proof we let $R>0$ be a constant changing from line to line that will be fixed at the very end and only depends on $\xi$ (hence also on $\kappa_0,\kappa_1$).
We claim that
\begin{equation}\label{chiefestimate}
    \sup_{p\in\mathbb{R}^{2n}}\mathbb{P}\left(\left\|\begin{bmatrix}
        \underline{A}_{[2,D]}X_{[2,D]}\\\underline{A}_{[2,D]}(cX_{[2,D]}+Y_{[2,D]})
    \end{bmatrix}-p\right\|_2\leq n\right)\leq (\frac{R}{NN_1})^{n-1}.
\end{equation}
To prove \eqref{chiefestimate}, we note that it can be implied by the following inequality
\begin{equation}
    \sup_{p\in\mathbb{R}^{2n}}\mathbb{P}\left(\left\|\begin{bmatrix}
        \underline{A}_{[2,D]}X_{[2,D]}\\\underline{A}_{[2,D]} Y_{[2,D]}
    \end{bmatrix}-p\right\|_2\leq 4(|c|+1)n\right)\leq (\frac{R}{NN_1})^{n-1}.
\end{equation}

Then via tensorization Lemma (Lemma \ref{Tensorization}, note that the first row of $\underline{A}$ is identically 0 so we only have exponent $n-1$), it suffices to prove the following estimate: denote by $\underline{a}_{2,i}$ the entries of the second row of $\underline{A}$, then 
\begin{equation}\label{onerowestimate}
\mathcal{L}\left(\begin{bmatrix}
    \sum_{i=2}^D \underline{a}_{2i}X_i,\\ \sum_{i=2}^D \underline{a}_{2i}Y_i
\end{bmatrix},4(|c|+1)\sqrt{n}\right)\leq \frac{R}{NN_1}
\end{equation}
for some $R>0$ which differs from the $R$ in \ref{chiefestimate} by multiplying a universal constant.

We now apply Proposition \ref{proponewlittlewood} to the vectors $N^{-1}X_{[2,D]}$ and $N^{-1}Y_{[2,D]}$. After multiplying by a constant depending on $\kappa_0,\kappa_1$ we may assume that the two vectors have norm 1 and $N_1/N$ each, and 
by construction we have $|\cos(X_{[2,D]},Y_{[2,D]})|\leq 0.01$. Also,  $D_\alpha(X_{[2,D]},Y_{[2,D]})\geq 16D^{-1/2}$ on $T$. We then multiply the entries of $\underline{A}$ by some $\xi$-dependent constant so that we can assume $\mathcal{L}(\underline{a}_{2i},1)\leq\frac{1}{2}$.

Then Proposition \ref{proponewlittlewood} implies that whenever $\log (R'N)\leq \alpha n$ for some $R'>0$ depending only on $\xi,\kappa_0,\kappa_1$, then \eqref{onerowestimate} can hold for some $R>0$ depending on $\xi,\kappa_0,\kappa_1$ and not on $\alpha$. We will fix the choice of $R>0$ at this place.
The condition $\log (R'N)\leq \alpha n$  is the condition guaranteeing the first term on the right hand side of \eqref{levytwoball} dominates the second term.
Therefore, we have proven that \eqref{chiefestimate} holds for some $R>0$ in this range of $N$. 

Finally, we can choose $\alpha>0$ sufficiently small such that by Lemma \ref{lem3.28}, $\mathbb{P}_{X,Y}^c((X,Y)\notin T)\leq (C_{\ref{lem3.28}}\alpha)^{D/4}= (\frac{16}{L})^{2n-2}$ for any $L\geq R$ where $R$ is the constant in \eqref{chiefestimate}. Then for $L>R$ we must have,  $$
\mathbb{P}_{X,Y}^c\left(\mathbb{P}_{\underline{A}}(\|M_{\underline{A}}(\lambda,cX+Y,X)\|_2\leq n)\geq (\frac{L}{NN_1})^{n-1}\right)\leq\mathbb{P}_{X,Y}^c((X,Y)\notin T)\leq  (\frac{16}{L})^{2n-2}.$$
This completes the proof of Lemma \ref{lemma3.7whatsoever}.

For the possible range of $N$, we combing $(C_{\ref{lem3.28}}\alpha)^{D/4}= (\frac{16}{L})^{2n-2}$ with $\log (R'N)\leq\alpha n$ to get that we may simply write the range of applicable $N$ to be $\log N\leq R'(nL^{-8n/D})$ for a different $R'>0$ depending only on $\xi,\kappa_0,\kappa_1$.

\end{proof}

Then the proof of Proposition \ref{cardinalityprop} is immediate.

\begin{proof}[\proofname\ of Proposition \ref{cardinalityprop}]
By Lemma \ref{lemma3.23}, the discrete set $\Lambda_{\epsilon,\epsilon_1}(c)$ admits a covering by $(N,N_1,\kappa,D_1,D_2)$, box pairs with at most $\kappa^{2n}$ boxes, and that each box has its cardinality at most $(16\kappa_{\ref{lemma3.23}}^2\kappa_0^2\epsilon_1/\epsilon^2)^n$ where we take $N=\kappa_0/4\epsilon$ and $N_1=\kappa_0\epsilon_1/4\epsilon$. Then multiplying the probability upper bound in Lemma \ref{lemma3.7whatsoever} by the cardinality of boxes gives the first result. The second result is proven in a similar way (and indeed much simpler), so we omit it.
The final numerical constants $R_{\ref{cardinalityprop}},R'_{\ref{cardinalityprop}}>0$ only differ from the constants $R_{\ref{lemma3.7whatsoever}}$ and $R'_{\ref{lemma3.7whatsoever}}$ in Lemma \ref{lemma3.7whatsoever} by multiplying or subtracting certain factors of $\kappa_0/4$. 
\end{proof}

\subsection{Taking everything together and proof completion}
\label{sec3.7}

Now we are ready to finish the proof of Theorem \ref{quasirandomtheorem3}. We will break the proof into pieces.

In this section we let $c_\Sigma\in(0,1)$ be a constant that only depends on $\xi$ and $\kappa_0,\kappa_1$ and its value will change from line to line. We always take \begin{equation}\label{rangetwolambdas}\hat{\lambda}\in[-4\sqrt{n},4\sqrt{n}], \text{ and  }\lambda\in [-C_{\ref{lemma3.9}}\sqrt{n},C_{\ref{lemma3.9}}\sqrt{n}]\end{equation} to be the range of $\hat{\lambda}$ and $\lambda$, without specifying them in each separate statement.

    \begin{lemma}\label{lemma3.122}
        For any fixed $c$ satisfying \ref{assumptiononcs}, any $\epsilon_1>\epsilon>0$, consider $$\mathcal{Q}_{\epsilon,\epsilon_1}(c):=\sup_{\lambda,\hat{\lambda}}\mathbb{P}_A^\mathcal{K}\left(\exists (v,w)\in\Sigma_{\epsilon,\epsilon_1}(c,\lambda):\| P_A(\hat{\lambda})(\lambda,w,v)\|_2\leq 2^{-n}\right).
$$
Then we can find some $L_{\ref{lemma3.122}}>0$ depending only on $\xi,\kappa_0,\kappa_1$such that for any $L\geq L_{\ref{lemma3.122}}$, we can find a constant $c_\Sigma>0$ depending further on $L$ satisfying  $
\mathcal{Q}_{\epsilon,\epsilon_1}(c)\leq 2^{-100n}$ for all $\epsilon\geq\exp(-c_\Sigma n)$.
    \end{lemma}
    
\begin{proof} 
    Let $(v',w')$ be the vector approximating $(v,w)$ in the sense of Lemma \ref{1712}. Then for $\epsilon\gg 2^{-n}$ we have that on $\mathcal{K}$, using that $\|A\|\leq4\sqrt{n}$, we have an inclusion
$$
  \{\|P_A(\hat{\lambda})(\lambda,w,v)\|_2\leq 2^{-n}\}\subset    \{\|P_A(\hat{\lambda})(\lambda,w',v')\|_2\leq 50H_{\ref{1712}}\epsilon\sqrt{n}\}.
 $$   Here we recall the fact that for an $n$-dimensional random variable $X$ and $r>t$ we have $\mathcal{L}(X,r)\leq (1+2r/t)^{n}\mathcal{L}(X,t)$, see for example \cite{campos2021singularity}, Fact 6.2.
    Therefore, we can use the net $\mathcal{N}_{\epsilon,\epsilon_1}(c,\lambda)$ to approximate $\Sigma_{\epsilon,\epsilon_1}(c,\lambda)$ and get, using Proposition 
    \ref{cardinalityprop} to bound the net cardinality, and using definition of the net $\mathcal{N}_{\epsilon,\epsilon_1}(c,\lambda)$,
    $$\begin{aligned}\mathcal{Q}_{\epsilon,\epsilon_1}(c)&\leq\sup_{\lambda}|\mathcal{N}_{\epsilon,\epsilon_1}(c,\lambda)|\sup_{(v,w)\in\mathcal{N}_{\epsilon,\epsilon_1}(c,\lambda)}\mathcal{L}_{A,op}(P_A(\hat{\lambda})(\lambda,w,v),50\epsilon H_{\ref{1712}}\sqrt{n})\\&\leq (C\epsilon_1/L^2\epsilon^2)^{n-1}\epsilon_1/\epsilon^2(2^{10}LH_{\ref{1712}}^2\epsilon^2/\epsilon_1)^{n-1}(1+100H_{\ref{1712}})^{n}\leq 2^{-100n}\end{aligned}$$ where the first inequality on the second line holds for all $L\geq R_{\ref{cardinalityprop}} \geq2$  and any $\log \epsilon^{-1}\leq R'_{\ref{cardinalityprop}}(nL^{-8n/D})$, and for the second inequality in the second line we set $L$ sufficiently large with respect to other absolute constants. The last inequality determines a range $
    \epsilon\geq\exp(-c_\Sigma n)$ for which the estimate is valid, and $c_\Sigma$ depends on $L$ but not on $
    \epsilon_1.$

\end{proof}

This lemma covers the case where $\epsilon_1>\epsilon$. There is a remaining case when $\epsilon_1\leq \epsilon$, where $v_{[2,D]}$ and $w_{[2,D]}$ are almost parallel. In this regime we also have:

\begin{lemma}\label{lemma3.123}  Fix a $c\in\mathbb{R}$ satisfying \ref{assumptiononcs}.  We can find $L_{\ref{lemma3.123}}>0$ depending only on $\xi,\kappa_0,\kappa_1$ such that for any $L>L_{\ref{lemma3.123}}$, we can find $c_\Sigma>0$ depending only on $\xi,\kappa_0,\kappa_1$ and $L$ such that the following estimate holds for all $\epsilon\geq \exp(-c_\Sigma n)$. If we define
     $$\mathcal{Q}_{\epsilon,0}(c):=\sup_{\lambda,\hat{\lambda}}\mathbb{P}_A^\mathcal{K}\left(\exists (v,w)\in\Sigma_{\epsilon,0}(c,\lambda):\| P_A(\hat{\lambda})(\lambda,w,v)\|_2\leq 2^{-n}\right),
$$
then we have  $
\mathcal{Q}_{\epsilon,0}(c)\leq 2^{-100n}$ for all such $\epsilon$.
\end{lemma}

\begin{proof}
This almost follows the proof of Lemma \ref{lemma3.122}. We need to prove an analogue of Lemma \ref{1712} which shows that $\mathcal{N}_{\epsilon,0}(c,\lambda)$ is a good $\epsilon$-net for $\Sigma_{\epsilon,0}(c,\lambda)$. The proof should follow the same line as Lemma \ref{1712} so we omit it.   
\end{proof}

At this point we are not going to get a joint probability estimate over all possible $c$ and $\epsilon_1$: this is not easy to get as the threshold function is hard to approximate. Instead, we no longer use the threshold function but replace them by the LCD of the vector $v$, which is our ultimate target. Thus we now start to construct another net via LCD.

Consider a new parameter $\alpha'>0$.
By Fact \ref{fact2.35} we can find some $\gamma'>0$ such that for any $(v,w)\in \mathcal{I}(D_1,D_2)$ we must have that \begin{equation}
\label{whatisgamma'?}D_{\alpha',\gamma'}(v)\geq (2\kappa_1)^{-1}\sqrt{n}.\end{equation} We shall fix the value of $\gamma'$ satisfying this estimate once $\alpha'$ is fixed.

Consider an integral lattice $G_{\epsilon,\epsilon_1}$ (depending on $\alpha'$ and $\epsilon_1$):
$$
G_{\epsilon,\epsilon_1}:=\{(\frac{p}{\|p\|_2},\frac{q}{\|q\|_2}):p\in \mathbb{Z}^n\cap B_n(0,\epsilon^{-1})\setminus\{0\},\quad q\in \sqrt{\alpha'}\mathbb{Z}^{n}\cap B_n(0,\frac{\epsilon}{\epsilon_1})\setminus\{0\}\}
$$
and we define a translated version of the net
$$
\widetilde{G}_{\epsilon,\epsilon_1}(c):=\{(p,cp+\epsilon_1q)\in\mathbb{R}^{2n}:(p,q)\in G_{\epsilon,\epsilon_1}\},$$ and define the subset of vectors parameterized by the LCD of its first component:
$$
\Sigma_{\epsilon,\epsilon_1}'(c):=\{(v,w)\in\mathcal{P}_{\epsilon_1}(c),D_{\alpha',\gamma'}(v)\in[(4\epsilon)^{-1},(2\epsilon)^{-1}]\}.
$$
In contrast to prior cases, we do not assume that $\epsilon\leq\epsilon_1$ here.

Then we can prove exactly as in Fact \ref{fact2.37} and Corollary \ref{netcorollarys} that

\begin{fact}\label{fact3.333} 
\begin{enumerate}
\item We have
$|G_{\epsilon,\epsilon_1}|=|\widetilde{G}_{\epsilon,\epsilon_1}(c)|\leq (\frac{K\epsilon_1}{\epsilon^2})^n(\alpha')^{-n/2}$
    for a universal constant  $K>0$. \item  For any $(v,w)\in\Sigma_{\epsilon,\epsilon_1}'$ we can find $(v',w'=cv'+\epsilon_1r')\in \widetilde{G}_{\epsilon,
    \epsilon_1}(c)$ so that $\|(v,w)-(v',w')\|_2\leq 8(|c|+1)\sqrt{\alpha'n}\epsilon$. \item Therefore, we can modify the subset $\widetilde{G}_{\epsilon,\epsilon_1}(c)$ to be a $16(|c|+1)\sqrt{\alpha'n}\epsilon$-net of $\Sigma_{\epsilon,\epsilon_1}'(c)$, which we denote by $\overline{G}_{\epsilon,\epsilon_1}(c)$. \end{enumerate}
\end{fact}

\begin{proof}
    We only check claim (1). In constructing $G_{\epsilon,\epsilon_1}$, the net needed for the $p$-component has cardinality $(\frac{K}{\epsilon})^n$ and the net needed for the $q$-component has cardinality $(\frac{K\epsilon_1}{\sqrt{\alpha'}\epsilon})^n$ for a universal constant $K>1$.
    
    Then claim (2) follows from the definition of LCD and from the construction that in the direction $q$ we have a $\sqrt{\alpha'}\epsilon$-greedy net. The claim (3) is standard.
\end{proof}

When $\epsilon_1\leq\epsilon$ we define $$\Sigma_{\epsilon,0}'(c):=\{(v,w)\in\mathcal{P}_\epsilon^0(c):D_{\alpha',\gamma'}(v)\in[(4\epsilon)^{-1},(2\epsilon)^{-1}]\},$$ 
$$
G_{\epsilon,0}:=\{(\frac{p}{\|p\|_2},\frac{q}{\|q\|_2}):p\in \mathbb{Z}^n\cap B_n(0,\epsilon^{-1})\setminus\{0\},\quad q\in\{\pm 1\}^n\}
$$
as well as the translated version
$
\widetilde{G}_{\epsilon,\epsilon_1}(c):=\{(v,cv+\epsilon_1q)\in\mathbb{R}^{2n}:(p,q)\in G_{\epsilon,\epsilon_1}\}.$  The we can similarly check that for any $\epsilon\leq Kn^{-1/2}$ for some $K>1$, we have:

\begin{fact}\label{net1ds} $\Sigma_{\epsilon,0}'(c)$ has a $16(|c|+1)\sqrt{\alpha'n}\epsilon$- net of cardinality $(\frac{K}{\epsilon})^n$, denoted by $\overline{G}_{\epsilon,0}(c)$.
\end{fact}

Now we transfer the high probability estimate to $\Sigma_{\epsilon,\epsilon_1}'(c)$ and $\Sigma_{\epsilon,0}'(c)$.

\begin{Proposition}\label{propostiaggewgwg}
 There is a choice of $\alpha'>0,\gamma'\in(0,1)$ and $c_\Sigma>0$ such that with probability at least $1-2^{-50n}$ the following statement is true: 
\begin{equation}\label{unitetheworld} \begin{aligned}   \mathbb{P}^\mathcal{K}&(
    \text{There exists }\lambda,\hat{\lambda},\text{and unit vectors } (v,w)\in\cup_{\epsilon_1\geq\exp(-c_\Sigma n),c}\mathcal{P}_{\epsilon_1}(c)\cup\mathcal{P}_{\exp(-c_\Sigma n)}^0(c)\\&\text{ with } D_{\alpha',\gamma'}(v)\leq\exp(c_\Sigma n)\text{ such that }\quad
    P_A(\hat{\lambda})(\lambda,w,v)=0)\leq 2^{-50n}.\end{aligned}\end{equation}

where the union is over $c$ satisfying \eqref{assumptiononcs} and $\lambda,\hat{\lambda}$ satisfying \eqref{rangetwolambdas}.\end{Proposition}

\begin{proof}
We choose $L>\max(1/4\kappa_{\ref{1712}},R_{\ref{cardinalityprop}})$ and initially fix a value of  
$c_\Sigma\leq 1$ so that any $\epsilon\geq\exp(-c_\Sigma n)$ would satisfy the restriction on $\epsilon$ in Lemma \ref{cardinalityprop} depending on this choice of $L$. Then $L$ is fixed but we will modify the value of $c_\Sigma$ later.  Throughout the proof we assume that $\epsilon\geq \exp(-c_\Sigma n)$. We take a dyadic decomposition of $\epsilon^{-1}\in[(2\kappa_1)^{-1}\sqrt{n},\exp(c_\Sigma n)]$.

We then fix a $2^{-n-5}n^{-1/2}$-net for each parameter $\lambda,\hat{\lambda},c$ and $\epsilon_1$ from the range specified in  \eqref{rangetwolambdas}, \eqref{assumptiononcs}. Suppose that there is a unit vector pair $(v,w)=(v,cv+\epsilon_1r)\in\ker P_A(\hat{\lambda})$, then we let $c',\lambda',\hat{\lambda}',\epsilon_1'$ be the closest element in the net approximating those constants, and denote by $(v',w')=(v,c'v+\epsilon_1'w)$,
then on the event $\mathcal{K}$ we have $\|P_A(\hat{\lambda}')(\lambda',w',v')\|\leq 2^{-n}$. By definition we have $(v',w')\in\mathcal{P}_{\epsilon_1'}(c')$ where we only need to immediately observe that $\frac{1}{2}\leq \|w'\|_2\leq\frac{3}{2}$. 
Next, we take a dyadic decomposition of $[(2\kappa_1)^{-1}\sqrt{n},\exp(c_\Sigma n)]$ to discretize the range of LCD of $v$, noting the bound \eqref{whatisgamma'?}.

Then we only need to prove that for any fixed $\epsilon_1,\epsilon,c,\lambda,\hat{\lambda}$ in the given range, we have\begin{equation}\label{onelevelbounds}
\mathcal{P}^\mathcal{K}(\exists (v,w)\in\Sigma_{\epsilon,\epsilon_1}'(c):\|P_A(\hat{\lambda})(\lambda,v,w)\|\leq 2^{-n})\leq 2^{-55n}.
\end{equation}

First consider the case $\epsilon\leq\epsilon_1$.
Taking a union bound in Lemma \ref{lemma3.122}, \ref{lemma3.123} imply
\begin{equation}
\mathcal{P}^\mathcal{K}(\exists (v,w)\in\cup_{\exp(-c_\Sigma n)\leq\epsilon'\leq\epsilon_1}\Sigma_{\epsilon',\epsilon_1}(c,\lambda):\|P_A(\hat{\lambda})(\lambda,w,v)\|\leq 2^{-n})\leq 2^{-60n}.
\end{equation}

\begin{equation}\label{secondprokaps}
\mathcal{P}^\mathcal{K}(\exists (v,w)\in\cup_{\exp(-c_\Sigma n)\leq\epsilon'}\Sigma_{\epsilon',0}(c,\lambda):\|P_A(\hat{\lambda})(\lambda,w,v)\|\leq 2^{-n})\leq 2^{-60n}.
\end{equation}

Then to prove \eqref{onelevelbounds} we only need to consider those $(v,w)$ with $\tau_{L,\epsilon_1}(v,w)(\lambda)\leq\epsilon$. (Recall that by Definition \ref{importantdefns}, the previous two bounds have covered all $(v,w)\in\mathcal{P}_{\epsilon_1}(c)$ with $\tau_{L,\epsilon_1}(v,w)(\lambda)\geq\epsilon$). That is, we only need to prove
\begin{equation}\label{onelevelbounds2nd}
\mathcal{P}^\mathcal{K}(\exists (v,w)\in\Sigma_{\epsilon,\epsilon_1}'(c)\mid \tau_{L,\epsilon_1}(v,w)(\lambda)\leq\epsilon:\|P_A(\hat{\lambda})(\lambda,w,v)\|\leq 2^{-n})\leq 2^{-60n}.
\end{equation}

The probability of \eqref{onelevelbounds2nd} is bounded from above by, thanks to Fact \ref{fact3.333}, 
\begin{equation}\label{lastinequality}
|\overline{G}_{\epsilon,\epsilon_1}(c)|(4L\alpha'\epsilon^2/\epsilon_1)^{n-1}\leq 2^{-60n}\end{equation}
 where we assume that $2^{-n}\leq \epsilon\sqrt{\alpha'n}$, and then the factor $(4L\alpha'\epsilon^2/\epsilon_1)^{n-1}$ comes from the fact that $\tau_{L,\epsilon_1}(v,w)(\lambda)\leq \epsilon$. The latter estimate holds for any $\epsilon\geq (\alpha')^{-1/2}\exp(-c_\Sigma n)$ where $c_\Sigma$ was previously determined. For the last inequality of \eqref{lastinequality} we take $\alpha'>0$ very small relative to the  constant $L$ and other universal constants. This further fixes $\gamma'>0$ by Fact \ref{fact2.35}. Finally, we further modify $c_\Sigma$ to be a constant $c_\Sigma'>0$ satisfying $\exp(-c_\Sigma' n)\geq \exp(-c_\Sigma n)(\alpha')^{-1/2}$ and $2^{-n}\leq \exp(-c_\Sigma'n)\sqrt{\alpha'n}$ for each $n$ sufficiently large and use this $c_\Sigma'$ for the $c_\Sigma$ in the Proposition statement.

In the remaining case $\epsilon\geq\epsilon_1$, we already have $\Sigma_{\epsilon,\epsilon_1}'(c)\subseteq\Sigma_{\epsilon,0}'(c)$. Then as we have proven \eqref{secondprokaps}, we see that it suffices to prove
\begin{equation}
\mathcal{P}^\mathcal{K}(\exists (v,w)\in\Sigma_{\epsilon,0}'(c)\mid \tau_{L,0}(v,w)(\lambda)\leq\epsilon:\|P_A(\hat{\lambda})(\lambda,w,v)\|\leq 2^{-n})\leq 2^{-60n}.\end{equation}This again follows from the condition $\tau_{L,0}(v,w)(\lambda)\leq\epsilon$ and cardinality of the net $\overline{G}_{\epsilon,0}(c)$ in Fact \ref{net1ds}. The computations are analogous to the previous case and omitted.

\end{proof}

Now Theorem \ref{quasirandomtheorem3} is almost immediate:
\begin{proof}[\proofname\ of Theorem \ref{quasirandomtheorem3}] The compressible vectors have been considered in Proposition \ref{incompmore} and \ref{proposition1199}. For incompressible vectors with fixed $D_1,D_2$ such that $D_1\cup D_2=[D]$, all such vector pairs $(v,w)$ are saturated by the union in the probability of \ref{unitetheworld} as we have taken a joint union over $\epsilon_1$ and $c$ (as we can always write $w=cv+\epsilon_1r$ for some $c,\epsilon_1$, and some $\|r\|=1$ satisfying $\langle v_{[2,D]},r_{[2,D]}\rangle=0$). Finally, it suffices to take a union bound for all possible subsets $D_1\subset[n],D_2\subset[n]$, which has a contributing factor $2^{2n}$. 
\end{proof}

\section{Two-point invertibility: the complex case}\label{secgwe444}

In this section we prove Theorem \ref{complextwoballbounds}, generalizing the main results of Section \ref{secgwe333} to complex i.i.d. matrices. The generalization is mostly straightforward, so we will omit many technical details in the proof and highlight where essential changes should be made.

We introduce some notations for the complex case. 
\begin{Definition}
For a vector $v=(v_1,\cdots,v_n)^T\in\mathbb{C}^n$ we denote by $$\hat{v}=(\Re(v_1),\cdots,\Re(v_n),\Im(v_1),\cdots,\Im(v_n))^T\in\mathbb{R}^{2n}$$ with $\Re (v_j),\Im(v_j)$ being the real and imaginary parts of $v_j$. Also, define $[v]\in\mathbb{R}^{2\times 2n}$ as 
$$
[v]:=\begin{bmatrix}\Re(v)^T&\Im(v)^T\\-\Im(v)^T&\Re(v)^T \end{bmatrix},
$$ and for another vector $w\in\mathbb{C}^n$ we use the notation $[v,w]$ to denote a $4\times 2n$ matrix: 
\begin{equation}\label{4dimensionalb}
[v,w]:=\begin{bmatrix}\Re(v)^T&\Im(v)^T\\-\Im(v)^T&\Re(v)^T \\
\Re(w)^T&\Im(w)^T\\-\Im(w)^T&\Re(w)^T 
\end{bmatrix}.
\end{equation}In the other direction, for a vector $X=(X_1,\cdots,X_{2n})^T\in\mathbb{R}^{2n}$ we denote by 
\begin{equation}\label{howdowedefinetau?}
    \tau(v)=(X_1+iX_{n+1},\cdots,X_n+iX_{2n})^T\in\mathbb{C}^n
\end{equation} the complex vector associated to $X$, so that $\tau(\hat{v})=v$ for all $v\in\mathbb{C}^n$.
\end{Definition}
We say a vector $v\in\mathbb{C}^n$ is $\delta$-sparse if $\operatorname{supp}(\hat{v})\leq 2\delta n$. A vector $v\in\mathbb{S}_\mathbb{C}^{n-1}$ is $(\delta,\rho)$-compressible if it is within Euclidean distance at most $\rho$ to a $\delta$-sparse vector. The complement of $(\delta,\rho)$-compressible vectors in $\mathbb{S}_\mathbb{C}^{n-1}$ are called $(\delta,\rho)$-incompressible vectors.

It is by now standard (see \cite{rudelson2008littlewood}) that the singular values to $\sigma_{min}(G-z_1I_n)$ and $\sigma_{min}(G-z_2I_n)$ are both $(\delta,\rho)$-incompressible with exponentially good probability, for any $z_1,z_2\in\mathbb{C}:|z_1|\leq 8\sqrt{n},|z_2|\leq 8\sqrt{n}$. More importantly, the no-gaps delocalization result \cite{rudelson2016no} continues to hold in the complex case so the conclusion of Corollary \ref{corollary3.22} is still valid. As a  consequence we can prove the following exact parallel to Proposition \ref{invertibilitydistance}: 
\begin{notation}
    In this section we always assume $z_1,z_2\in\mathbb{C}$ satisfies $|z_1|,|z_2|\leq 8\sqrt{n}$.
\end{notation}
\begin{Proposition}\label{prop4.2final}
    For each $i=1,2$ and $j\in[n]$, let $X[i]_j$ denote the $j$-th column of $G-z_i I_n$ and let $H[i]_j$ denote the linear span of all the columns of $G-z_i I_n$ in $\mathbb{C}^n$ except the $j$-th column. Then for $G$ be as in Theorem \ref{complextwoballbounds}, we can find $c>0,\vartheta>0$ depending only on $\xi$ such that, for any $\epsilon>0$, 
\begin{equation}\begin{aligned}
    \mathbb{P}&(\sigma_{min}(G-z_1 I_n)\leq\epsilon n^{-1/2},\sigma_{min}(G-z_2I_n)\leq \epsilon n^{-1/2})\\&\leq\frac{5}{n}\sum_{j=1}^n\mathbb{P}(\operatorname{dist}(X[1]_j,H[1]_j)\leq\epsilon/\vartheta,\operatorname{dist}(X[2]_j,H[2]_j)\leq\epsilon/\vartheta) +e^{-cn}.\end{aligned}
\end{equation}
\end{Proposition}

We let $X^*[1]_j$ denote a unit normal vector to $H[1]_j$ in $\mathbb{C}^n$, and $X^*[2]_j$ a unit normal vector to $H[2]_j$. In this section we let $\mathcal{K}$ denote the event $\{\|G\|\leq 8\sqrt{n}\}$, then we still have $\mathbb{P}(\mathcal{K})\geq 1-\exp(-\Omega(n))$ by Fact \ref{operatornormfact} and triangle inequality. We have the following analogue of Lemma \ref{realcomplexcorrelate}, showing that $X^*[1]_j$ and $X^*[2]_j$ are not colinear when $z_1\neq z_2$:

\begin{lemma}\label{compoverlap}
Take an orthogonal projection $X^*[2]_j=\alpha X^*[1]_j+\beta r$ where $X^*[1]_j$ and $r$ are orthogonal in $\mathbb{C}^n$ and $r\in\mathbb{C}^n$ has unit norm. Then we can find $C>0$ and $C'>0$ depending only on $\xi$ such that there is an event $\Omega_{\ref{compoverlap}}$ holding with probability $1-\exp(-C'n)$ so that $$|\beta|\geq C|z_1-z_2|/\sqrt{n}\quad\text{ on }\mathcal{K}\cap\Omega_{\ref{compoverlap}}.$$
\end{lemma}
The proof uses the fact that when $X^*[1]_j$ and $X^*[2]_j$ are incompressible then their mass on $[n]\setminus\{j\}$ are bounded away from 0. Then use the definition of a normal vector.

The main theorem of this section on arithmetic structures is given as follows:
\begin{theorem}\label{whatistheorem4.3?}
    We can find $\alpha>0$, $\gamma\in(0,1)$ and $c_\Sigma>0$, $c>0$ depending only on $\xi$ such that with probability at least $1-\exp(-cn)$, we have $D_{\alpha,\gamma}([X^*[1]_j,X^*[2]_j])\geq\exp(c_\Sigma n)$.
\end{theorem}
 $[X^*[1]_j,X^*[2]_j]$ (defined in \eqref{4dimensionalb}) is a $4\times 2n$ matrix with essential LCD defined in \eqref{whatissmalllcd?}.

As we proceed to proving Theorem \ref{whatistheorem4.3?}, we should first rule out compressible vectors in the linear span of $[X^*[1]_j]$ and $[X^*[2]_j]$. Let $a$ denote the $4\times n$ matrix $[X^*[1]_j,X^*[2]_j]$. Then for any $\theta\in\mathbb{R}^4$, $\theta\cdot a$ equals the image of $\zeta_1X^*[1]_j+\zeta_2X^*[2]_j$ under the mapping $\hat{\cdot}:\mathbb{C}^n\to\mathbb{R}^{2n}$ for $\zeta_1=\theta_1+i\theta_2$ and $\zeta_2=\theta_3+i\theta_4$. Then we only need to check that, with probability at least $1-\exp(-\Omega(n))$, any vector $\zeta_1X^*[1]_j+\zeta_2X^*[2]_j$ of unit norm, for any $\xi_1,\xi_2\in\mathbb{C}$, is $(\delta,\rho)$- incompressible. This can be proven in exactly the same way as in Proposition \ref{proposition1199} since $\zeta_1X^*[1]_j+\zeta_2X^*[2]_j$ solves a quadratic equation in $G$ similar to the one in \eqref{righthanddsidebound}. Then the results and proofs in Lemma \ref{lemmadegenerate} and Proposition \ref{proposition1199} carry over to this complex case with straightforward modifications. This will lead to the following result, where we take $j=1$ without loss of generality (the other case is nowhere different):

\begin{Proposition}\label{incompcomplex}
 Let $v\in\mathbb{C}^n$ be any unit vector such that $v=\xi_1X^*[1]_1+\xi_2X^*[2]_1$ for some $\xi_1,\xi_2\in\mathbb{C}$. Then with probability $1-\exp(-\Omega(n))$, we can find $\rho,\delta>0$ depending only on $\xi$ such that any such vector $v$ (and thus for any such $\xi_1,\xi_2$) should satisfy that $\hat{v}$ is $(\delta,\rho)$-incompressible and $w:=(G-z_2I_n)_{/[1]}v/\|(G-z_2I_n)_{/[1]}v\|$ also satisfies that $\hat{w}$ is $(\delta,\rho)$-incompressible whenever $(G-z_2I_n)_{/[1]}v\neq 0$. 
\end{Proposition}
By Lemma \ref{incpmpspreads}, incompressible vectors are spread. We claim the following: 

\begin{lemma}\label{lemma4.70}
We can always find two subsets $D_1,D_2\subset[2n]$ and constants ${\kappa_0}_{\ref{lemma4.70}},{\kappa_1}_{\ref{lemma4.70}}>0,Q_{\ref{lemma4.70}}>0$, $R_{\ref{lemma4.70}}>0$ depending only on $\xi$ such that  \begin{enumerate}
\item The length of intervals satisfy $|D_1|=|D_2|=d\geq R_{\ref{lemma4.70}} n$ and vectors $v,w$ satisfy
\begin{equation}\label{1strelationdiff}
{\kappa_0}_{\ref{lemma4.70}}n^{-1/2}\leq |\hat{v}_i|,|\hat{w}_j|\leq {\kappa_1}_{\ref{lemma4.70}} n^{-1/2},\quad \forall i\in D_1,j\in D_2.
\end{equation}
\item Each $D_1$ or $D_2$ is either a subset of $[2,n]$ or a subset of $[n+2,2n]$. That is, we discard either $D_1\cap[1,n]$ or $D_2\cap[n+1,2n]$ and discard either $D_2\cap[1,n]$ or $D_2\cap[n+1,2n]$. We also remove the 1st and $n+1$-st coordinate from $D_1$ and $D_2$.
\item Denoting by $\widetilde{D}_1=2n-D_1:=\{2n-z:z\in D_1\}$ and $\widetilde{D}_2=2n-D_2$, then we further assume that 
\begin{equation}\label{2ndrelationdiff}
|\hat{v}_i|,|\hat{w}_j|\leq Q_{\ref{lemma4.70}}n^{-1/2}\quad \forall i,j\in D_1\cup D_2\cup\widetilde{D}_1\cup\widetilde{D}_2.
\end{equation}
\end{enumerate}
\end{lemma}
\begin{proof}To check that all these claims can follow all at once, we use Lemma \ref{incpmpspreads} to get point (1). Then we replace $D_1$ by the one from $D_1\cap[2,n]$ or $D_1\cap[n+2,2n]$ which has larger cardinality to get point (2). Finally we use similar ideas as in Reduction \ref{assumption3.23} to remove certain subsets of $D_1,D_2$ and get point (3).\end{proof}

We generalize the reduction in Lemma \ref{fact3.11} to the complex setting, as follows: 

\begin{Proposition}\label{ratify} To prove Theorem \ref{whatistheorem4.3?}, it suffices to prove that there exists three constants $c_\Sigma>0$, $c>0$ and $C_{\ref{ratify}}>0$ depending only on $\xi$, such that for any fixed intervals $D_1,D_2\subset[2n]$ satisfying the assumptions in Lemma \ref{lemma4.70}, we must have
$$\begin{aligned}\mathbb{P}^\mathcal{K}(&\text{There exists } \lambda,\hat{\lambda}\in\mathbb{C}:|\lambda|\leq C_{
\ref{ratify}}\sqrt{n},|\hat{\lambda}|\leq 16\sqrt{n}\\&\text{and exists unit vectors }v,w\in\mathbb{C}^n,D_{\alpha,\gamma}(v)\leq \exp(c_\Sigma n)\\&\text{such that }(\lambda,w,v) \in\ker P_G(\hat{\lambda})\text{ and $v,w$ satisfy \eqref{1strelationdiff},\eqref{2ndrelationdiff}}\leq 2^{-10n},\end{aligned}$$
\end{Proposition} where $P_G(\hat{\lambda})$ is the following $(2n+1)\times(2n+1)$ complex matrix 
\begin{equation}
    P_G(\hat{\lambda})=\begin{bmatrix}
        0&0&0\\0&-\hat{\lambda}I_n&G_{/[1]}-z_2 {I_n}_{/[1]}\\G_{[1]}-z_1e_1&G_{/[1]}-z_1 {I_n}_{/[1]}&0
    \end{bmatrix}.\end{equation}
    where $G_{/[1]}$ is an $n\times n$ matrix with i.i.d. entries of distribution $\xi+i\xi'$ but with its first row set identically zero, and ${I_n}_{/[1]}$ is the identity matrix with its first row set to be zero. The notation $G_{[1]}$ denotes the first column of $G_{/[1]}$, and $e_1$ is the unit coordinate vector in $\mathbb{C}^n$.

\begin{proof}(Sketch)
    This follows from the same reasoning as in the proof of Lemma \ref{fact3.11}. Here we choose to fix the subsets $D_1,D_2$ in advance and the probability here corresponds to the $\mathbb{P}_4$ of  Lemma \ref{fact3.11}. The probability $\mathbb{P}_3$ there can be chosen to be exponentially small when $c_\Sigma>0$ is small enough, as it only concerns the LCD to the kernel of $G_{/[1]}-z_1{I_n}_{/[1]}$ and thus can be addressed by the previous work \cite{luh2018complex}. 
\end{proof}

We need a slight generalization of Definition \ref{definitionboxnew} to the complex case:

\begin{Definition}\label{twodimensionalfourbox} For two integers $N\geq N_1$, a constant $\kappa\geq2$ and two intervals $D_1,D_2\subset[2n]$ we define a notion of $2n$-dimensional $(N,N_1,\kappa,D_1,D_2)$ box pair, generalizing Definition \ref{definitionboxnew} in the following sense: we now define 
\begin{equation}
\mathcal{B}_1:=B_1^1\times\cdots\times B_{2n}^1\subset\mathbb{Z}^{2n},\quad \mathcal{B}_2:=B_1^2\times\cdots\times B_{2n}^2\subset\mathbb{Z}^{2n}
\end{equation}
and we make the following variants of conditions (1) to (3) in Definition \ref{definitionboxnew}. For condition (1) we now assume $|B_i^1|\geq N,|B_i^2|\geq N_1\forall i\in[2n]$. Condition (2) is again $B_i^1=[-\kappa N,-N]\cup[N,\kappa N]$, $i\in D_1$, and  $B_i^2=[-\kappa N_1,-N_1]\cup[N_1,\kappa N_1]$, $i\in D_2$. Condition (3) is now $|\mathcal{B}_1|\leq(\kappa N)^{2n}$ and $|\mathcal{B}_2|\leq(\kappa N_1)^{2n}$. Thanks to \eqref{2ndrelationdiff} we also assume that $$(4) \quad \max B_i^1\leq Q_{\ref{lemma4.70}}\kappa N/{\kappa_0}_{\ref{lemma4.70}} \text{ and } \max B_i^2\leq Q_{\ref{lemma4.70}}\kappa N_1/{\kappa_0}_{\ref{lemma4.70}}$$ for all $i\in D_1\cup D_2\cup \widetilde{D}_1\cup \widetilde{D}_2$, where $\max B_i^1$ is the largest element in $B_i^i\subset \mathbb{Z}$.
\end{Definition}
We now show that randomly generated vector pairs from such boxes have large LCD and small inner product with high probability. 

Let $Z$ denote the following $4\times 2n$ matrix with $X\in\mathcal{B}_1$ and $Y\in\mathcal{B}_2$:
$$Z:=
\begin{bmatrix}X_{[1,n]}&X_{[n+1,2n]}\\-X_{[n+1,2n]}& X_{[1,n]}\\Y_{[1,n]}&Y_{[n+1,2n]}\\-Y_{[n+1,2n]}& Y_{[1,n]}\end{bmatrix}
$$ where for any subset $D\subset[2n]$, we use $X_D$ to denote the subvector of $X$ restricted to $D$ and we use $Z_D$ to denote the submatrix of $Z$ with only columns indexed by $D$ remaining but the other columns removed.

\begin{Proposition}\label{prop4.61}
    Consider a $(N,N_1,\kappa,D_1,D_2)$ box pair in Definition \ref{twodimensionalfourbox}, with $X$ uniformly chosen from $\mathcal{B}_1$ and $Y$ independently and uniformly chosen from $\mathcal{B}_2$. Denote by \begin{equation}\label{dtildes}\widetilde{D}=D_1\cup D_2\cup \widetilde{D}_1\cup\widetilde{D}_2,\end{equation} and we keep using this notation of $\widetilde{D}$ in the rest of this section. Then we have,
    \begin{enumerate}
        \item For any $\tau\in(0,1)$ we can find $S_{\ref{prop4.61}}>0$ depending on $\tau$, ${\kappa_0}_{\ref{lemma4.70}},{\kappa_1}_{\ref{lemma4.70}}$ such that for any $\alpha>0$,
   \begin{equation}\begin{aligned}
        \mathbb{P}_{X,Y}&(\text{There exists $\phi\in\mathbb{R}^4: \|\phi_{[1,2]}\|\geq \tau(2N)^{-1}$ or $
        \|\phi_{[3,4]}\|\geq \tau(2N_1)^{-1}$},  \text{ $\|\phi\|\leq 16d^{-1/2}$} \\&\text{such that }\|\phi^T\cdot Z_{\widetilde{D}}\|_\mathbb{T}\leq \sqrt{\alpha d})\leq (S_{\ref{prop4.61}}\alpha)^{d/4}.
\end{aligned}\end{equation} 
        \item Let $X'$ denote the vector $[-X_{[n+1,2n]},X_{[1,n]}]$ and similarly define the vector $Y'$. Then we can find $T_{\ref{prop4.61}}>0$ depending only on ${\xi},{\kappa_0}_{\ref{lemma4.70}},{\kappa_1}_{\ref{lemma4.70}}$ such that we have 
    \begin{equation}\label{innerproductsbook}
\mathbb{P}_{X,Y}(\max(|\cos(X_{\widetilde{D}},Y_{\widetilde{D}})|,|\cos(X_{\widetilde{D}},Y'_{\widetilde{D}})|)\leq T_{\ref{prop4.61}}d^{-1/2})\geq\frac{3}{4}.
        \end{equation}
    \end{enumerate}
\end{Proposition}

\begin{proof}(Sketch) For part (1), we first condition on the value of $X_{[2n]\setminus D_1}$ and $Y_{[2n]\setminus D_2}$. Assume without loss of generality that $\|\phi_{[1,2]}\|\geq\tau(2N)^{-1}$, then we further condition on $Y_{D_2}$ and use randomness only from $X_{D_1}$. The assumption (2) after \ref{1strelationdiff} ensures that on columns with label in $D_1$, we only have randomness from $X_{D_1}$ entries. Then we follow the proof of Lemma \ref{randmgenerationoflcd} with no major change. For part (2), this can be proven similarly as Fact \ref{newoldfact}.
\end{proof}

We now define a truncated version of $P_G(\hat{\lambda})$:

\begin{Definition} Fix some $\nu\in(0, 2^{-8})$ and define the random variable $\xi_\nu$ as in \eqref{truncationrule}.
Denote by $\underline{G}$ a random matrix with entry distribution $\xi_\nu+i\xi_\nu'$ where $\xi_\nu'$ is an independent copy of $\xi_\nu$. We then define the following matrix which is a truncation of $P_G(\hat{\lambda})$:
     \begin{equation}\label{hilbertmascomplexs}M_{\underline{G}}=\begin{bmatrix}
0&0&0\\0&0&\underline{G}_{/[1]}\\\underline{G}_{[1]}&\underline{G}_{/[1]}&0
    \end{bmatrix}.\end{equation} 
\end{Definition}

\begin{Definition}    
For two complex vectors $v,w\in\mathbb{C}^n$ we define a notion of overlap over a subset $D\subset[2n]$ as follows:
\begin{equation}
    \operatorname{Ang}_D(v,w):=\max(|\cos(\hat{v}_D,\hat{w}_D)|,|\cos(\hat{v}_D,\hat{(iw)}_D)|,|\cos(\hat{(iv)}_D,\hat{w}_D)|,|\cos(\hat{(iv)}_D,\hat{(iw)}_D)|).
\end{equation} 
Also, for two real vectors $X,Y\in\mathbb{R}^{2n}$ we can define without ambiguity
\begin{equation}\operatorname{Ang}_D(X,Y):=\max(|\cos(X_D,Y_D)|,|\cos(X_D',Y_D)|,|\cos(X_D,Y_D')|,|\cos(X_D',Y'_D)|),
\end{equation} where $X':=[-X_{[n+1,2n]},X_{[1,n]}]$ and similarly define $Y'$. The two definitions are consistent because, if $\hat{v}=X$, then $\hat{iv}=X'.$
\end{Definition}

\begin{fact}
    
When $D=\widetilde{D}$ (see \eqref{dtildes}), then the maximum of the four terms in $\operatorname{Ang}_D(v,w)$ is exactly the maximum of the two terms in \eqref{innerproductsbook}. Also,  $\operatorname{Ang}_{\widetilde{D}}(v,w)=0$ if and only if $v_{\widetilde{D}\cap[1,n]}$ and $w_{\widetilde{D}\cap[1,n]}$ are orthogonal with respect to the complex inner product in $\mathbb{C}^{\widetilde{D}\cap[1,n]}$.\end{fact} 

\begin{proof}
    Both claims follow from the symmetry of $\widetilde{D}$ under the map $x\mapsto 2n+1-x,x\in[n]$, which yields that $|\cos(\hat{v}_D,(\hat{iw})_D|=|\cos((\hat{iv})_D,\hat{w}_D)|$. The second claim follows from the fact that $\operatorname{Ang}_{\widetilde{D}}(v,w)=0$ implies the real and complex part of the $\mathbb{C}$- inner product of $v_{\widetilde{D}\cap[1,n]}$ and $w_{\widetilde{D}\cap[1,n]}$ are both 0.
\end{proof}

Now we define a family of nets which are the complex analogue of Definition \ref{netsandlevels}.

\begin{Definition}\label{netsincomplexcase} Fix two intervals $D_1,D_2\subset[2n]$ and $\kappa_0={\kappa_0}_{\ref{lemma4.70}},\kappa_1={\kappa_1}_{\ref{lemma4.70}}$. For two constants $0<\epsilon\leq\epsilon_1$ define the following subsets
$$\begin{aligned}\mathcal{I}(D_1,D_2):=&\{(v,r)\in \mathbb{S}_\mathbb{C}^{n-1}\times\mathbb{S}_\mathbb{C}^{n-1}:\\&(\kappa_0+\kappa_0/2)n^{-1/2}\leq |\hat{v}_i|,|\hat{r}_j|\leq(\kappa_1-\kappa_0/2)n^{-1/2},\forall i\in D_1,\forall j\in D_2\}.\end{aligned}$$
$$\begin{aligned}\mathcal{I}'(D_1,D_2):=&\{(v,r)\in \mathbb{C}^{n+n}:\kappa_0n^{-1/2}\leq |\hat{v}_i|,|\hat{r}_j|\leq \kappa_1n^{-1/2},\forall i\in D_1,\forall j\in D_2\}.\end{aligned}$$

\item For this fixed $D_1,D_2,\epsilon\leq\epsilon_1$, and any $c\in\mathbb{C}$ define the following sets
$$\begin{aligned}
\mathcal{P}_{\epsilon_1}(c)&:=\{(v,w)\in\mathbb{C}^n\times\mathbb{C}^n:\|v\|_2=1,\frac{1}{2}\leq\|w\|_2\leq\frac{3}{2},\exists r\in\mathbb{S}_\mathbb{C}^{n-1}\text{ such that }\\&w=cv+\epsilon_1r,\quad \operatorname{Ang}_{\widetilde{D}}(v,r)=0,(v,r)\in\mathcal{I}(D_1,D_2)
\},\end{aligned}$$  
$$\begin{aligned}\mathcal{P}^0_\epsilon(c)&:=\{(v,w)\in\mathbb{C}^{n}\times\mathbb{C}^{n}:\|v\|_2=1,\frac{1}{2}\leq\|w\|_2\leq\frac{3}{2},\exists r\in\mathbb{S}_\mathbb{C}^{n-1},\exists t\in\mathbb{R},|t|\leq\epsilon\\&\text{such that }w=cv+tr,\operatorname{Ang}_{\widetilde{D}}(v,r)=0,(v,r)\in\mathcal{I}(D_1,D_2)\}.\end{aligned}$$

Then we define discrete approximations of $\mathcal{P}_{\epsilon_1}(c)$ and for $\mathcal{P}_\epsilon^0$ at scale $\epsilon:$ when $\epsilon_1\geq \epsilon$,
$$\begin{aligned}
\Lambda_{\epsilon,\epsilon_1}(c):=&\{
(v,w)\in B_n^\mathbb{C}(0,2)\times B_n^\mathbb{C}(0,2): v\in 4\epsilon n^{-1/2}(1+i)\cdot\mathbb{Z}^n,\\&\text{ there exists } r\in B_n^\mathbb{C}(0,2)\cap \frac{4\epsilon}{ \epsilon_1}n^{-1/2}(1+i)\cdot\mathbb{Z}^n\text { such that }
w=cv+\epsilon_1r,\\&(v,r)\in \mathcal{I}'(D_1,D_2), \text{ and that }\operatorname{Ang}_{\widetilde{D}}(v,r)\leq 0.01\},\end{aligned}
$$

$$
\Lambda_\epsilon^0:=\{(v,v)\in B_n^\mathbb{C}(0,2)\times B_n^\mathbb{C}(0,2):v\in 4\epsilon n^{-1/2}(1+i)\cdot\mathbb{Z}^n,(v,v)\in\mathcal{I}'(D_1,D_2)\}.
$$
Threshold: for given $L>0$, for $(v,w)\in\mathcal{P}_{\epsilon_1}$(c) or $(v,w)\in{\Lambda_{\epsilon,\epsilon_1}}(c)$ define
$$
\tau_{L,\epsilon_1}(v,w)(\lambda):=\sup\{t\in[0,1]:\mathbb{P}(\|M_{\underline{G}}(\lambda,w,v)\|_2\leq t\sqrt{n})\geq(4Lt^2/\epsilon_1)^{2n-2}\},
$$
and for $(v,w)\in\mathcal{P}_{\epsilon}^0(c)$ or $(v,w)\in \Lambda_\epsilon^0$, define
$$
\tau_{L,0}(v,w)(\lambda):=\sup\{t\in[0,1]:\mathbb{P}(\|M_{\underline{G}}(\lambda,w,v)\|_2\leq t\sqrt{n})\geq(4Lt)^{2n-2}\}.
$$
We define subsets of $\Lambda_{\epsilon,\epsilon_1}(c)$ stratified by threshold function: for $\epsilon\leq\epsilon_1$, 
$$
\Sigma_{\epsilon,\epsilon_1}(c,\lambda):=\{(v,w)\in\mathcal{P}_{\epsilon_1}(c):\tau_{L,\epsilon_1}(v,w)(\lambda)\in [\epsilon,2\epsilon]\},
$$
$$
\Sigma_{\epsilon,0}(c,\lambda):=\{(v,w)\in\mathcal{P}_\epsilon^0(c)
:\tau_{L,0}(v,w)(\lambda)\in [\epsilon,2\epsilon]\}.$$
Finally we define nets: for fixed $\epsilon_1\geq\epsilon$ and fixed $\lambda\in\mathbb{C}$, we define 
\begin{equation}\begin{aligned}
\mathcal{N}_{\epsilon,\epsilon_1}(c,\lambda)&:=\{(v,w)\in\Lambda_{\epsilon,\epsilon_1}(c):(\frac{L\epsilon^2}{\epsilon_1})^{
2n-2}\leq \mathbb{P}((\|M_{\underline{A}}(\lambda,w,v)\|\leq 2H_{\ref{complexnet}}\epsilon\sqrt{n}),\\&\mathbb{P}^\mathcal{K}(\|P_A(\hat{\lambda})(\lambda,w,v)\|\leq \epsilon\sqrt{n})\leq (2^{10}\frac{LH_{\ref{complexnet}}^2\epsilon^2}{\epsilon_1})^{
2n-2}\}. \end{aligned}
\end{equation}

\begin{equation}\begin{aligned}
    \mathcal{N}_{\epsilon,0}(c,\lambda):=\{(v,w)\in\Lambda_\epsilon^0:&(L\epsilon)^{2n-2}\leq \mathbb{P}(\|M_{\underline{A}}(\lambda,w,v)\|\leq 8\epsilon\sqrt{n}),\\&
    \mathbb{P}^\mathcal{K}(\|P_A(\hat{\lambda})(\lambda,w,v)\|\leq \epsilon\sqrt{n})\leq (2^{10}L\epsilon)^{2n-2}\}.
    \end{aligned}
\end{equation}
\end{Definition}

\begin{remark}
    We still assume that the algorithm in Definition \ref{importantdefns} is valid here, so that the algorithm determines whether $(v,w)$ lies in some $\Sigma_{\epsilon,\epsilon_1}(c,\lambda)$ or in $\Sigma_{\epsilon,0}(c,\lambda)$. 
\end{remark}

We have a version of Fourier replacement lemma similar to Lemma \ref{lemmafs}; 
\begin{lemma}\label{complexlemmafs}
    For any $L>0$, $\epsilon_1>0$ and $t>0$ and any $(\lambda,v,w)\in\mathbb{C}^{2n+1}$ satisfying that 
    $$
\mathbb{P}(\|M_{\underline{G}}(\lambda,w,v)\|_2\leq t\sqrt{n})\leq (4L^2t^2/\epsilon_1)^{2n-2},\quad \text{(resp. $(4Lt)^{2n-2}$)},
    $$
    we must have, for any $\hat{\lambda}\in\mathbb{C}$,
    $$
\mathcal{L}\left(P_G(\hat{\lambda})(\lambda,w,v),t\sqrt{n}\right)\leq (50L^2t^2/\epsilon_1)^{2n-2},\quad\text{(resp. $(50Lt)^{2n-2}$)}.
    $$ where for a complex random variable $\xi\in\mathbb{C}^n$, we define $\mathcal{L}(\xi,t):=\sup_{p\in\mathbb{C}^n}\mathbb{P}(\|\xi-p\|\leq t)$. This is an immediate generalization of Lévy concentration for real random variables.
\end{lemma}
This Lemma can be proven in the same way as Lemma \ref{lemmafs}. The difference here is we need to consider separately the real and imaginary parts of $M_{\underline{G}}$, $P_G(\hat{\lambda})$, so we first transfer the problem to $\mathbb{R}^{2n}$ and then do Fourier replacement there. Details are omitted.

We also have an analogue of Lemma \ref{lemma3.23} on the number of boxes covering $\Lambda_{\epsilon,\epsilon_1}(c)$.
\begin{fact}\label{factcomball} We can find a constant $\kappa_{\ref{factcomball}}>0$ depending only on ${\kappa_0}_{\ref{lemma4.70}},{\kappa_1}_{\ref{lemma4.70}}$ such that, for each $c\in\mathbb{C}$, with $N=\kappa_0/4\epsilon$ and $N_1=\kappa_0\epsilon_1/4\epsilon$,
    we can find a family $\mathcal{F}_{\epsilon,\epsilon_1}(c)$ of $2n$-dimensional $(N,N_1,\kappa_{\ref{factcomball}},D_1,D_2)$ box pairs whose union cover $\Lambda_{\epsilon,\epsilon_1}(c)$. The covering is in the sense that for any $(v,w=cv+\epsilon_1r)\in \Lambda_{\epsilon,\epsilon_1}(c)$, then $\hat{v}\in 4\epsilon n^{-1/2}\mathcal{B}_1$
    and $\hat{r}\in 4\epsilon n^{-1/2}/\epsilon_1\mathcal{B}_2$ for a certain box $\mathcal{B}_1,\mathcal{B}_2$ in the family $\mathcal{F}_{\epsilon,\epsilon_1}(c)$.
    The number of box pairs satisfies $|\mathcal{F}_{\epsilon,\epsilon_1}(c)|\leq \kappa_{\ref{factcomball}}^{4n}$ and each box has cardinality at most $(16\kappa_{\ref{factcomball}}^2\kappa_0^2\epsilon_1/\epsilon^2)^{2n}$.
\end{fact}
The proof is analogous to Lemma \ref{lemma3.23} and omitted. 

Next, we state an analogue of Lemma \ref{1712}, showing that $\mathcal{N}_{\epsilon,\epsilon_1}(c)$ is a net for $\Sigma_{\epsilon,\epsilon_1}(c)$:

\begin{lemma}\label{complexnet} There exists $H_{\ref{complexnet}}>0$ and $\kappa_{\ref{complexnet}}>0$ depending only on ${\kappa_0}_{\ref{lemma4.70}},{\kappa_1}_{\ref{lemma4.70}}$ such that 
\begin{enumerate}
    \item Any $c\in\mathbb{C}$, $\epsilon_1>0$ such that $\mathcal{P}_{\epsilon_1}(c)$ is nonempty must satisfy that $|c|^2+2\leq (H_{\ref{complexnet}}-1)^2/256B^4$ and $|\epsilon_1|^2+2\leq (H_{\ref{complexnet}}-1)^2/256B^4$. 
    \item Fix $\epsilon\in(0,\kappa_{\ref{complexnet}})$ and $\epsilon_1\geq\epsilon$. Then for any $(v,w)\in\Sigma_{\epsilon,\epsilon_1}(c,\lambda)$ we can find $(v',w')\in\mathcal{N}_{\epsilon,\epsilon_1}(c,\lambda)$ with $\|(v,w)-(v',w')\|_\infty\leq 2H_{\ref{complexnet}}\epsilon n^{-1/2}.$
\end{enumerate}
    
\end{lemma}The proof is the same as in Lemma \ref{1712}. For example, for part (1) we use orthogonality of $v_{\widetilde{D}\cap[1,n]}$ and $r_{\widetilde{D}\cap[1,n]}$, the fact that $\|(cv+\epsilon_1r)_{ \widetilde{D}\cap[1,n]}\|_2\leq \frac{3}{2}$ and that $(v,r)\in\mathcal{I}(D_1,D_2)$.

We will use a 4-dimensional version of Proposition \ref{proponewlittlewood} as follows: 
\begin{Proposition}\label{proponewlittlewoodcomplex}
Given a random vector $\xi=(\xi_1,\cdots,\xi_n)$ with i.i.d. coordinates $\xi_k$, where each $\xi_k$ has mean 0, variance 1, is subgaussian and satisfies $\mathcal{L}(\xi,1)\leq 1-p$ for some $p\in(0,1)$.

Consider vectors $\mathbf{c}=(c_1,\cdots,c_n),\mathbf{c}'=(c_1',\cdots,c_n')$, $\mathbf{d}=(d_1,\cdots,d_n),\mathbf{d}'=(d_1',\cdots,d_n')$ in $\mathbb{R}^n$ satisfying $\|\mathbf{c}\|_2=\|\mathbf{c}'\|_2=1$ and $\|\mathbf{d}\|_2=\|\mathbf{d}'\|_2=\omega_n$ for some $w_n\in(0,1]$. Ae assume that $\langle c,c'\rangle=\langle d,d'\rangle=0$ and $\max(|\cos(c,d)|,|\cos(c',d)|,|\cos(c,d')|,|\cos(c',d')|)\leq 0.01.$ Define a $4\times n$ matrix $a=(a_1,\cdots,a_n)$ where each $a_k=(c_k,c_k',d_k,d_k')^T. $

Consider the random sum $S=\sum_{k=1}^n a_k\xi_k$. Then for any $\alpha>0,\gamma\in(0,1)$ and 
$
\epsilon\geq\frac{2}{D_{\alpha,\gamma}(a)}
$ we have the following small ball probability bound for some universal constant $C>0$:
\begin{equation}\label{levytwoballcomp}
\mathcal{L}(S,2\epsilon)\leq(\omega_n)^{-2}(\frac{C\epsilon}{\gamma\sqrt{p}})^4 +C^2e^{-2p\alpha n}
\end{equation} 
\end{Proposition}
The proof is analogous to Proposition \ref{proponewlittlewood} and sketched in Section \ref{usedinsection5.2}.

We then prove the following result analogous to Lemma \ref{lemma3.7whatsoever}:
\begin{lemma}\label{lemma3.7whatsoevercomp} There exist $R_{\ref{lemma3.7whatsoevercomp}}>0$ and $R'_{\ref{lemma3.7whatsoevercomp}}>0$ depending only on ${\kappa_0}_{\eqref{lemma4.70}},{\kappa_1}_{\eqref{lemma4.70}}$ such that for all   $L\geq R_{\ref{lemma3.7whatsoevercomp}},$ all $ N\geq N_1\geq 2$ with $\log N\leq R'_{\ref{lemma3.7whatsoevercomp}}(nL^{-8n/D})$ and $\kappa>2$, the following statement holds. Let $\mathcal{B}_1$, $\mathcal{B}_2$ be a $(N,N_1,\kappa,D_1,D_2)$ box pair in Definition \ref{twodimensionalfourbox} with $D_1\cup D_2=[D]$.   We denote by $\mathbb{P}^c_{X,Y}$ this conditioned probability where $X$ is chosen uniformly from $\mathcal{B}_1$ and $Y$ chosen uniformly from $\mathcal{B}_2$, conditioning on the event that $\operatorname{Ang}_{\widetilde{D}}(X,Y)\leq 0.01.$ Then
    $$
\mathbb{P}_{X,Y}^c\left(\mathbb{P}_{\underline{G}}(\|M_{\underline{G}}(\lambda,c\tau(X)+\tau(Y),\tau(X))\|_2\leq n)\geq (\frac{L}{NN_1})^{2n-2}\right)\leq (\frac{16}{L})^{4n-4}.$$ 
\end{lemma} where we recall that $\tau(X)$ was defined in \eqref{howdowedefinetau?}.
\begin{proof}(Sketch) We only sketch where the proof differs from Lemma \ref{lemma3.7whatsoever}. Now we have proven Proposition \ref{prop4.61}, so we can prove an analogous result as in Lemma \ref{lem3.28}. That is, we can find a constant $C_{\ref{lemma3.7whatsoevercomp}}>0$ depending only on $\xi,\kappa_0,\kappa_1$ such that
$$
\mathbb{P}_{X,Y}(D_\alpha(X_{\widetilde{D}},X'_{\widetilde{D}},Y_{\widetilde{D}},Y'_{\widetilde{D}})\leq 16d^{-1/2})\leq (C_{\ref{lemma3.7whatsoevercomp}}\alpha)^{d/4}.
$$
We introduce a $(4n+2)\times (8n+4)$ matrix $\hat{M}_{\underline{G}}$ as follows:
$$
\hat{M}_{\underline{G}}=\begin{bmatrix}
0&0&0\\0&0&E_{\underline{G}}\\E_{\underline{G}}[1]&E_{\underline{G}}&0
    \end{bmatrix}
$$ where we set 
$$
E_{\underline{G}}:=\begin{bmatrix}
\Re\underline{G}_{/[1]}&-\Im\underline{G}_{/[1]}&0&0\\0&0&\Re\underline{G}_{/[1]}&-\Im\underline{G}_{/[1]}
\end{bmatrix}, E_{\underline{G}}[1]:=\begin{bmatrix}
\Re\underline{G}_{[1]}&-\Im\underline{G}_{[1]}&0&0\\0&0&\Re\underline{G}_{[1]}&-\Im\underline{G}_{[1]}
\end{bmatrix}.
$$ 

Recall that $c$ is any complex number, so we write $c=c_1+ic_2$ for $c_1,c_2\in\mathbb{R}$. We let $X_c$ denote the $2n$-dimensional vector whose first $n$-component is $c_1X_{[1,n]}-c_2X_{[n+1,2n]}$ and whose last $n$-component is $c_2X_{[1,n]}+c_1X_{[n+1,2n]}$. Then we define $X_c'$ which has first $n$ coordinates $-c_2X_{[1,n]}-c_1X_{[n+1,2n]}$ and its last $n$ coordinates $c_1X_{[1,n]}-c_2X_{[n+1,2n]}$.
Then we can immediately check that $$\|M_{\underline{G}}(\lambda,c\tau(X)+\tau(Y),\tau(X)\|_2=\|\hat{M}_{\underline{G}}(\Re\lambda,\Im\lambda,-\Im\lambda,\Re\lambda,X_c+Y,X_c'+Y',X,X')\|_2$$ where the latter is a $8n+4$-dimensional vector with its 5 to $2n+4$-th coordinate be that of $X_c+Y$, its $2n+5$ to $4n+4$-th coordinate be that of $X_c'+Y'$, and so forth.

    Since the entries of $G$ have i.i.d. real and complex part, then $\Re G$ and $\Im G$ are i.i.d., so we can now apply Littlewood-Offord theorem to $\hat{M}_{\underline{G}}$. For each application we take the $j+2,j+2+n,j+2+2n,j+2+3n$-th rows of $\hat{M}_G$ all at once for each $j\in[2,n]$, and by definition these four rows are independent from the rest of the matrix. We condition on the randomness that are not in the columns labeled by $(\widetilde{D}+4)\cup(\widetilde{D}+4n+4)=\{d+4,d+4n+4:d\in\widetilde{D}\}$ and apply Proposition \ref{proponewlittlewoodcomplex}. Our construction of $\widetilde{D}$ ensures that the randomness of $\hat{M}_{\underline{G}}$ from the columns indexed by $(\widetilde{D}+4)\cup(\widetilde{D}+4n+4)$ is independent of the randomness away from these columns. More precisely, let $\mathbf{\xi}$ be any $2n$-dimensional random vector which is a row of $[\Re \underline{G}_{/[1]},-\Im \underline{G}_{/[1]}]$. Then we need to estimate, for some $x>0$,
    \begin{equation}\label{firstsmallball}\mathcal{L}\left((\langle \xi,X\rangle,\langle \xi,X'\rangle,\langle \xi,X_c+Y\rangle,\langle \xi, X_c'+Y'\rangle)^T,x\right),\end{equation}
     which is the Lévy concentration function of a random vector in $\mathbb{R}^4$. From definition of $X_c$ we can easily check that this Lévy concentration function is bounded from above by    
\begin{equation}\label{secondsmallball}
    \mathcal{L}((\langle \xi,X\rangle,\langle \xi,X'\rangle,\langle \xi,Y\rangle,\langle \xi, Y'\rangle)^T,8(|c|+1)x).
    \end{equation}
    Then we apply Proposition \ref{proponewlittlewoodcomplex} to the vector pair $(X,X',Y,Y')$. The assumption on inner product of $c$ and $d$ in Proposition \ref{proponewlittlewoodcomplex} is verified by our assumption $\operatorname{Ang}_{\widetilde{D}}(X,Y)\leq 0.01$.
\end{proof}

From now on, we will take supremum over all \begin{equation}\label{rangetwolambda}\lambda\in\mathbb{C}:|\lambda|\leq C_{\ref{ratify}}\sqrt{n},\quad \hat{\lambda}\in\mathbb{C}:|\hat{\lambda}|\leq 16\sqrt{n},\end{equation} and we will not explicitly specify the range in each separate statement.

First, we can prove the following analogue of Lemma \ref{lemma3.122}, \ref{lemma3.123}.

\begin{lemma}\label{lemma4.177}  We can find some $L_{\ref{lemma3.122}}>0$ depending only on ${\kappa_0}_{\ref{lemma4.70}},{\kappa_1}_{\ref{lemma4.70}}$such that for any $L\geq L_{\ref{lemma3.122}}$, we can find a constant $c_\Sigma>0$ depending further on $L$ such that the following holds:
 Define for any $\epsilon\geq\epsilon_1$
 $$\mathcal{Q}_{\epsilon,\epsilon_1}(c):=\sup_{\lambda,\hat{\lambda}}\mathbb{P}_G^\mathcal{K}\left(\exists (v,w)\in\Sigma_{\epsilon,\epsilon_1}(c,\lambda):\| P_G(\hat{\lambda})(\lambda,v,w)\|_2\leq 2^{-n}\right),
$$ and define 
$$\mathcal{Q}_{\epsilon,0}(c):=\sup_{\lambda,\hat{\lambda}}\mathbb{P}_G^\mathcal{K}\left(\exists (v,w)\in\Sigma_{\epsilon,0}(c,\lambda):\| P_G(\hat{\lambda})(\lambda,v,w)\|_2\leq 2^{-n}\right).$$
Then  $
\mathcal{Q}_{\epsilon,\epsilon_1}(c)\leq 2^{-100n}$ and $\mathcal{Q}_{\epsilon,0}\leq 2^{-100n}$ for all $\epsilon\geq\exp(-c_\Sigma n)$.    
\end{lemma}

\begin{proof}(Sketch) Let $R,R'>0$ be two constants depending only on $\kappa_0,\kappa_1$.  We only need to prove a version of Proposition \ref{cardinalityprop}, which in our case states that whenever $L\geq R>2$ and $\log\epsilon^{-1}\leq R'(nL^{-8n/d})$, then for any $c,\lambda\in\mathbb{C}$, 
$$
|\mathcal{N}_{\epsilon,\epsilon_1}(c,\lambda)|\leq (\frac{C\epsilon_1}{L^2\epsilon^2})^{2n-2}\epsilon_1^2/\epsilon^4
$$ for a universal constant $C>0$. We shall also prove that a similar bound holds for $|\mathcal{N}_{\epsilon,0}(c,\lambda)|$. 

To prove this, we first fix a $(N,N_1,\kappa,D_1,D_2)$-box and use Lemma \ref{lemma3.7whatsoevercomp}. Then we count the number of boxes and the cardinality of each box via Fact \ref{factcomball}.
\end{proof}

Now we return to constructing a net for two vectors with small joint LCD, rather than for four vectors. We will use a new parameter $\alpha'>0$ which will finally be set sufficiently small. By \cite{luh2018complex}, Lemma 5.12 we can find a $\gamma'>0$ and $K'>0$ depending only on ${\kappa_0}_{\ref{lemma4.70}},{\kappa_1}_{\ref{lemma4.70}}$ such that any $v\in\mathcal{I}(D_1,D_2)$ must satisfy that $D_{\alpha',\gamma'}(\hat{v})\geq K'n^{1/2}.$

The following result is proven in \cite{luh2018complex}, Lemma 5.14:
\begin{lemma}\label{lemma4.18nets}
    The subset $\{v\in\mathbb{S}_\mathbb{C}^{n-1}: D_{\alpha',\gamma'}(\hat{v})\in[(2\epsilon)^{-1},\epsilon^{-1}]\}$ admits a $2\sqrt{\alpha'n}\epsilon$-net $\mathcal{D}_\epsilon$ with cardinality at most 
    $$
C\frac{(2\sqrt{\alpha'n}+2\epsilon^{-1})^2}{\alpha'n}(\frac{10}{n^{1/2}\epsilon})^{2n}
    $$ for any $\epsilon^{-1}\leq Kn^{1/2}$ and $\alpha'\leq K'^2$, where $C>0$ is a universal constant.
\end{lemma}

Consider the following subset $G_{\epsilon,\epsilon_1}$ (depending on $\alpha'$ and $c,\epsilon_1$):
$$
G_{\epsilon,\epsilon_1}:=\{p,\frac{q}{\|q\|_2}):p\in \mathcal{D}_{2\epsilon},\quad q\in \sqrt{\alpha'}(1+i)\mathbb{Z}^{n}\cap B_n^\mathbb{C}(0,\frac{\epsilon}{\epsilon_1})\setminus\{0\}\},
$$ with $\mathcal{D}_\epsilon$ from Lemma \ref{lemma4.18nets}
and we define a translated version of $G_{\epsilon,\epsilon_1}$:
$$
\widetilde{G}_{\epsilon,\epsilon_1}(c):=\{(p,cp+\epsilon_1q)\in\mathbb{C}^{n+n}:(p,q)\in G_{\epsilon,\epsilon_1}\},$$
and define the subset of vectors parameterized by the LCD of its first component:
$$
\Sigma_{\epsilon,\epsilon_1}'(c):=\{(v,w)\in\mathcal{P}_{\epsilon_1}(c),D_{\alpha',\gamma'}(v)\in[(4\epsilon)^{-1},(2\epsilon)^{-1}]\}.
$$

Then we can prove exactly as in Fact \ref{fact3.333} that

\begin{fact}\label{fact3.333comp} 
\begin{enumerate}
\item We have
$$|G_{\epsilon,\epsilon_1}|=|\widetilde{G}_{\epsilon,\epsilon_1}(c)|\leq \frac{K(\sqrt{\alpha'n}+\epsilon^{-1})^2}{\alpha'n}(\frac{K\epsilon_1}{\epsilon^2})^{2n}(\alpha')^{-n}$$
    for a universal constant  $K>0$. \item  For any $(v,w)\in\Sigma_{\epsilon,\epsilon_1}'$ we can find $(v',w'=cv'+\epsilon_1r')\in \widetilde{G}_{\epsilon,
    \epsilon_1}(c)$ so that $\|(v,w)-(v',w')\|_2\leq 8(|c|+1)\sqrt{\alpha'n}\epsilon$. \item Therefore, we can modify the subset $\widetilde{G}_{\epsilon,\epsilon_1}(c)$ to be a $16(|c|+1)\sqrt{\alpha'n}\epsilon$-net of $\Sigma_{\epsilon,\epsilon_1}'(c)$, which we denote by $\overline{G}_{\epsilon,\epsilon_1}(c)$. \end{enumerate}
\end{fact}

When $\epsilon_1\leq\epsilon$ we define $\Sigma_{\epsilon,0}'(c):=\{(v,w)\in\mathcal{P}_\epsilon^0(c):D_{\alpha',\gamma'}(v)\in[(4\epsilon)^{-1},(2\epsilon)^{-1}]\}$. Then 
we also have the following result, which is similar to Fact \ref{net1ds}:
\begin{fact} $\Sigma_{\epsilon,0}'(c)$ has a $16(|c|+1)\sqrt{\alpha'n}\epsilon$- net of cardinality $(\frac{K}{\epsilon})^{2n}$, denoted by $\overline{G}_{\epsilon,0}(c)$.
\end{fact}

Then we can prove the following proposition, which is analogous to Proposition \ref{propostiaggewgwg}:

\begin{Proposition}\label{propostiaggewgwgcomplex}
 There is a choice of $\alpha'>0,\gamma'\in(0,1)$ and $c_\Sigma>0$ such that with probability at least $1-2^{-50n}$ the following statement is true:
\begin{equation}\label{unitetheworldcomplex} \begin{aligned}   \mathbb{P}^\mathcal{K}&(
    \text{There exists }\lambda,\hat{\lambda},\text{and unit vectors } (v,w)\in\cup_{\epsilon_1\geq\exp(-c_\Sigma n),c}\mathcal{P}_{\epsilon_1}(c)\cup\mathcal{P}_{\exp(-c_\Sigma n)}^0(c)\\&\text{ with } D_{\alpha',\gamma'}(\hat{v})\leq\exp(c_\Sigma n)\text{ such that }\quad
    P_G(\hat{\lambda})(\lambda,v,w)=0)\leq 2^{-50n}.\end{aligned}\end{equation}

where the union is over $c,\epsilon_1$ satisfying Lemma \ref{complexnet} (1) and $\lambda,\hat{\lambda}$ satisfying \eqref{rangetwolambda}.
\end{Proposition}
The proof is exactly the same as that of Proposition \ref{propostiaggewgwg} and is omitted. The first step is to discretize the range of $\epsilon_1,c,\lambda,\hat{\lambda}$. Then we use the cardinality of the net $\overline{G}_{\epsilon,\epsilon_1}(c)$, the definition of $\Sigma_{\epsilon,\epsilon_1}(c,\lambda)$ and Lemma \ref{lemma4.177}. In the final step, we will set $\alpha'>0$ sufficiently small with respect to all other absolute constants.

Now we have almost proven Theorem \ref{whatistheorem4.3?}.

\begin{proof}[\proofname\ of Theorem \ref{whatistheorem4.3?}]
    Note that for any fixed $\epsilon>0$, the union $\cup_{\epsilon_1\geq\epsilon,c}\mathcal{P}_{\epsilon_1}(c)\cup \mathcal{P}_{\epsilon}^0(c)$ exhausts all unit vector pairs $(v,w)$ with $w=cv+\epsilon_1r$ and $(v,r)\in\mathcal{I}(D_1,D_2)$. Then we take a union bound with respect to all possible subsets $D_1,D_2$, which has cardinality $O(2^{4n})$. The final result follows from Proposition \ref{ratify}.
\end{proof}

Finally we complete the proof of Theorem \ref{complextwoballbounds}:
\begin{proof}[\proofname\ of Theorem \ref{complextwoballbounds}]
    Let $\mathbf{\xi}$ denote an $2n$-dimensional random vector with i.i.d. coordinates of distribution $\xi$. Let $v_1,v_2$ respectively denote the complex vectors $X^*[1]_j,X^*[2]_j$. Then by Proposition \ref{prop4.2final}, it suffices to estimate the following Lévy concentration function
    $$
\mathcal{L}\left((\langle \mathbf{\xi},\hat{v_1}\rangle,\langle \mathbf{\xi},\hat{iv_1}\rangle,\langle \mathbf{\xi},\hat{v_2}\rangle,\langle \mathbf{\xi},\hat{iv_2}\rangle)^T,\epsilon\right),
    $$
    By Theorem \ref{whatistheorem4.3?} we have proven that the pair of four $2n$-dimensional vectors satisfy $$D_{\alpha',\gamma'}(\hat{v_1},\hat{iv_1},\hat{v_2},\hat{iv_2})\geq\exp(c_\Sigma n)$$ with probability $1-\exp(-\Omega(n))$. Also in Lemma \ref{compoverlap} we have computed the overlap between $v_1$ and $v_2$. Then it suffices to apply the Littlewood-Offord theorem (Proposition \ref{proponewlittlewoodcomplex}) and conclude the proof. (More precisely, we can subtract a complex multiple of $v_1$ from $v_2$ and do a reduction similar to the reduction from \eqref{firstsmallball} to \eqref{secondsmallball}). The details are omitted.
\end{proof}

\section{Proof of auxiliary results}\label{prooflittlewoodofford}
This section consists the proof of a few auxiliary results.

\subsection{Fourier replacement principle}
We first give the proof of Lemma \ref{lemmareplacement}. For this proof only, the threshold function $\tau_L(v)$ is defined in equation \ref{tauells}.

\begin{fact}\label{fourierreplace}[\cite{campos2024least},Lemma V.1]
    For any $t\in\mathbb{R}$ and $\nu\leq 1/4$,
    $$
|\phi_\xi(t)|\leq |\phi_{\tilde{\xi}Z_\nu}(t)|
    $$ where $\phi_\xi$ (resp. $\phi_{\tilde{\xi}Z_\nu}$) are characteristic functions of random variables $\xi$ (resp. $\tilde{\xi}Z_\nu$).
\end{fact}

Then we prove the following fact:
\begin{fact}\label{newfacts}
    Fix $v\in\mathbb{S}^{n-2}\times\mathbb{S}^{n-1}$ and $t\geq\tau_L(v)$. Then 
    $$
\mathbb{E}\exp(-\pi \|Mv\|_2^2/2t^2)\leq (9Lt)^{2n-1}.
    $$
\end{fact}

\begin{proof} We bound
    \begin{equation}\begin{aligned}
      &  \mathbb{E}\exp(-\pi \|Mv\|_2^2/2t^2)\leq\mathbb{P}(\|Mv\|\leq t\sqrt{n})\\&+\sqrt{n}\int_t^\infty \exp(-s^2n/t^2)\mathbb{P}(\|Mv\|_2\leq s\sqrt{n})ds.\end{aligned}
    \end{equation} Using $t\geq\tau_L(v)$ and thus $\mathbb{P}(\|Mv\|_2\leq s\sqrt{n})\leq (4Ls)^{2n-1}$ for any $s\geq t$, the second term on the right hand side is bounded by 
    $$\begin{aligned}
 \sqrt{n}(4Lt)^{2n-1}\int_t^\infty \exp(-\frac{s^2n}{t^2})(\frac{s}{t})^{2n-1}ds.\end{aligned}
  $$
  After changing variable $u=s/t$, the right hand side shall be bounded by
  $$\begin{aligned}
t\sqrt{n}(4Lt)^{2n-1}&\int_1^\infty\exp(-u^2(2n-1))u^{n}du\\&\leq t\sqrt{n}(4Lt)^{2n-1}\int_1^\infty\exp(-u^2/2)du\leq (9Lt)^{2n-1}.
\end{aligned}  $$ 
\end{proof}

Next, we compare the characteristic function of $\mathcal{L}_Av$ and $Mv$. For any $v\in\mathbb{R}^{2n-1}$ denote respectively by $\psi_v$ and $\chi_{v}$ the characteristic functions of $\mathcal{L}_Av$ and $Mv$: for $x\in\mathbb{R}^{2n-1},$ 
$$
\psi_v(x)=\mathbb{E}_{\mathcal{L}_A}e^{2\pi i\langle \mathcal{L}_Av,x\rangle}=
\prod_{j\in[n-1],k\in [n,2n-1]}\phi_\xi(x_jv_k+x_kv_j),
$$
$$\begin{aligned}
\chi_v(x):=\mathbb{E}_{M}e^{2\pi i\langle Mv,x\rangle}&=\prod_{j\in[|D|],k\in[2n-1-|D|,2n-1]}\phi_{\tilde{\xi}Z_\nu}(x_jv_k+x_kv_j)\\&\prod_{j\in[|D|+1,n-1],k\in[n,2n-2-|D|]}\phi_{\tilde{\xi}Z_\nu}(x_jv_k+x_kv_j).
\end{aligned}$$
Then Fact \ref{fourierreplace} implies that 
$\psi_v(x)\leq \chi_v(x)$ for any $x\in\mathbb{R}^{2n-1}$.

Now we can complete the proof of Lemma \ref{lemmareplacement}.

\begin{proof}[\proofname\ of Lemma \ref{lemmareplacement}]
    We apply Markov's inequality with 
    \begin{equation}
\mathbb{P}(\|\mathcal{L}_Av-w\|_2\leq t\sqrt{n})\leq\exp(\pi n/2)\mathbb{E}\exp(-\pi\|\mathcal{L}_Av-w\|_2^2/2t^2).
    \end{equation} By Fourier inversion, we have
    \begin{equation}
        \mathbb{E}_{\mathcal{L}_A}\exp(-\pi \|\mathcal{L}_Av-w\|_2^2/(2t^2))=\int_{\mathbb{R}^{n}}e^{-\pi\|\xi\|_2^2}\cdot e^{-2\pi i t^{-1}\langle w,\xi\rangle}\psi_v(\xi/t)d\xi
    \end{equation}
    which is bounded by, using triangle inequality and non-negativity of $\chi_\nu$,
    $$
\leq \int_{\mathbb{R}^{2n-1}}e^{-\pi\|\xi\|_2^2}\chi_v(\xi/t)d\xi=\mathbb{E}_M\exp(-\pi \|Mv\|_2^2/2t^2).
    $$
Applying Fact \ref{newfacts} with $t\geq\tau_L(v)$ we obtain the desired bound.

\end{proof}

We then prove the replacement principle in Lemma \ref{lemmafs}.

\begin{proof}[\proofname\ of Lemma \ref{lemmafs}]
    First, we prove that 
    \begin{equation}\label{exposmall}
\mathbb{E}\exp(-\pi\|M_{\underline{A}}(\lambda,w,v)\|_2^2/2(t^4/\epsilon_1^2))\leq (9L^2t^2/\epsilon_1)^n.
    \end{equation} This follows from upper bounding the left hand side of \eqref{exposmall} by 
   $$
\mathbb{P}(\|M_{\underline{A}}(\lambda,w,v)\|_2\leq t^2\sqrt{n}/\epsilon_1)+\sqrt{n}\int_t^\infty e^{-s^4n/t^4}\mathbb{P}(\|M_{\underline{A}}(\lambda,w,v)\|_2\leq s^2\sqrt{n}/\epsilon_1\|)ds.
$$ Next we use
    $
\mathbb{P}(\|M_{\underline{A}}(\lambda,w,v)\|_2\leq s^2\sqrt{n}/\epsilon_1)\leq (4L^2s^2/\epsilon_1)^n
    $ for $s\geq t$, so we only need to bound 
    $$
\sqrt{n} (8L^2t^2/\epsilon_1)^n\int_t^\infty\exp(-s^4n/t^4)(s/t)^{2n}ds.
    $$Taking $u=s/t$, this integral is bounded by 
    $$
t\sqrt{n}(8L^2t^2/\epsilon_1)^n\leq (16L^2t^2/\epsilon_1)^n.
    $$ This justifies \eqref{exposmall}.

    Next we prove the replacement principle. We first introduce three characteristic functions: for $v,x\in\mathbb{R}^{2n+1}$ and $\nu\in(0,1)$, define
    $$
\psi_v(x):=\mathbb{E}_Ae^{2\pi i\langle P_A(\hat{\lambda})v,x\rangle}, \quad \chi_v(x):=\mathbb{E}_{\underline{A}}e^{2\pi i\langle M_{\underline{A}}v,x\rangle},\quad\phi_v(x):=\mathbb{E}_Ae^{2\pi i\langle P_0^0v,x\rangle}
    $$
where $P_A^0$ is the matrix obtained from $P_A(\hat{\lambda})$ by setting $\lambda_1=\lambda_2=\hat{\lambda}=0$. Then by Lemma we have $|\phi_v(x)|\leq\chi_v(x)$ for any $x$. By Markov,
$$
\mathbb{P}_A(\|P_A(\hat{\lambda})x-w\|\leq t^2\sqrt{n}/\epsilon_1)\leq \mathbb{E}_A\exp(\pi n/2)\exp(-\pi\|P_A(\hat{\lambda})x-w\|_2^2/2(t^4/\epsilon_1^2)).
$$ Then by Fourier inversion, we can find a vector $w'\in\mathbb{R}^{2n+1}$ depending on $w$, $x$ and $\lambda_1,\lambda_2$ (i.e., we extract all deterministic vectors into $w'$) such that
$$\begin{aligned}&\exp(-\pi\|P_A(\hat{\lambda})x-w\|_2^2/2(t^4/\epsilon_1^2))  =\int_{\mathbb{R}^{2n+1}}e^{-\pi\|\xi\|_2^2}e^{-2\pi it^{-2}\epsilon_1\langle w',\xi\rangle}\phi_x(t^{-2}\epsilon_1\xi)d\xi  \\&\leq \int_{\mathbb{R}^{2n+1}}e^{-\pi\|\xi\|_2^2}\chi_x(t^{-2}\epsilon_1\xi)d\xi=\mathbb{E}_{\underline{A}}\exp(-\pi \|M_{\underline{A}}x\|_2^2/2(t^4/\epsilon_1^2))
\end{aligned} $$   via triangle inequality and the non-negativity of $\chi_x$. Finally, using \eqref{exposmall} completes the proof.
\end{proof}

\subsection{A Littlewood-Offord theorem of stretched length}\label{usedinsection5.2}

Then we outline the proof of Proposition \ref{proponewlittlewood}. This will essentially be modifying \cite{rudelson2009smallest}, Theorem 3.3. The difference here is that we keep one vector of length 1 but scale the length of the other vector.

\begin{proof}[\proofname\ of Proposition \ref{proponewlittlewood}] We mostly follow the proof of \cite{rudelson2009smallest}, Theorem 3.3. Some changes will be made at proving equation \eqref{volumerecurrence}.

The proof begins with applying Esséen's Lemma (see \cite{tao2006additive}, p.290) which shows that
$$
\mathcal{L}(S,\epsilon\sqrt{2})\leq C\int_{B(0,\sqrt{2})}\prod_{k=1}^n\phi(\langle\theta,a_k\rangle/\epsilon)d\theta
$$ where $C>0$ is a universal constant and $\phi(t):=\mathbb{E}_\xi\exp(2\pi it\xi)$ is the characteristic function of $\xi$.  Let $\bar{\xi}:=\xi-\xi'$ where $\xi'$ is an independent copy of $\xi$. Then we have \begin{equation}\label{squares1}|\phi(t)|^2=\mathbb{E}\cos(2\pi t\bar{\xi})\end{equation}

We can also check that 
\begin{equation}\label{squares2}
1-\mathbb{E}\cos(2\pi t\bar{\xi})\geq 16p\cdot\mathbb{E}(\min_{q\in\mathbb{Z}}\|t\bar{\xi}-q\|^2\mid|\bar{\xi}|\geq 1).
\end{equation}
For any $z\in\mathbb{R}$ and $\theta\in\mathbb{R}^2$ we define $$f(\theta):=\min_{p\in\mathbb{Z}^n}\|\frac{z}{\epsilon}\theta\cdot a-p\|_2,$$ then using the inequality $|x|\leq\exp(-\frac{1}{2}(1-x^2))$ combined with \eqref{squares1}, \eqref{squares2}, we have
$$
\mathcal{L}(S,\epsilon\sqrt{2})\leq C\sup_{z\geq 1} \int_{B(0,\sqrt{2})}\exp(-8pf^2(\theta))d\theta.
$$
We define, for any $t>0$,
$$
I(t):=\{\theta\in B(0,\sqrt{2}):f(\theta)\leq t\}.
$$
We shall verify that for $t\leq\sqrt{\alpha n}/2$ we have
\begin{equation}\label{volumerecurrence}
\operatorname{Vol}(I(t))\leq (w_n)^{-1}(Ct\epsilon/\gamma)^2,\quad t<\sqrt{\alpha n}/2.
\end{equation}
    We first complete the proof of the theorem assuming the validity of \eqref{volumerecurrence}. We have 

    $$
\int_{B(0,\sqrt{2})\setminus I(\sqrt{\alpha n}/2)} \exp(-8pf^2(\theta))d\theta\leq C\exp(-2p\alpha n),
$$
$$\begin{aligned}
\int_{I(\sqrt{\alpha n}/2)}\exp(-8pf^2(\theta))d\theta&\leq \int_0^{\sqrt{\alpha n}/2}16pt\exp(-8pt^2)|\operatorname{Vol}(I(t))dt\\&\leq (w_n)^{-1}16p(C\epsilon/\gamma)^2\int_0^\infty t^3\exp(-8pt^2)dt\\&\leq (w_n)^{-1}(\frac{C'\epsilon}{\gamma\sqrt{p}})^2
\end{aligned}.$$
for a universal $C'>0$. Combining the previous estimates completes the proof of 
Proposition \ref{proponewlittlewood}.     Finally we check the validity of the volume estimate \eqref{volumerecurrence}. We fix $t<\sqrt{\alpha n}/2$. For two points $\theta',\theta''\in I(t)$ we find $p',p''\in\mathbb{Z}^2$ with 
    $$
\|\frac{z}{\epsilon}\theta'\cdot a-p'\|_2\leq t,\quad \|\frac{z}{\epsilon}\theta''\cdot a-p''\|_2\leq t.
    $$ We write $\tau:=\frac{z}{\epsilon}(\theta'-\theta'')$ and $p=p'-p''$. then 
    $$
\|\tau\cdot a-p\|_2\leq 2t.
    $$ By our assumption on the LCD of $\alpha$ we have either $\|\tau\|_2\geq\frac{\sqrt{2}}{\epsilon}$ or 
    $$
\|\tau\cdot a-p\|_2\geq \gamma\|\tau\cdot a\|_2. $$ In the latter case, we have $2t\geq\gamma\|\tau\cdot a\|_2$. We let $\tau=(\tau_1,\tau_2)$ denote the two coordinates of $\tau$. Then we must have, in this case, $\tau_1^2+(w_n)^{2}\tau_2^2+2w_n\tau_1\tau_2\langle c,d/\|d\|\rangle \leq 4t^2/\gamma^2$. The assumption on the inner product of $c,d$ implies $\tau_1^2+(w_n)^2\tau_2^2\leq 8t^2/\gamma^2$. Thus the range of possible values of $\tau$ is confined in an ellipse with axes length 1 and $1/w_n$.
Combining the above argument, we see that two points $\theta',\theta''\in I(t)$ must satisfy 
$$
\text{either } \|\theta'-\theta''\|_2\geq\frac{\sqrt{2}}{z}:=R\quad\text{ or } |(\theta'-\theta'')_1^2+(w_n)^{2}(\theta'-\theta'')_2^2|\leq\frac{8t^2\epsilon^2}{z^2\gamma^2}:=r^2.
$$ We first find a maximal subset in $B(0,\sqrt{2})$ that is $R$-separated, and this set has cardinality $O(\frac{3\sqrt{2}}{R})^2$ by a standard volumetric argument. Then around each point in this spanning subset, we inscribe an ellipse with axis length $Cr$ and $C(w_n)^{-1}r$. This ellipse has area $\pi C^2(w_n)^{-1}r^2$. Combining the above volumetric arguments yields \eqref{volumerecurrence}.
\end{proof}

\begin{proof}[\proofname\ of Proposition \ref{proponewlittlewoodcomplex}] We only need to replace all the $\mathbb{R}^2$ by $\mathbb{R}^4$ in the above proof. The difference here is that we need to show, for a similarly defined quantity $I(t)$, we have $$
\operatorname{Vol}(I(t))\leq(w_n)^{-2}(Ct\epsilon/\gamma)^4.
$$ To check this, set $\tau:=\frac{z}{\epsilon}(\theta'-\theta'')$, then using the conditions on $c$ and $d$ (in particular the small inner product between them) we get  
    $\|\tau\cdot a\|_2^2\geq \frac{1}{4}(\|\tau_{[1,2]}\|_2^2+(w_n)^2\|\tau_{[3,4]}\|_2^2)$. This, combined with $\|\tau\cdot a\|_2\leq 2t/\gamma$, allow us to estimate $\operatorname{Vol}(I(t))$ as in the previous case.
\end{proof}

\subsection{Proof of Corollary \ref{strongrepulsion} and \ref{corollary2comp}}

Here we outline the proof of Corollary \ref{strongrepulsion} and \ref{corollary2comp}, both of which use a very straightforward covering argument.

\begin{proof}[\proofname\ of Corollary \ref{strongrepulsion}] 
    For any $\epsilon>0$ we find an $\epsilon(4\sqrt{n}(1+a_n))^{-1}$-net ${N}_\epsilon$ of $[-4\sqrt{n},4\sqrt{n}]$. For any $\lambda_2\in[-4\sqrt{n},4\sqrt{n}]$ we find $\lambda^2\in N_\epsilon$ closest to $\lambda_2$. Then using $\|A\|\leq 4\sqrt{n}$ on $\mathcal{K}$, we see that if $\lambda^1$ is chosen to satisfy $|\lambda_1-\lambda^1|\leq \epsilon(4\sqrt{n})^{-1}$, then
    $$
\{\lambda_1,\lambda_2\text{ are eigenvalues of } A,\text{ on }\mathcal{K}\}\subset \{\sigma_{min}(A-\lambda^1 I_n)\leq\epsilon,\sigma_{min}(A-\lambda^2 I_n)\leq\epsilon\}.
    $$
Then we use the algebraic relation $\lambda_1=-b_n-a_n\lambda_2$. Having chosen $\lambda^2\in N_\epsilon$ closest to $\lambda_2$, we can find at most one $\lambda^1\in N_\epsilon$ such that $|\lambda^1+b_n+a_n\lambda^2|\leq \epsilon(4\sqrt{n}(1+|a_n|))^{-1}$. These inequalities combined imply that $|\lambda^1-\lambda_1|\leq \epsilon(4\sqrt{n})^{-1}$.

From the construction we have $||\lambda^1-\lambda^2|-|\lambda_1-\lambda_2||\leq 2\epsilon(4\sqrt{n})^{-1}$. Therefore as we range over all real numbers $\lambda_1,\lambda_2$ satisfying the algebraic constraint $\lambda_1=-b_n-a_n\lambda_2$, there are at most $\epsilon^{-3/4}(4\sqrt{n}(1+|a_n|))$ pairs of $\lambda^1,\lambda^2$ thus constructed satisfying $|\lambda^1-\lambda^2|\leq\epsilon^{1/4}$.
For these pairs we forget one location and use the one-location estimate 
$$
\mathbb{P}(\sigma_{min}(A-\lambda^1 I_n)\leq\epsilon,\sigma_{min}(A-\lambda^2 I_n)\leq\epsilon)\leq\mathbb{P}( \sigma_{min}(A-\lambda^2 I_n)\leq\epsilon)\leq \epsilon\sqrt{n}+e^{-cn}.
$$
    For all other pairs we use the two-location estimate in Theorem \ref{theorem1.3}:
$$
\mathbb{P}(\sigma_{min}(A-\lambda^1 I_n)\leq\epsilon,\sigma_{min}(A-\lambda^2 I_n)\leq\epsilon)\leq \frac{\epsilon^2n^{1.5}}{\epsilon^{1/4}}+e^{-cn}\leq\epsilon^{7/4} n^{1.5}+e^{-cn}.
$$
    Now we combine everything and get, noting that $|N_\epsilon|\leq 4 \epsilon^{-1}\sqrt{n}(1+a_n),$
      \begin{equation}\begin{aligned}
         &\mathbb{P}(\text{There exists two real eigenvalues $\lambda_1,\lambda_2$ of $A$ such that } \lambda_1+a_n\lambda_2=b_n)
         \\&\leq \epsilon^{-3/4}(4\sqrt{n}(1+|a_n|))\cdot\epsilon\sqrt{n}+\epsilon^{7/4}n^{1.5}|N_\epsilon|+|N_\epsilon|e^{-cn}=e^{-\Omega(n)}
     \end{aligned}\end{equation} once we set $\epsilon=\exp(-c'n)$ for any $c'>0$. We also use the assumption $a_n=\exp(o(n))$. 
\end{proof}

\begin{proof}[\proofname\ of Corollary \ref{corollary2comp}]
We work on the event $\mathcal{K}$ where $\|G\|\leq 8\sqrt{n}$, and this event happens with probability $1-\exp(-\Omega(n))$. We consider a $\epsilon n^{-1/2}$-net for $[-8\sqrt{n},8\sqrt{n}]$ on the real axis and a $\epsilon n^{-1/2}$-net for the interval $[-8\sqrt{n}i,8\sqrt{n}i]$ on the complex axis, and denote by $N_\epsilon$ the product net. Then we can easily check that 
\begin{equation}\label{reductioncomplexsum}\begin{aligned}
&\mathbb{P}^\mathcal{K}(G\text{ has two eigenvalues whose sum is real})\\&\leq\mathbb{P}^\mathcal{K}(\text{There exists }\lambda^1,\lambda^2\in\mathcal{N}_\epsilon, \text{with }|\Im(\lambda^1+\lambda^2)|\leq2\epsilon n^{-1/2}\\&\text{ such that }\sigma_{min}(G-\lambda^1I_n)\leq16\epsilon,\sigma_{min}(G-\lambda^2I_n)\leq16\epsilon).
\end{aligned}\end{equation}
    The net has cardinality at most $256n^2\epsilon^{-2}$, and for any fixed $\lambda^1$, the number of possible $\lambda^2$ with $|\Im(\lambda^1+\lambda^2)|\leq2\epsilon n^{-1/2}$ is at most $32\epsilon^{-1}n$. We consider two possible scenarios, when (i) $|\Im\lambda_1|\geq \epsilon^{0.1}$ then we always have $|\lambda^1-\lambda^2|\geq \epsilon^{0,1}$ so the corresponding probability of the right hand side of \eqref{reductioncomplexsum} is bounded by $2^{10}\epsilon^{-3}n^3\cdot(Cn^3\epsilon^{3.8}+e^{-cn})$ where the first term is the total number of possible pairs of $(\lambda^1,\lambda^2)$ with small $|\Im(\lambda^1-\lambda^2)|\leq 2\epsilon n^{-1/2}$ and the second term is the probability upper bound obtained from applying Theorem \ref{complextwoballbounds}.

    In the second scenario we have $|\Im \lambda^1|\leq \epsilon^{0.1}$. All such possible pairs have cardinality $2\epsilon^{0.1}\sqrt{n}\cdot (32\epsilon^{-1}n)^2$. For such $\lambda^1,\lambda^2$ we forget about the second constraint and only use the first one, namely the bound $\mathbb{P}(\sigma_{min}(G-\lambda^1I_n)\leq\epsilon)$. Then we can use the result of Luh \cite{luh2018complex} to bound this probability by $\epsilon^2n+e^{-cn}$, and the overall contribution is $O(\epsilon^{-1.9}n^{2.5}(\epsilon^2n+e^{-cn}))$. Taking $\epsilon$ to be exponentially small confirms that the right hand side of \eqref{reductioncomplexsum} is also exponentially small in $n$.
\end{proof}

\subsection{Leftover technical proofs}\label{appendix4a}

We first give the proof of Lemma \ref{specified850}.
\begin{proof}[\proofname\ of Lemma \ref{specified850}]We essentially modify the proof of \cite{campos2025singularity}, Lemma 7.8.

For each $\ell\geq 1$ define an interval intersected with integers $$I_\ell:=[-2^\ell N,2^\ell N]\setminus[-2^{\ell-1}N,2^{\ell-1}N]$$ and $I_0:=[-N,N]$. Fix $J:=[-\kappa N,-N]\cup[N,\kappa N]$.

For any sequence of integers $(\ell_1,\cdots,\ell_i,\cdots,\ell_{2n-1})_{i\in[2n-1]\setminus(D_1\cup D_2)}\in\mathbb{Z}_{\geq 0}^{2n-1-2d}$ we consider the following box
$$
B(\ell_1,\cdots,\ell_{2n-1}):=\oplus_{j=1}^{2n-1}(J\mathbf{1}_{j\in D_1\cup D_2}+I_{\ell_j}\mathbf{1}_{j\notin D_1\cup D_2}), 
$$
In words, for each $j\notin D_1\cup D_2$  we have an integer $\ell_j$ of freedom and use the interval $I_{\ell_j}$ in the $j$-th coordinate. Then for each $j\in D_1\cup D_2$ we use the interval $J$ on the $j$-th coordinate. The resulting box $B(\ell_1,\cdots,\ell_{2n-1})$ is understood that it depends on $2n-1-2d$ free variables and that the labels $\ell_j,j\in D_1\cup D_2$ are omitted from the variables determining $B$.
Then we will be using the following family of boxes
$$
\mathcal{F}:=\left\{B(\ell_1,\cdots,\ell_{2n-1}):\sum_{j\in[2n-1]\setminus(D_1\cup D_2):\ell_j>0}2^{2\ell_j}\leq 16n/\kappa_0^2
\right\}.
$$
We first check that $\mathcal{F}$ forms a net for $\Lambda_\epsilon$. Given a $v\in\Lambda_\epsilon$, we normalize to have $X:=vn^{1/2}/(4\epsilon)\in\mathbb{Z}^{2n-1}$. For $i\in[2n-1]\setminus(D_1\cup D_2)$ define $\ell_i$ to be the unique integer satisfying $X_i\in I(\ell_i)$. Then we can check that $X\in B(\ell_1,\cdots,\ell_{2n-1})$: note that by assumption on $|v_i|,i\in D_1\cup D_2$ we have that for all these $i$, we must have $\kappa_0/(4\epsilon)\leq |X_i|\leq\kappa_1/(4\epsilon)$, so that by the choice $N=\kappa_0/(4\epsilon)$ we have $X_i\in J$. We also used $\kappa\geq \kappa_1/\kappa_0$. To further check that $B(\ell_1,\cdots,\ell_{2n-1})\in\mathcal{F}$ it suffices to note (since $\|v\|_2^2=2$)
$$
\sum_{j:\ell_j>0}2^{2(\ell_j-1)}N^2\leq\sum_{j=1}^{2n-1} X_j^2\leq n/(4\epsilon)^2(\sum_{j=1}^{2n-1} v_j^2)\leq 8nN^2/\kappa_0^2,
$$ which verifies the claim.

Then we check that $|\mathcal{F}|\leq\kappa^{2n}$. This is equivalent to counting the number of non-negative integer sequences $(\ell_1,\cdots,\ell_i,\cdots,\ell_{2n-1})_{i\notin D_1\cup D_2}$ with $\sum_{\ell_i>0}4^{\ell_i}\leq 16n/\kappa_0^2$. Given any $t\geq 0$ there should be no more than $16n/(4^t\kappa_0^2)$ values of $i\in[2n-1]\setminus (D_1\cup D_2)$ with $\ell_i=t$, so the total number of choices is at most
$$
\prod_{t\geq 0}\binom{2n}{\leq 16n/(\kappa_0^24^t)}\leq (\kappa_0/8)^{-8n}<\kappa^{2n}.
$$
The size of the box $B(\ell_1,\cdots,\ell_{2n-1})\in\mathcal{F}$ can be bounded via
$$
|\mathcal{B}(\ell_1,\cdots,\ell_{2n-1})|\leq N^{2n-1}\kappa^d2^{n+\sum_j\ell_j}\leq (\kappa N)^{2n-1}
$$ where we use $\prod_j 2^{\ell_j}\leq (\frac{1}{n}\sum_j 2^{2\ell_j})^n\leq (8/\kappa_0^2)^n$.
\end{proof}

We then sketch the proof of Lemma \ref{incompmore}.
\begin{proof}[\proofname\ of Lemma \ref{incompmore}] The assumption of this lemma requires that any such unit vector $r$ must satisfy $\|r\|_{[n]\setminus\{1\}}>M_{\ref{lemmadegenerate}}>0.$
    If we can find constants $\lambda$,$\hat{\lambda}$, vector pairs $(v,w)$ and $(c_1,c_2)$ satisfying the assumptions, then we can find some $\lambda_i\in\mathbb{R},i=3,4,5,6,7$ such that $(A_{/[1]}-\lambda_2{I_n}_{/[1]})r=\lambda_3 v+\lambda_4e_1$ and $(A_{/[1]}-\lambda_5{I_n}_{/[1]})v=\lambda_6 r+\lambda_7e_1$. Combining both, we get $$(A_{/[1]}-\lambda_5{I_n}_{/[1]})(A_{/[1]}-\lambda_2{I_n}_{/[1]})r=\lambda_3\lambda_6 r+\lambda_4(A_{/[1]}-\lambda_5 {I_n}_{/[1]}) e_1+\lambda_3\lambda_7e_1.$$

Then we follow the same procedure as in the proof of Proposition \ref{proposition1199} to show that any such vector $r$ solving this equation must be $(\delta,\rho)$-incompressible (for constants $\rho,\delta$ depending only on $\xi$ but not on the various $\lambda_i$'s) with high probability. Then we take a union bound over all $\lambda_i$ chosen over some $\epsilon$-nets.
\end{proof}

Then we prove Lemma \ref{randmgenerationoflcd}.

\begin{proof}[\proofname\ of Lemma \ref{randmgenerationoflcd}]
We discretize the range of possible $\phi$, but we use different scales to discretize $\phi_1$ and $\phi_2$. Consider the interval $i\in[\tau/\alpha,32T_{\ref{randmgenerationoflcd}}ND^{-1/2}/\alpha]=I,$ and also consider the longer interval $\widetilde{I}=[0,32T_{\ref{randmgenerationoflcd}}ND^{-1/2}/\alpha]$ 
. Define also $I'=[\tau/\alpha,32T_{\ref{randmgenerationoflcd}}ND^{-1/2}/\alpha]$ and $\widetilde{I}'=[0,32T_{\ref{randmgenerationoflcd}}ND^{-1/2}/\alpha]$.

We define, for each integer $i\in \widetilde{I}$, $\phi_1^i=i\alpha/(2T_{\ref{randmgenerationoflcd}} N)$ and for each $j\in\widetilde{I}'$ define $\phi_2^j=j\alpha/(2T_{\ref{randmgenerationoflcd}} N_1)$. Assume that we can find $\phi\in\mathbb{R}^2$ such that the event in the probability of \eqref{assumptionmatrix} holds true. We claim that we can find $\phi_1^i$, $\phi_2^j$ for some $i\in\widetilde{I},j\in\widetilde{I}'$ with at least one $i\in I$ or $
j\in I'$ satisfying
$$
\|\phi_1^iX+\phi_2^j Y\|_\mathbb{T}\leq 5\sqrt{\alpha D}.
$$
To see this, we simply take $\phi_1^i$ for which $|\phi_1^i-\phi_1|\leq \alpha/(T_{\ref{randmgenerationoflcd}} N)$ and take $\phi_2^j$ for which  $|\phi_2^j-\phi_2|\leq \alpha/(T_{\ref{randmgenerationoflcd}} N_1)$, then applying triangle inequality to get, using $\alpha\leq 1$,
$$
\|\phi_1^i X+\phi_2^jY\|_\mathbb{T}\leq \|\phi_1 X+\phi_2Y\|_\mathbb{T}+|\phi_1^i-\phi_1|\sqrt{D}T_{\ref{randmgenerationoflcd}} N+|\phi_2^j-\phi_2|\sqrt{D}T_{\ref{randmgenerationoflcd}} N_1\leq 5\sqrt{\alpha D
}.
$$ Moreover, we have at least one $i\in I$ or $j\in I'$ by assumption of $\phi_1,\phi_2$ in \ref{assumptionmatrix}.

We first consider the case $i\in I$ so that $
|\phi_1^i|\geq \tau/(2T_{\ref{randmgenerationoflcd}}N)$. In this case we will condition on the random vector $Y$ and only use the randomness in $X$. We estimate, for any fixed $Y$,
\begin{equation}\label{whatisthesum?}
    \sum_{i\in I,j\in\widetilde{I}'}\mathbb{P}_X(\|\phi_1^i X+\phi_2^j Y\|_\mathbb{T}\leq 5\sqrt{\alpha D}).
\end{equation}
Observe that, if $\|\phi_1^i X+\phi_2^j Y\|_\mathbb{T}\leq 5\sqrt{\alpha D}$ then we can find a subset $S(i,j)\subset D_1\subset [D]$ with $|S(i,j)|\geq d/2$ such that for any $k\in S(i,j)$, $\|\phi_1^i X_k+\phi_2^j Y_k\|_\mathbb{T}\leq 10\sqrt{\alpha}$ (here we use $d\leq D\leq 2d$). Then we take a union bound over all subsets $S(i,j)\subset D_1$ and apply independence of the coordinates $X_i$ to show
\begin{equation}\label{backintoshow}
 \eqref{whatisthesum?}\leq \sum_{i\in I,j\in\widetilde{I}}\sum_{S(i,j)\in D_1:|S(i,j)|=d/2}\prod_{k\in S(i,j)}\mathbb{P}_{X_k}(\|\phi_1^i X_k+\phi_2^j Y_k\|_\mathbb{T}\leq 10\sqrt{\alpha}).
\end{equation}

Then we only have to study the following probability for any fixed $k\in D_1$, fixed $\phi_1^i,\phi_2^j$ and fixed  $Y_k$:
$\mathbb{P}_{X_k}(\|\phi_1^i X_k+\phi_2^jY_k\|_\mathbb{T}\leq 10\sqrt{\alpha}).$

Indeed, $\|\phi_1^i X_k+\phi_2^j Y_k\|_\mathbb{T}\leq10\sqrt{\alpha}$ implies that  $\|\phi_1^i X_k+\phi_2^j Y_k-p|\leq 10\sqrt{\alpha}$ for some $p\in\mathbb{Z}$. We take $\alpha\leq 10^{-2}$, Then we have $|p-\phi_2^jY_k|\leq |\phi_1^i X_k|+1\leq |\phi_1^i|\kappa N+1:=T_i$ so that we now have
\begin{equation}\label{propbounds2nd}\begin{aligned}&
\mathbb{P}_{X_k}(\|\phi_1^i X_k+\phi_2^j Y_k\|_\mathbb{T}\leq 10\sqrt{\alpha})\\&\leq\sum_{p\in\mathbb{Z}:|p-\phi_2^jY_k|\leq T_i}\mathbb{P}_{X_k}(|X_k-p(\phi_1^i)^{-1}+(\phi_2^j) Y_k(\phi_1^i)^{-1}|\leq 10\sqrt{\alpha}(\phi_1^i)^{-1})\\&\leq \frac{(2T_i+1)(10\sqrt{\alpha}(\phi_1^i)^{-1}+1)}{2(\kappa -1)N}.\end{aligned}
\end{equation}In the last step we use the fact that $X_k$ is uniformly distributed over $[-\kappa N,-N]\cup[N,\kappa N]$.

Then we use the fact that $10\sqrt{\alpha}(\phi_i^i)^{-1}\geq 10\sqrt{\alpha}D^{1/2}/16\gg 1$ for $D$ large and $|\phi_1^i|\kappa N\geq \frac{\tau\kappa}{2T_{\ref{randmgenerationoflcd}}}:=\frac{1}{S}$ to get $2T_i+1\leq (2+S)|\phi_i^i|\kappa N$. Then $\eqref{propbounds2nd}\leq 20(2+\frac{1}{S})\sqrt{\alpha}.$

Finally we take this equation back into \eqref{backintoshow}, then use the upper bound $|\widetilde{I}|\leq 3^D,|\widetilde{I}'|\leq 3^D$ by assumption. The case when $j\in I'$ can be proven in exactly the same way. This completes the whole proof.
\end{proof}

\subsection{On the number of real eigenvalues and Theorem \ref{universalityrealroots}}
This section is devoted to the proof of Theorem \ref{universalityrealroots}. We first introduce some notions on correlation functions.

Let $m_r,m_c\in\mathbb{N}_+$, for a random matrix $X$ with real entries we define its correlation function $\rho_X^{m_r,m_c}$ as the unit function satisfying, for all $f\in C_c(\mathbb{R}^{m_r}\times\mathbb{C}_+^{m_c}),$ $$
\mathbb{E}[\sum f(u_{j_1},\cdots,u_{j_{m_r}},z_{k_1},\cdots,z_{k_{m_c}}]=\int_{\mathbb{R}^{m_r}\times \mathbb{C}_+^{m_c}}f(\mathbf{u},\mathbf{z})\rho_X^{m_r,m_c}(\mathbf{u},\mathbf{z})d\mathbf{u}d\mathbf{z}
$$ with the sum on the left hand side is over $m_r$ distinct real eigenvalues of $A$ and $m_c$ distinct complex eigenvalues of $A$ in $\mathbb{C}_+:=\{z\in\mathbb{C}:\Im z>0\}$. Here $m_r$ is the number of real eigenvalues and $m_c$ is the number of complex conjugate vector pairs (since $X$ is a real matrix, its complex eigenvalues appear as conjugate pairs.) We then define scaled copy of the correlation function, scaled at a certain reference point $u\in\mathbb{R}$ by 
$$
\rho_{X,r}^{m_r,m_c}(\mathbf{u},\mathbf{z},u,\sigma_u)=\rho_X^{m_r,m_c}(u+\frac{\mathbf{u}}{\sqrt{n\sigma_u}},u+\frac{\mathbf{z}}{\sqrt{n\sigma_u}}).
$$ The subscript $r$ in $\rho_{X,r}^{m_r,m_c}$ means we are considering the real bulk with $u\in\mathbb{R}$.

Let $\operatorname{GinOE}$ denote the real Ginibre ensemble, which is realized as $G=(\frac{1}{\sqrt{n}}g_{ij})$ where $g_{ij}$ are i.i.d. real Gaussian variables with mean 0 and variance 1. (
In this section the complex bulk is irrelevant to us, but we havve a similar correlation function $\rho_{X,c}^{m_r,m_c}(\mathbf{z};z,\sigma_z)$ there.)

In the Gaussian case, the scaling limits of this rescaled correlation function is computed in \cite{borodin2009ginibre}, \cite{forrester2007eigenvalue}, \cite{sommers2008general} and is denoted by $\rho_{\operatorname{GinOE},r}^{m_r,m_c}$.

The main result
Theorem 2.2 of \cite{osman2025bulk} implies that, for any $u\in(-1,1)$, we can find a $\delta>0$ such that
\begin{equation}\label{realuniversals}
    \begin{aligned}
&\int_{\mathbb{R}^{m_r}\times\mathbb{C}_+^{m_c}}f(\mathbf{u},\mathbf{z})\rho_{A/\sqrt{n},r}^{m_r,m_c}(\mathbf{u},\mathbf{z};u,1)d\mathbf{u}d\mathbf{z}\\&=\int_{\mathbb{R}^{m_r}\times\mathbb{C}_+^{m_c}}f(\mathbf{u},\mathbf{z})\rho^{m_r,m_c}_{\operatorname{GinOE},r}(\mathbf{u},\mathbf{z})d\mathbf{u}d\mathbf{z}+O(n^{-\delta}),
    \end{aligned}
\end{equation} and by going through the proof in \cite{osman2025bulk} (especially checking \cite{osman2025bulk},Theorem 2.1) we can ensure that for a test function $f$ with bounded support and for $u\in(-1+\epsilon,1-\epsilon)$ we can take $\delta>0$ in \eqref{realuniversals} uniformly over all such $u\in(-1+\epsilon,1-\epsilon)$. Now we prove Proposition \ref{universalityrealroots}.

\begin{proof}[\proofname\ of Theorem \ref{universalityrealroots}]
    Due to the restriction $u\in(-1+\epsilon,1-\epsilon)$ in \eqref{realuniversals} we consider $N_{[-1+\epsilon,1-\epsilon]}(A_n/\sqrt{n})$, which is the number of real eigenvalues of $A_n$ in $((-1+\epsilon)\sqrt{n},(1-\epsilon)\sqrt{n})$. We consider smooth, nonnegative functions $F_{-}(x),F_+(x)$ satisfying 
    $$
1_{[-1+10\epsilon,1-10\epsilon]}(x)\leq F_{-}(x/\sqrt{n})\leq 1_{[-1+\epsilon,1-\epsilon]}(x)\leq F_+(x/\sqrt{n})\leq 1_{[-1+\epsilon/10,1-\epsilon/10]}(x).
    $$ Then $$
    \mathbb{E}N_{[-1+\epsilon,1-\epsilon]}(A_n/\sqrt{n})\geq \int_\mathbb{R}\rho_{A/\sqrt{n},r}^{(1,0)}(x)F_{-}(x/\sqrt{n})dx.$$
    We can take a smooth partition of $F_{-}(x)$ into $O(\sqrt{n})$ pieces of smooth functions supported on intervals of length $O(1)$, and then apply \eqref{realuniversals} to each piece. We get that upon swapping, the difference to the right hand side is at most $O(n^{1/2-\delta})$, so that 
     \begin{equation}\label{newvassar}\begin{aligned}
    \mathbb{E}N_{[-1+\epsilon,1-\epsilon]}(A_n/\sqrt{n})&\geq \int_\mathbb{R}\rho_{\operatorname{GinOE},r}^{(1,0)}(x)F_{-}(x/\sqrt{n})dx-O(n^{1/2-\delta})\\&\geq \mathbb{E}N_{[-1+10\epsilon,1-10\epsilon]}(G_n)-O(n^{1/2-\delta'})
    \end{aligned},\end{equation} where $G_n$ is the real Ginibre matrix. In \cite{edelman1994many}, Corollary 4.3 Edelman, Kostlan and Shub obtained the marginal probability density $f_n(x),x\in\mathbb{R}$ of the real eigenvalues of $G_n$, and in Corollary 4.5 they proved that $f_n(x)$ converges to $\frac{1}{2}$ for $|x|<1,$ to $(2+\sqrt{2})/8$ for $|x|=1$ and to 0 for $|x|>1$. The convergence is both pointwise and in $L^p$ norms for $1\leq p<\infty$. That is, asymptotically the density of real eigenvalues of $G_n$ is uniform on $[-1,1]$. By computations in \cite{edelman1994many}, Chapter 5 they obtained $\mathbb{E}_\mathbb{R}(G_n)=\sqrt{\frac{2\pi}{n}}+O(1)$. Combining all these facts, we get that $\mathbb{E}N_{[-1+10\epsilon,1-10\epsilon]}(G_n)=\sqrt{\frac{2\pi}{n}}(1-10\epsilon)(1+o(1))$. Taking this back into \eqref{newvassar}, we get
    \begin{equation}\label{expectationlower}
\mathbb{E}N_{[-1+\epsilon,1-\epsilon]}(A_n/\sqrt{n})\geq \sqrt{\frac{2n}{\pi}}(1-10\epsilon)(1+o(1))-O(n^{1/2-\delta'}).
    \end{equation}

    Next we get an upper bound for the variance of $N_{[-1+\epsilon,1-\epsilon ]}(A_n/\sqrt{n})$. The situation is more complicated here as \eqref{realuniversals} implies universality only in the bulk, so we cannot directly use the result on the variance of $N_\mathbb{R}(G_n)$. We denote by $$I_1:=[-1+\epsilon+n^{-1/4},1-\epsilon-n^{-1/4}],I_2:=[-1+\epsilon,1-\epsilon],I_3:=[-1+\epsilon-n^{-1/4},1-\epsilon+n^{-1/4}].$$ This time, we consider smooth, nonnegative functions $R_{-}(x),R_+(x)$ satisfying 
    $$\begin{aligned}
&1_{I_1}(x)\leq R_{-}(x/\sqrt{n})\leq 1_{I_2}(x)\leq R_+(x/\sqrt{n})\leq 1_{I_3}(x).\end{aligned}$$
    Moreover, we shall choose $R_{+}(x)$ and $R_{-}(x)$  in such a way that they can be chopped into the sum of $O(\sqrt{n})$ smooth functions with support length $O(1)$ each, so the previous computations still work here.

    We begin with some observations. First, the previous proof already implies, using $|a^2-b^2|=|a-b||a+b|$,
    \begin{equation}\label{steps1}
|(\mathbb{E}N_{[-1+\epsilon,1-\epsilon]}(A_n/\sqrt{n}))^2-(\mathbb{E}N_{[-1+\epsilon,1-\epsilon]}(G_n))^2|=O(n^{1-\delta'}).
    \end{equation}
    Using the fact that $f_n(x)$ is approximately uniform on $[-1,1]$, we can check that
    \begin{equation}\label{steps2}
|(\mathbb{E}N_{[-1+\epsilon-n^{-1/4},1-\epsilon+n^{-1/4}]}(G_n))^2-(\mathbb{E}N_{[-1+\epsilon,1-\epsilon]}(G_n))^2|=O(n^{1/2-\delta'}),
    \end{equation} whenever $\delta'\leq \frac{1}{4}$.
 Then from the fact that
$$\begin{aligned}
\mathbb{E}N_{I_1}(A_n/\sqrt{n})^2&\leq\int_\mathbb{R}\rho_{A_n}^{(1,0)}(x)R_{-}(x/\sqrt{n})^2dx
+\int_\mathbb{R}\int_\mathbb{R}
\rho^{2,0}_{A_n}(x,y)R_{-}(x/\sqrt{n})R_{-}(y/\sqrt{n})dxdy
\\&\leq \mathbb{E}N_{I_2}(A_n/\sqrt{n})^2\\&\leq
\int_\mathbb{R}\rho_{A_n}^{(1,0)}(x)R_{+}(x/\sqrt{n})^2dx
+\int_\mathbb{R}\int_\mathbb{R}
\rho^{2,0}_{A_n}(x,y)R_{+}(x/\sqrt{n})R_{+}(y/\sqrt{n})dxdy
\\&\leq \mathbb{E}N_{I_3}(A_n/\sqrt{n})^2,
\end{aligned}$$
    we can apply \eqref{realuniversals} to do the swapping $O(n)$ times with an overall cost $O(n^{1-\delta'})$ and obtain
\begin{equation}\label{steps3}
\mathbb{E}N_{I_2}(A_n/\sqrt{n})^2\leq \mathbb{E} N_{I_3}(G_n/\sqrt{n})^2+O(n^{1-\delta'}).
\end{equation}
    
Combining the three estimates \eqref{steps1}, \eqref{steps2}, \eqref{steps3} above, we deduce that
$$
\operatorname{Var}N_{I_2}(A_n/\sqrt{n})\leq \operatorname{Var}N_{I_3}(G_n)+O(n^{1-\delta'}).
$$
We will prove, in Lemma \eqref{gaussiancasevariance}, that 
\begin{equation}\label{leftovervariance}
\operatorname{Var}N_{I_3}(G_n)=O(n^\frac{1}{2}).
\end{equation}
This would imply $\operatorname{Var}N_{[-1+\epsilon,1-\epsilon]}(A_n/\sqrt{n})=O(n^{1-\delta'})$.

Using Markov's inequality with the expectation lower bound from \eqref{expectationlower}, we conclude that we can find some $c'>0$ such that, with probability at least $1-n^{-c'}$, we have $N_\mathbb{R}(A_n)\geq N_{[-1+\epsilon,1-\epsilon]}(A_n/\sqrt{n})\geq \sqrt{\frac{2n}{\pi}}(1-10\epsilon)-O(n^{1/2-c'})$. This completes the proof upon changing the value of $\epsilon>0$.
    
\end{proof}

Finally, we prove the left-over computation \eqref{leftovervariance}
\begin{lemma}\label{gaussiancasevariance}
    For any $\epsilon>0$, the estimate \eqref{leftovervariance} holds. \end{lemma}
While the quantity $\operatorname{Var}N_\mathbb{R}(G_n)$ has been computed in \cite{forrester2007eigenvalue}, this quantity does not immediately provide an upper bound for the variance of real roots in a certain interval. Therefore we have to follow the path in \cite{forrester2007eigenvalue} and do the calculation again on this interval.

\begin{proof}
    We will use computations from \cite{forrester2007eigenvalue}, see also \cite{byun2023progress} and \cite{forrester2023review} for a more detailed presentation. The variance $V_n$ of $N_\mathbb{R}(G_n)$ can be computed via the integral of two-point correlation functions (see \cite{forrester2007eigenvalue}, the lines between eq (17) and (18), or \cite{byun2022progress}, equation (3.2)). The precise formula is  $V_n=\int_\mathbb{R}dx\int_\mathbb{R}dy\rho_{(2)}^{rT}(x,y)+E_n$ where $E_n=\mathbb{E}_\mathbb{R}(G_n)$ and $\rho_{(2)}^{rT}(x,y):=\rho_{(2)}^r(x,y)-\rho_{(1)}^r(x)\rho_{(1)}^{r}(y)$. It is not hard to check that we can restrict the integral of $x$ and $y$ to compact regions to get an estimate for $\operatorname{Var}N_{I_3}(G_n)$, in the form 
    $$
\operatorname{Var}N_{I_3}(G_n)=\int_{\sqrt{n}I_3}dx\int_{\sqrt{n}I_3}dy\rho_{(2)}^{rT}(x,y)+\mathbb{E}N_{I_3}(G_n)
    $$ where $\sqrt{n}I_3$ denotes the interval $\{\sqrt{n}x:x\in I_3\}$.
     The two-point correlation kernel $\rho_{(2)}^r(x,y)$ has the expression 
    $$
\rho_{(2)}^r(x_1,x_2)=\operatorname{Pf}\begin{bmatrix}
    \operatorname{sgn}(x_1-x_2)+I^r(x_1,x_2)&S^r(x_1,x_2)\\-S^r(x_1,x_2)&D^r(x_1,x_2)
\end{bmatrix},
    $$
 where the exact expression of functions involved can be found in \cite{forrester2007eigenvalue}, equation (15). 
The large $n$ limit of $\rho_{(2)}^r(x,y)$ converges to, for fixed $x,y$, (see \cite{byun2023progress}, equation (2.36))
\begin{equation}
    \operatorname{Pf}\begin{bmatrix}
        \frac{1}{2\sqrt{2\pi}}(y-x)e^{-(x-y)^2/2}&\frac{1}{\sqrt{2\pi}}e^{-(x-y)^2/2}\\-\frac{1}{\sqrt{2\pi}}e^{-(x-y)^2/2}&\operatorname{sgn}(x-y)\operatorname{erfc}(|x-y|/\sqrt{2})
    \end{bmatrix}
\end{equation} which has Gaussian tails and thus integrable over $\mathbb{R}$. From this expression we can readily compute that $\operatorname{Var}N_{I_3}(G_n)=O(n^\frac{1}{2})$.

\end{proof}

\section*{Funding}
The author is supported by a Simons Foundation Grant (601948, DJ).

\section*{Acknowledgment}
The author thanks Prof. Julian Sahasrabudhe for a brief discussion on this topic. The author also thanks Prof. Alexei Borodin for some interesting remarks and Prof. Giorgio Cipolloni for pointing out some references.

\printbibliography

\end{document}